\documentclass[11pt, a4paper, leqno]{amsart}
\usepackage[utf8]{inputenc}
 
\usepackage[euler-digits]{eulervm}

\usepackage{amsfonts, amsthm, amssymb, amsmath}
\usepackage{graphicx}
\usepackage{xcolor,colortbl}
\usepackage{mathtools}
\usepackage{mathrsfs,array}
\usepackage{eucal,fullpage,times,color,enumerate,accents}
\usepackage{url}
\usepackage{comment}
\usepackage{multicol}
\usepackage{hyperref}
\hypersetup{
  colorlinks   = true,          %Colors links instead of ugly boxes
  urlcolor     = blue,          %Color for external hyperlinks
  linkcolor    = teal,          %Color of internal links
  citecolor   = orange             %Color of citations
}

\usepackage{color}
\usepackage{tikz-cd}
\usepackage{bm}
\usetikzlibrary{calc}
\usepackage{subfigure}

\newtheorem{theorem}{Theorem}[section]
\newtheorem{example}[theorem]{Example}
\newtheorem{observation}[theorem]{Observation}

\newtheorem{quasi-theorem}[theorem]{Quasi-Theorem}

\newtheorem{rem1}[theorem]{Remark}
\newenvironment{remark}{\begin{rem1}\em}{\end{rem1}}

\newtheorem{not1}[theorem]{Notation}
\newenvironment{notation}{\begin{not1}\em}{\end{not1}}

\DeclarePairedDelimiter{\ceil}{\lceil}{\rceil}
\DeclarePairedDelimiter{\floor}{\lfloor}{\rfloor}

\theoremstyle{plain}

\newtheorem{conjecture}[theorem]{Conjecture}
\newtheorem{corollary}[theorem]{Corollary}
\newtheorem{question}[theorem]{Question}

\theoremstyle{definition}
\newtheorem{definition}[theorem]{Definition}

\theoremstyle{plain}
\newtheorem{lemma}[theorem]{Lemma}
\newtheorem{proposition}[theorem]{Proposition}
\numberwithin{equation}{section}
\newtheorem*{proposition*}{Proposition}
\newtheorem*{definition*}{Definition}
\newtheorem*{remark*}{Remark}
\newtheorem*{observation*}{Observation}
\theoremstyle{plain}
\newtheorem*{lemma*}{Lemma}
\newtheorem*{theorem*}{Theorem}
\newtheorem*{question*}{Question}
\newtheorem*{corollary*}{Corollary}
\newtheorem*{claim*}{Claim}

\usepackage[draft=true]{minted}
\setminted{breaklines=true}
\definecolor{lbcolor}{rgb}{0.9,0.9,0.9}
\setminted{bgcolor=lbcolor}
\setminted{fontsize=\footnotesize}

\usepackage{tabularx,float,capt-of}

\newcommand{\tmfloatcontents}{}
\newlength{\tmfloatwidth}
\newcommand{\tmfloat}[5]{
	\renewcommand{\tmfloatcontents}{#4}
	\setlength{\tmfloatwidth}{\widthof{\tmfloatcontents}+1in}
	\ifthenelse{\equal{#2}{small}}
	{\setlength{\tmfloatwidth}{0.45\linewidth}}
	{\setlength{\tmfloatwidth}{\linewidth}}
	\begin{minipage}[#1]{\tmfloatwidth}
		\begin{center}
			\tmfloatcontents
			\captionof{#3}{#5}
		\end{center}
\end{minipage}}

\title{Vanishing of weight one syzygies of projective varieties}
\author{Debjit Basu}
\address{Department of Mathematics, University of Kansas,
Snow Hall, 1460 Jayhawk Blvd, Lawrence, Kansas 66045, USA}
\email{debjitbasu@ku.edu}

\begin{document}
\vspace{-0.8cm}
\begin{abstract}
In this article we study conditions under which weight one Koszul cohomology vanishes on projective varieties. As corollary of more general results, we obtain statements on the so-called property $(M_q)$ reflecting on the higher syzygies of minimal surfaces and higher dimensional projective varieties. We also provide lower bounds on $\ell$ such that  $L = \ell B$ and adjoint $L' = K_{X}+\ell B$ satisfies  $(M_{q})$ for base point free and ample line bundle $B$ on a projective variety X. Property $(M_{q})$ is complementary to property$-(N_{p})$ though it is far less studied. Notably, our bounds for property $(M_{q})$ for both $L$ and $L'$ are equal to or better than existing bounds for property $(N_{p})$, as established by other authors—particularly in relation to Mukai’s conjecture and its higher-dimensional analogues. By considering both properties $(M_{q})$ and $(N_{p})$, we gain a significantly deeper understanding of the minimal free resolution. 
\end{abstract}
\maketitle
\vspace{-0.8cm}
\small{\tableofcontents}
\section*{Introduction}
In this article we study the vanishing of weight one syzygies of smooth projective varieties embedded by a complete linear series. \\ \\
\textbf{Background.}
Let $X$ be a smooth complex projective (irreducible) variety of dimension $n = \mathrm{dim}(X),$ $L$ be an ample line bundle on $X,$ $S = \mathrm{Sym}(H^{0}(L))$ be the symmetric algebra on $H^{0}(X, L)$ and $R(L) = \bigoplus_{q \geq 0}H^{0}(X, L)$ be the ring of sections of $L.$ In this setting, one has the minimal graded free resolution
\begin{small}
\begin{equation}
    \label{free_resolution_of_R_L}
        0 \rightarrow F_{r-1} \overset{\phi^{r-1}}{\overset{}{{\longrightarrow}}} \cdots \overset{\phi^{i+1}}{\overset{}{\longrightarrow}} F_{i} = \bigoplus_{j \geq 0} K_{i, j}(X, L) \otimes_{\mathbb{C}} S(-i-j) \overset{\phi^{i}}{\overset{}{{\longrightarrow}}} \cdots \overset{\phi^{1}}{\overset{}{{\rightarrow}}} \ F_{0} \  \overset{\phi^{0}}{\overset{}{{\rightarrow}}} \ R(L) \rightarrow 0
\end{equation}
\end{small}
of $R(L)$ as a $S-$module with Betti numbers $\beta_{i,i+j} = \mathrm{dim}_{\mathbb{C}}K_{i, j}(X, L),$ where $K_{i, j}(X, L)$ is the Koszul cohomology group of weight $j$ at the $i^{th}-$stage $F_{i}$. We say that each $K_{i, j}(X, L) \otimes_{\mathbb{C}} S(-i-j)$ is the weight$-j$ syzygy of $L$ at the $i^{th}$ stage. Let $p, q$ be positive integers. When, $L$ is projectively normal, and the entries of the matrices associated to $\phi^{1}, \cdots, \phi^{p}$ have degree $1$, one says that $L$ satisfies property-$(N_{p}).$ In this case, there are only weight one ($j = 1$) syzygies in stages $F_{1}, \cdots, F_{p}.$ 
The left side of the resolution involving maps $\phi^{r-1}, \cdots, \phi^{r-q}$, reflects on the so called property$-(M_{q})$. This captures the vanishing of weight one syzygies in $F_{r-1}, \cdots, F_{r-q}$ and hence is complementary to the property-$(N_{p}).$ We now give a precise definition of property$-(M_{q})$ which was given in the case $n = 1$ (when $X$ is a curve) by \textit{M. Green} and \textit{R. Lazarsfeld} in \cite[Section-3]{GL86} in the context of the \textit{Gonality conjecture}. Here, we define it in all dimensions to initiate a broader study of the concept.
\begin{definition} $\left[ \ (M_{q})-\textbf{properties} \ \right]$ \\
\label{Intro_def_$(M_q)$}
     The line bundle $L$ is said to satisfy \textbf{Property}$-(\mathbf{m_{q}})$ if its resolution (\ref{free_resolution_of_R_L}) does not have any weight-one syzygies starting from the end (stage$ (r-1)$) of the resolution till stage$ (r-q),$ meaning
    \begin{equation}
        \label{$(M_q)-$property}
        F_{i} = \bigoplus_{j \geq 2}S(-i-j)^{\oplus \beta_{i, i+j}} \text{ or equivalently, } K_{i, 1}(X, L) = 0 \text{ for all } \ r-1 \geq i \geq r-q
    \end{equation}
    If the line bundle $L$ on $X$ is projectively normal and satisfies property$-(m_{q})$ then we say that it satisfies \textbf{Property}$-(\mathbf{M_{q}}).$
\end{definition}
Note that, since Property$ (m_{q})$ is a property concerned with the end (or tail) of the resolution one needs to have the condition $L$ is ample, so that the resolution is finite. \\
Note also that to satisfy property$-(m_{q})$ the line bundle $L$ need not be projectively normal or even very ample and in particular property $(m_{q})$ need not give any information about the syzygies of the homogeneous coordinate ring $S_{X/\mathbb{P}^{r}}$ of the embedding $\phi_{L} :  X \hookrightarrow \mathbb{P}^{r}$ defined by the complete linear series $|L|$ when $L$ is very ample but not projectively normal. Therefore, one needs to add projective normality to the definition to conclude any information of the resolution of $S_{X/\mathbb{P}^{r}}$ from the vanishing of weight one Koszul cohomologies. \par
The Property-$(M_{q})$ enjoys a very distinct rich geometry. Here we mention two of these concepts known in literature.
First, from the $\mathit{K_{p, 1}}-$\textit{Theorem} of \textit{M. Green} \cite[Theorem (3.c.1)]{Gre84I}, it follows that an ample and base point free line bundle $L$ satisfies Property$-(m_{q})$ for all $1 \leq q \leq n-1$ and if $L$ is very ample, the first non-trivial $(m_{q})$-property is property-$(m_{n})$ and $L$ satisfies property-$(m_{n})$ unless the image of $X$ under $\phi_{L}$ is a variety of minimal degree. Moreover a very ample line bundle $L$ of sufficiently high self-intersection $(L^{n})$ satisfies property$-(m_{n+1})$ unless the image of $X$ under $\phi_{L}$ lies in a $(n+1)-$fold of minimal degree in $\mathbb{P}^{r}$. \par
Second, in the case of curves the \textit{Gonality conjecture} by \textit{M. Green} and \textit{R. Lazarsfeld} asserts that for a line bundle of sufficiently high degree on a curve the highest value $q_{\max}$ of $q$ for which $L$ satisfies Property-$(M_{q})$ is $q_{\max} = \mathrm{gon}(C)-1,$ where $\mathrm{gon}(C)$ is the gonality of the curve $C$ (see \cite[(3.7) Conjecture]{GL86}). \par
The above makes it clear that knowing properties$-(M_q)$ and $-(N_p)$ of the line bundle in question, gives a more comprehensive understanding of the resolution. A lot of attention has been given in a series of important papers studying property$-(N_{p})$ for curves and some higher-dimensional varieties. We refer the reader to the articles \cite{BG99} to \cite{GL86} and \cite{EPark05} for an extensive treatment on property-$(N_{p})$. \par
But Property-$(M_{q})$ has not received enough attention yet except for the case of curves.  The gonality conjecture has received wide attention by several authors and it has been proved by \textit{L. Ein} and \textit{R. Lazarsfeld} and an effective version has been proved recently by \textit{W. Niu} and \textit{J. Park}.  Also, \textit{M. Aprodu} and \textit{L. Lombardi} have found a lower bound for $\ell$ depending on $q$ so that given an ample line bundle $A$ on an abelian variety $X,$ the multiple line bundle $L = \ell A$ satisfies Property-$(M_{q}).$ But apart from this, not much is known regarding property-$(M_{q})$ for higher-dimensional varieties. \\ \\
\textbf{Main results.}
In this article we prove a general result on the vanishing of weight one Koszul cohomology for all smooth projective varieties. But for a large class of interesting varieties we prove much finer results. As a consequence of this more general study we prove new results on property-$(M_{q})$ for all minimal algebraic surfaces and higher-dimensional varieties such as varieties with nef canonical bundles, ruled varieties and Fano varieties. We also link the property $(M_{q})$ which is a statement about free resolutions, to the geometric property of $k-$very ampleness for del-Pezzo surfaces. \par
Before we delve on these results, we state a general theorem we prove in this article about the property$-(M_{q})$ for $\ell B$ where $B$ is an ample and base point free line bundle.
\begin{theorem}
\label{CM_theorem}
(Multiple line bundles on smooth projective varieties) \\
    Let $X$ be a smooth projective variety and let $B$ be an ample and base point free line bundle on $X.$ Then, the line bundle $L = \ell B$ satisfies Property-$(m_{q})$ for all $\ell \geq \ell^{\mathrm{ceil}}_{q},$ where
    $
        \ell^{\mathrm{ceil}}_{q} := \ceil[\bigg]{\dfrac{q+1}{n-1}}.
    $
    Moreover, $L = \ell B$ satisfies Property-$(M_{q})$ if $\ell \geq \max(\rho, \ell^{\mathrm{ceil}}_{q})$ where $\rho$ is the regularity of $B$ in the sense of \textit{Castelnuovo} and \textit{Mumford}.
\end{theorem}
In particular, line bundles with regularity $\rho \leq 2$ satisfy property-$(M_{q})$ for $\ell \geq \ell^{\mathrm{ceil}}_{q}$. This bound follows from Theorem \ref{general_vanishing_multiple_line_bundles} as a corollary of a better bound $\ell^{\mathrm{reg}}_{q}$ which depends on $B.$ The examples in Section \ref{Open_problems_examples} show that the bound in the above theorem is sharp. \par
Since $\ell^{\mathrm{ceil}}_{q} \geq 2,$ Theorem \ref{CM_theorem} raises the questions \textit{what happens for} $\ell = 1 ?$ and \textit{what is the largest} $q$ \textit{for which a line bundle} $L$ \textit{satisfies Property}$-(M_{q}) ?$ We answer both of these questions for rational surfaces and Fano varieties of dimension $n \geq 3$ and index $\lambda \geq n-1$ by generalizing the effective bounds by \textit{W. Niu} and \textit{J. Park} \cite{NP24} of \textit{Gonality conjecture} to these higher dimensional varieties. In Theorem \ref{$(M_q)$_rational_surfaces} prove the following Theorem on rational surfaces
\begin{theorem}
Let $X$ be a rational surface and let $L$ be an ample and base point free line bundle on $X$. Let $|L|$ be a linear system containing curves $C$ of genus $g \geq 1$ and $q$ be an integer satisfying $q \geq 2 = \mathrm{dim}(X)$. Then,  if $(-K_{X} \cdot L) \geq q+2,$ then $L$ satisfies property-$(M_{q})$ if and only if $q \leq \mathrm{gon}_{\max}(L).$
\end{theorem}
where $\mathrm{gon}_{\max}(L)$ is the maximal gonality among all smooth urves in $|B|.$ This theorem tells us that in the conditions of the theorem, the highest $q$ for which the line bundle $L = B$ satisfies $(M_{q})$ is $q = \mathrm{gon}_{\max}(B).$ This connects the algebraic property $(M_{q})$ of the resolution with the geometric quantity $\mathrm{gon}_{\max}(B).$ We also analyze the cases $(-K_{X} \cdot L) \geq q+1,$ and $(-K_{X} \cdot L) \geq q.$ In Theorem \ref{Fano}, we extend this result to Fano varieties of dimension $n \geq 3$ and of index $\lambda \geq n-1.$ \par
For rational surfaces, we also make the following conjecture on the length of the strand of mixed weight syzygies which is the strand of the resolution which has both weight one and higher weight syzygies
\begin{conjecture}
\label{conjecture}
    Let $X$ be a rational surface and let $L$ be an ample and base point free line bundle on $X$. Let $|L|$ be a linear system containing curves $C$ of genus $g \geq 1$. Then,  if $(-K_{X} \cdot L) \geq \mathrm{gon}_{\max}(L)+2,$ then the length $\delta(L)$ of the region of syzygies of mixed weights is given by
    \begin{equation}
        \delta(L) = h^{0}(K_{X}+L)-\mathrm{gon}_{\max}(L)+1
    \end{equation}
\end{conjecture}
In Section \ref{Open_problems_examples} we justify this conjecture through examples.
\par
Even though the bound of Theorem \ref{CM_theorem} is optimal, in varieties of particular Kodaira dimension,  with some interesting additional positivity condition on $(X, B)$ one is able to find lower bounds for Property$-(M_{q})$ for multiple line bundles $L = \ell B$. We explore these results in Sections \ref{M2}, \ref{K3}, \ref{Enriques} and \ref{Abelian_Bielliptic} by treating surfaces of each Kodaira dimension separately. \par
We now turn our attention to the syzygies adjoint linear series of higher dimensional varieties. \par
A famous open problem in this area is the Mukai's conjecture on syzygies of surfaces which asserts the $(N_{p})-$property of adjoint linear series $|K_{X}+\ell A|$ on any surface $X$ the multiple $\ell$ higher than a lower bound which generalizes the well-known Green's Theorem on curves to surfaces.
\begin{conjecture} (Mukai's Conjecture)
    Let $A$ be an ample line bundle on a surface $X.$ Then, for any integer $p \geq 0,$ the line bundle $L = K_{X}+\ell A$ satisfies Property-$(N_{p})$ for all $\ell \geq p+4.$
\end{conjecture}
Mukai conjecture is not known in its full generality even for $p = 0.$ However, with \textit{Fujita's conjecture} (see \cite[page 42]{Fuj88} and \cite{GhLa24}) as inspiration, one might ask the analogous question for higher dimensional varieties: namely for any ample line bundle $A$ on a smooth projective variety $X,$ whether the adjoint linear series $|K_{X}+\ell A|$ satisfies Property-$(N_{p})$ for $\ell \geq n+p+2?$ Since \textit{Fujita's very ampleness conjecture} is not known even for $3-$folds, the best we can hope is to study property-$(N_{p})$ of adjoint linear series assuming $A$ is ample and base point free. With the additional hypothesis $A$ very ample this was proved by \textit{L. Ein} and \textit{R. Lazarsfeld} in \cite{EL93}. The original question which was open for thirty years, was settled recently by \textit{P. Bangere} and \textit{J. Lacini} in \cite{BL25}. Explicitly, they prove
\begin{theorem} (P. Bangere and J. Lacini)
    Let $X$ be a smooth projective variety and let $B$ be an ample and base point free line bundle on $X.$ Then, for any integer $p \geq 0,$ the line bundle $L = K_{X}+\ell B$ satisfies Property-$(N_{p})$ if $\ell \geq n+p+1.$
\end{theorem}
In last three decades several interesting results have been proved along the lines of Mukai's conjecture on algebraic surfaces and higher dimensional varieties. The reader should refer to \cite{BG99} to \cite{EL93} for a detailed treatment on \textit{Mukai's conjecture}. These results give a broad understanding of the resolution of adjoint linear series from the beginning of the resolution. \par
The analogous case for property-$(M_{q})$ for adjoint line bundles is unknown until now. In the following Theorems we shed light on the resolution starting from the end, namely on Property-$(M_{q})$ of adjoint linear series. We start with adjoint linear series on higher dimensional varieties with nef canonical bundle.
\begin{theorem} (Adjoint line bundles varieties with nef canonical bundle) \\
\label{adjoint_theorem}
    Let $X$ be a smooth projective variety with nef canonical bundle and $d \geq 1$ be a rational number. Let $B$ be an ample and base point free line bundle on $X$ such that $dB-K_{X}$ is nef and $K_{X}+(n-1)B$ is base point free. Then the line bundle $L = K_{X}+\ell B$ satisfies Property-$(M_{q})$ provided $\ell \geq \max(d+n, \ell^{\mathrm{floor}}_{q})$ when $n \geq 3$ and $\ell \geq \max(d+2, q+2)$ when $n = 2,$ where $\ell^{\mathrm{floor}}_{q} = \floor[\bigg]{\dfrac{q+1}{n-1}}.$
\end{theorem}

In a joint work under progress \cite{BBL25} with \textit{P. Bangere} and \textit{J. Lacini} , we prove a general result for vanishing of weight one syzygies of adjoint line bundles on all smooth projective varieties which is complementary to the above result of the other two authors about property-$(N_{p})$ as in \cite{BL25}. This result will be analogous to Theorem \ref{CM_theorem} for adjoint bundles but involves a complete different set of techniques than we see in this paper. But do note that the bounds in Theorem \ref{adjoint_theorem} are better than the bound we obtain for all smooth projective varieties  \par
For ruled varieties in corollaries \ref{Butler2A} and \ref{Butler2B}, we prove results on Property-$(M_{q})$ for multiple and adjoint line bundles by dropping the condition base point free and we get bounds parallel to the work of Butler \cite[Theorem 2A, 2B]{Bu94} on Property-$(N_{p}).$
\begin{theorem}
    For ample line bundles $A_{1}, \cdots A_{t}$ on a ruled variety $X = \mathbb{P}(E) \xrightarrow{\pi} C$ over a curve $C$ of genus $g$ with invariant $e = -\mathrm{deg}(\bigwedge^{n} E).$ Then for any $q \leq n+\begin{pmatrix}
        n+a-1 \\
        a
    \end{pmatrix}-2$,
    \begin{itemize}
        \item[\textbf{(a)}] The line bundle $L = A_{1}+\cdots+A_{t}$ satisfies Property-$(M_{q})$ if $t \geq 2q+2.$
        \item[\textbf{(b)}] The adjoint line bundle $L' = K_{X}+A_{1}+\cdots+A_{t}$ satisfies Property-$(M_{q})$ if $t \geq 2q+1+\max(1, e+1-g).$
    \end{itemize}
\end{theorem}
Moreover, in Theorem \ref{M_q_ruled_variety_gen_line_bundle} we show that there are convex subsets of the Picard lattice of line bundles on ruled varieties satisfying property-$(M_{q}).$ \par
Now finally we state the implications of our results on $k-$very ampleness of adjoint linear series.
\begin{definition} ($k^{th}-$order embeddings) \\
    Let any integer $k \geq 0.$ a line bundle $L$ on a smooth projective variety $X$ is said to \textbf{separate} a $0-$dimensional subscheme $Z$ if the restriction map the restriction map
    \begin{equation}
        H^{0}(L) \xrightarrow{\rho_{Z}} H^{0}(L_{Z})
    \end{equation}
    is surjective. $L$ called $k-$\textbf{very ample} (resp. $k-$\textbf{very ample}) if it separates any $0-$dimensional subscheme (resp. curvilinear subscheme) $Z$ of length $\mathrm{deg}(Z) = k+1,$ in $X.$ $L$ is said to be \textbf{birationally} $q-$\textbf{very ample} (resp. birationally $q-$spanned) if there exists a Zariski-open dense subset $U \subseteq X$ such that, $L$ separates any $0-$dimensional subscheme (resp. curvilinear subscheme) $Z$ with support in $U$ and of length $\mathrm{deg}(Z) = k+1.$
\end{definition}
In the following theorem, we relate the concept of $k-$very ampleness of adjoint linear series with Property-$(M_{q}).$ We prove this in Corollary \ref{K} as a corollary of Theorem \ref{$(M_q)$_rational_surfaces} and \cite[Theorem 1.1]{Knu03}
\begin{theorem} (Weight-one syzygies and $k-$very ampleness)
    Let $L$ be an ample and base point free line bundle on a del Pezzo surface $X$ with $(-K_{X} \cdot L) \geq k+4.$ Then, the following conditions are equivalent
    \begin{itemize}
        \item[\textbf{(a)}] $L$ satisfies Property-$(M_{2+k}).$
        \item[\textbf{(b)}] $K_{X}+L$ is birationally $k-$very ample.
        \item[\textbf{(c)}] $K_{X}+L$ is birationally $k-$spanned.
        \item[\textbf{(d)}] Any smooth curve $C$ in $|L|$ has gonality $\mathrm{gon}(C) \geq k+2.$
    \end{itemize}
\end{theorem}
\textbf{Organization of the paper.}
This paper is organized as follows. We begin the first section with collecting relevant facts that we use throughout the paper. In the second section, we prove one of our main results Theorem \ref{CM_theorem}. Starting the third and fourth sections deal with property-$(M_{q})$ for line bundles on varieties with nef canonical bundle and varieties of Kodaira dimension $-\infty$ respectively. By the birational classification of varieties, the results of these two sections gives us knowledge of $(M_{q})-$property on all minimal surfaces. We prove Theorem \ref{adjoint_theorem} at the end of the third section. In the fourth section we initiate a detailed study of all line bundles (not just multiple or adjoint bundles) on two very important classes of varieties with Kodaira dimension $-\infty,$ namely rational surfaces and ruled varieties. This gives us information of the geography of property-$(M_{q})$ in the Picard group as convex sets. Finally, we end our study in the fifth section by mentioning some examples and conjectures which are crucial in understanding property-$(M_{q})$ and more broadly the entire resolution of a smooth projective variety embedded by a complete linear series. In the Appendix we prove some results on normal generation of line bundles on surfaces.
\section*{Acknowledgements}
I would first like to thank my thesis advisor Purnaprajna Bangere for suggesting these circle of problems, for his guidance throughout and encouragement. I am also indebted to Justin Lacini for many valuable discussions on syzygies of projective varieties and to Euisung Park, A. L. Knutsen and Raneeta Dutta for helpful discussions. Without these help this article would not be possible. 
\section*{Notation and Terminology}
We now introduce some relevant notation and terminologies.
\flushleft{(1)} \textbf{Tensor and wedge product} \ For any two sheaves $\mathcal{F}$ and $\mathcal{G}$ on a smooth projective variety, we denote the tensor product $\mathcal{F} \otimes_{\mathcal{O}_{X}} \mathcal{G}$ by $\mathcal{F} \otimes \mathcal{G}$ or $\mathcal{F}(\mathcal{G})$ or $\mathcal{G}(\mathcal{F}).$ Often we call tensor product of sheaves as product of sheaves. We preserve the notation $\bigwedge$ to denote wedge product. \\
\flushleft{(2)} \textbf{Divisors and Line Bundles} \ We do not make any distinction between a divisor and its associated line bundle. We write tensor power of line bundles $L^{\otimes q}$ additively as $qL.$ We preserve the notation equality $=$ of two line bundles $L_{1}$ and $L_{2}$ to mean that they are isomorphic or that their associated divisors are linearly equivalent. We use the notation $\equiv$ to denote numerical equivalence of divisors or line bundles.
\flushleft{(3)} \textbf{Cohomology} \ For any coherent sheaf $\mathcal{F}$ on a smooth projective variety $X,$ we denote the $i^{th}-$cohomology group $H^{i}(X, \mathcal{F})$ by $H^{i}(\mathcal{F}),$ for integers $i \geq 0,$ when the base scheme $X$ for the sheaf $\mathcal{F}$ is understood.
\flushleft{(4)} \textbf{Product of Kernel Bundles} \ Let $q$ be a non-negative integer. For globally generated coherent sheaves $\mathcal{F}_{1}, \cdots, \mathcal{F}_{q},$ and a reduced and irreducible member $Q$ in a complete linear series $|P|$ of an effective line bundle $P$ on a smooth projective variety $X,$ we denote
\begin{itemize}
    \item[(a)] the collection of sheaves $\left\{ \mathcal{F}_{1}, \cdots, \mathcal{F}_{q}; \mathcal{O}_{Q} \right\}$ by $\mathfrak{F}|Q \ ;$
    \item[(b)] the index set $\left\{1, 2, \cdots, q \right\}$ by $(\underline{q}),$ the index subsets $\left\{1, 2, \cdots, q-k \right\} \subseteq (\underline{q})$ by $(\underline{q-k}),$ for any $0 \leq k \leq q,$ and the empty set by $(\underline{0}),$ any index subset $\left\{i_{1}, i_{2}, \cdots, i_{q'} \right\} \subseteq (\underline{q})$ by $(\underline{i}),$ for any $q' \leq q \ ;$
    \item[(c)] integer vectors $(m_{1}, \cdots, m_{q}) \in{\mathbb{Z}^{q}},$ by $\vec{m},$ and in $\mathbb{Z}^{q}$ we denote $(1, 0, \cdots, 0)$ by $\vec{e}_{1},$ and $(0, 1, 0 \cdots, 0)$ by $\vec{e}_{2},$ $(0, 0, 1, 0, \cdots, 0)$ by $\vec{e}_{3}$ etc, and the vector $(1, 1, \cdots, 1) = \vec{e}_{1}+\vec{e}_{2}+\vec{e}_{3}+\cdots+\vec{e}_{q}$ by $\vec{\delta} \ ;$ \\
    we denote,
    \item[(d)] the product of kernel bundles (see (\ref{kernel_bundle})) $\left( \bigotimes^{k}_{l = 1} M_{\mathcal{O}_{Q}(\mathcal{F}_{i_{l}})} \right) \otimes \left( \bigotimes^{q'}_{l = k+1}M_{\mathcal{F}_{i_{l}}} \right) $ by $M_{\mathfrak{F}|{Q}}[\underline{i}, k],$ for $0 \leq k \leq q',$ so that  $M_{\mathfrak{F}|Q}[\underline{i}, q] = M_{\mathcal{O}_{Q}(\mathcal{F}_{1})} \otimes \cdots \otimes M_{\mathcal{O}_{Q}(\mathcal{F}_{q})}$ and $M_{\mathfrak{F}|Q}[\underline{i}, 0] = M_{\mathcal{F}_{1}} \otimes \cdots \otimes M_{\mathcal{F}_{q}};$
    \item[(e)] $M_{\mathfrak{F}|Q}[\underline{q}, k]$ by $M_{\mathfrak{F}|Q}[q, k],$ $M_{\mathfrak{F}|Q}[\underline{i}, 0]$ by $M_{\mathfrak{F}|Q}[\underline{i}]$ and $M_{\mathfrak{F}|Q}[q, 0]$ by $M_{\mathfrak{F}|Q}$ for any non-negative integer $q,$ any  $0 \leq k \leq q,$  and any index subset $(\underline{i}) \subseteq (\underline{q})$ ;
    \item[(f)] $M_{\mathfrak{F}|Q}$ by $M$ when $\mathfrak{F}|Q$ is already mentioned or clear from the context;
    \item[(g)] the collection $\left\{m_{1}B, \cdots, m_{q}B; \mathcal{O}_{C} \right\}$ by $\vec{m}B$ for any base point free line bundle $B$ on $X,$ with a smooth member $C$ in the linear series $|B|$ a non- negative integer vector $\vec{m} = (m_{1}, \cdots, m_{q})$ (i.e. $m_{i} \geq 0$ for all $i$). Existence of $C$ is guaranteed by Bertini's Theorem. \\
    Therefore, in this case one has $M_{\vec{m}B} = \bigotimes^{q}_{i = 1}M_{m_{i}B}, \ M_{\vec{m}B}[q-k] = \bigotimes^{q-k}_{i = 1}M_{m_{i}B}$ and $M_{\vec{m}B}[q, k] = \left( \bigotimes^{k}_{l = 1} M_{\mathcal{O}_{C}(m_{l}B)} \right) \otimes \left( \bigotimes^{q}_{l = k+1}M_{m_{l}B} \right),$ for $0 \leq k \leq q,$ and finally by convention $M_{\vec{0}B} = \mathcal{O}_{X} = M_{0 B}.$ 
\end{itemize}
\section{Preliminaries}
We now collect some definitions and previously known results that are crucially used throughout the paper.
\subsection{Koszul cohomology and the Syzygy Bundle}
For any line bundle $L$ on a smooth projective variety $X$ and any coherent sheaf $\mathcal{F}$ on $X,$ we use the notation $V = H^{0}(L)$ and $S = \mathrm{Sym}(V).$ Then, one has the following graded $S-$module of sections
\begin{equation}
\label{moduleofsections}
    R(\mathcal{F}, L) := \bigoplus_{j \in {\mathbb{Z}}}H^{0}(\mathcal{F}(jL))
\end{equation}
with $q^{th}-$graded component $R_{j}(\mathcal{F}, L) = H^{0}(\mathcal{F}(jL))$ for any integer $j.$ Then, for any finitely generated graded $S-$module $R = \bigoplus_{j \in \mathbb{Z}} R_{j}$ one gets the following minimal graded free resolution of $R$:
\begin{small}
\begin{equation}
\label{free_resolution_of_R}
    0 \rightarrow F_{\mathrm{pd}(R)} \rightarrow \cdots \rightarrow F_{i} \rightarrow \cdots \rightarrow F_{0} \rightarrow R \rightarrow0
\end{equation}
\end{small}
where $F_{i} = \bigoplus_{j \in \mathbb{Z}} S(-i-j)^{\oplus \beta_{i, i+j}},$ $\mathrm{pd}(R)$ is the projective dimension of $R$ and $\beta_{i, j}$ are called the \textbf{Betti numbers} of $R.$
\begin{definition} $\left[ \ \textbf{Koszul cohomology} \ \right]$ \\
\label{definition_koszul_cohomology}
    In the above situation, one defines the \textbf{Koszul cohomology} $K_{i, j}(R, V)$ as the cohomology $K_{i, j}(\mathcal{F}, L) = \mathrm{Ker}(\partial_{i, j})/\mathrm{Im}(\partial_{i+1, j-1})$ at the $(i, j)^{th}-$stage of the following (Koszul) complex
    \begin{small}
    \begin{equation}
    \label{Koszul_complex}
        \cdots \xrightarrow[]{} \bigwedge^{i+1}V \otimes R_{j-1} \xrightarrow[]{\partial_{i+1, j-1}} \bigwedge^{i}V \otimes R_{j} \xrightarrow[]{\partial_{i, j}} \bigwedge^{i-1}V \otimes R_{j+1} \xrightarrow[]{} \cdots
    \end{equation}
    \end{small}
    where the maps $\partial_{i, j}$ are defined as obvious alternate sums and one has $\beta_{i, i+j} = \mathrm{dim}_{\mathbb{C}}(K_{i, j}(R, V))$ and $F_{i} = \bigoplus_{j \in{\mathbb{Z}}}K_{i, j}(R, V) \otimes S(-i-j).$ \par
    It is a common practice to list the Betti numbers $\beta_{i, j}$ in a table called the \textbf{Betti table} whose rows are different values of $j$ and columns are values of $i$ and the $(j, i)^{th}-$entry is $\beta_{i, i+j} = \mathrm{dim}(K_{i, j}(R, V)).$ The last column is $i = \mathrm{pd}(R)$ and the last row is $j =  j_{\max}$ which is the maximum possible weight in the entire resolution (\ref{free_resolution_of_R}), which is the same as the Castelnuovo Mumford regularity of $R$ as defined in the next section (see Definition-(\ref{CM_definition})). We refer the reader to \cite{Gre84I} and the excellent expositions and books \cite{Eis05}, \cite{Gre89} and \cite{Laz89} for more information about Koszul cohomology and the theory of syzygies. If $R = R(\mathcal{F}, L),$ we denote $K_{i, j}(R, V)$ by $K_{i, j}(\mathcal{F}, L).$ \\
    In the Betti table, we denote any indeterminate Betti number (could be zero or non-zero) as $\ast$ and any non-zero indeterminate Betti number as $\$$.
\end{definition}
\begin{definition} $\left[ \ (M_{q})-\textbf{property} \ \right]$
    \\ Let $X \subseteq \mathbb{P}^{r}$ be a smooth projective variety. The homogeneous coordinate ring $S_{X} = S_{X/\mathbb{P}^{r}}$ is said to satisfy \textbf{property}$-(\mathit{M_{q}})$ if its resolution (\ref{free_resolution_of_R}) for $R = S_{X}$ does not have any weight-one syzygies starting from the end (stage$-(r-1)$) of the resolution till stage$-(r-q),$ meaning
    \begin{multicols}{3}
\hspace{0.5cm} $\text{ }$ \par $F_{i} = \bigoplus_{j \geq 2} S(-i-j)^{\oplus \beta_{i,\, i+j}}$ \par \vspace{0.5cm}
\hspace{0.5cm} $ r-1 \geq i \geq r-q$ \\
\begin{tabular}{c | c c c c c c c}
\cline{2-8}
& \multicolumn{7}{c}{Betti table for $(M_q)$-property of $S_{X}$} \\
\cline{1-8}
$j \diagdown i$ & $0$ & $1$ & $\cdots$ & $r-q-1$ & $r-q$ & $\cdots$ & $r-1$ \\
\cline{1-8}
& & & & & & & \\
0 & 1 & $\ast$ & $\cdots$ & $\ast$ & 0 & $\cdots$ & 0 \\
1 & $\ast$ & $\ast$ & $\cdots$ & $\ast$ &
\cellcolor[gray]{0.9}\color{red}{0} &
\cellcolor[gray]{0.9} $\cdots$ &
\cellcolor[gray]{0.9}\color{red}{0} \\
2 & $\ast$ & $\ast$ & $\cdots$ & $\ast$ & $\ast$ & $\cdots$ & $\ast$ \\
$\vdots$ & $\vdots$ & $\vdots$ & $\ddots$ & $\vdots$ & $\vdots$ & $\ddots$ & $\vdots$ \\
$j_{\max}$ & $\ast$ & $\ast$ & $\cdots$ & $\ast$ & $\ast$ & $\cdots$ & $\ast$ \\
\end{tabular} \\
\hspace{0.5cm}$ \text{ }$
\end{multicols}
\end{definition}
\begin{remark}
    Similarly, one says that $S_{X}$ satisfies property-$(N_{p})$ if $S$ is normal and $F_{i} = S(-i-1)^{\oplus \beta_{i, i+1}}$ for all $1 \leq i \leq p.$ In other words, if $S_{X}$ has only weight-one syzygies from stage $i = 1$ to stage $i = p$ of the resolution (\ref{free_resolution_of_R}) for $R = S_{X}.$ Notice that, if $L$ is a line bundle on $X,$ then the properties-$(N_{p})$ and $(M_{q})$ of the line bundle $L$ (or of the section ring $R(L)$) defined in Definition \ref{Intro_def_$(M_q)$} of the Introduction imply properties-$(N_{p})$ and $(M_{q})$ of $S_{X}$ respectively, due to projective normality of $L$. However, property-$(m_{q})$ does not necessarily imply any weight-one syzygy vanishing of $S_{X}$ when projective normality is absent.  \\
    If $S_{X}$ satisfies both properties $(N_{p})$ and $(M_{q}),$ then its Betti table looks like the following
    \begin{center}
        \begin{small}
        \begin{tabular}{c | c c c c c c c c c c}
        \cline{2-11}
        & & & \multicolumn{7}{c}{Betti table for $(N_{p})$ and $(M_q)-$properties} & \\
        \cline{1-11}
$j  \diagdown i$ & $0$ & $1$ & $\cdots$ & $p$ & $p+1$ & $cdots$ & $r-q-1$ & $r-q$ & $\cdots$ & $r-1$  \\
\cline{1-11}
&  &  &  & &  &  &  & & & \\
$0$ & \cellcolor[gray]{0.9} \color{blue}{$1$} & $0$ & $\cdots$ & $0$ & $\ast$ & $\cdots$ & $\ast$ & \ $0$ & $\cdots$ & \ $0$ \\
$1$ & $0$ & \cellcolor[gray]{0.9} \color{blue}{\$} & \cellcolor[gray]{0.9} $\cdots$ & \cellcolor[gray]{0.9} \color{blue}{\$} & $\ast$ & $\cdots$ & $\ast$ & \cellcolor[gray]{0.9} \color{red}{$0$} & \cellcolor[gray]{0.9} $\cdots$ & \cellcolor[gray]{0.9} \color{red}{$0$} \\
$2$ & $0$ & $0$ & $\cdots$ & $0$ & $\ast$ & $\cdots$ & $\ast$ & $\ast$ & $\cdots$ & $\ast$ \\
$\vdots$ & $\vdots$ & $\vdots$ & $\ddots$ & $\vdots$ & $\vdots$ & $\ddots$ & $\vdots$ & $\vdots$ & $\ddots$ & $\vdots$ \\
$j_{\max}$ & $0$ & $0$ & $\cdots$ & $0$ & $\ast$ & $\cdots$ & $\ast$ & $\ast$ & $\cdots$ & $\ast$
\end{tabular}
\end{small}
\end{center}
In Section \ref{Open_problems_examples}, through examples we explore the comparison between properties-$(M_{q})$ and $(N_{p})$  and determine or give conjectures on the estimates of $p = p_{\max}, \ q = q_{\max}$ of the highest possible $p$ and $q$ for properties-$(N_{p})$ and $(M_{q})$ respectively for a chosen embedding.
\end{remark}
\begin{remark}
    From Definition \ref{Intro_def_$(M_q)$}, it follows that if a line bundle $L$ satisfies Property$-(M_{q}),$ then the homogeneous coordinate ring $S_{X}$ also satisfies Property$-(M_{q})$ and by \cite{Gre84I}, an ample and base point free line bundle $L$ satisfies Property$-(M_{q})$ for all $1\leq q \leq n-1.$ So, we assume that $q \geq n.$
\end{remark}
Now we make the definition of the syzygy bundle which captures the syzygies of line bundles in an extremely proficient way and all of our results are obtained through the vanishing of the cohomologies of vector bundles which originate from this bundle with (\ref{SamplingVectorBundle}) as the theme principle.
\begin{definition} $\left[ \ \textbf{Syzygy Sheaf and Syzygy Bundle} \ \right]$ \\
Let $\mathcal{E}$ be a globally generated coherent sheaf on a smooth projective variety $X$. Then, we define the \textbf{syzygy sheaf} $M_{\mathcal{E}}$ of $\mathcal{E}$ by the following short-exact sequence
\begin{equation}
\label{kernel_bundle_general}
    0 \rightarrow M_{\mathcal{E}} \rightarrow H^{0}(\mathcal{E}) \otimes \mathcal{O}_{X} \xrightarrow[]{\mathrm{ev}_{\mathcal{E}}} \mathcal{E} \rightarrow 0
\end{equation}
If $\mathcal{E} =  E$ is a vector bundle, then the syzygy sheaf $M_{E}$ is a vector bundle and we call it the \textbf{syzygy bundle} or the \textbf{kernel bundle} and call (\ref{kernel_bundle_general}) the \textbf{syzygy bundle sequence} or the \textbf{kernel bundle sequence} for $\mathcal{E} = E.$ \par
\textbf{ } In the situation of (\ref{moduleofsections}), if  we replace our coherent sheaf $\mathcal{F}$ with a vector bundle $E$ on $X$ and we assume our line bundle $L$ is base point free.  Then $L$ defines a projective morphism $\phi_{L} : X \rightarrow \mathbb{P}^{r}$ and thus we have the kernel bundle sequence:
\begin{equation}
\label{kernel_bundle}
0 \rightarrow M_{L} \rightarrow H^{0}(L) \otimes \mathcal{O}_{X} \rightarrow L \rightarrow 0
\end{equation}
This sequence is the pull-back by $\phi_{L}$ of the following twisted Euler-sequence in $\mathbb{P}^{r}$
\begin{equation}
\label{Twisted_Euler_Sequence}
0 \rightarrow \Omega^{1}_{\mathbb{P}^{r}}(1) \rightarrow \mathcal{O}_{\mathbb{P}^{r}}^{r+1} \rightarrow \mathcal{O}_{\mathbb{P}^{r}}(1) \rightarrow 0
\end{equation}   
\end{definition}
\begin{theorem}
\label{SamplingVectorBundle}
In this situation, one has
\begin{equation}
K_{p, q}(X, E, L) = \mathrm{Ker} \left[ H^{1}(\bigwedge^{p+1}M_{L}\otimes E((q-1)L)) \rightarrow \bigwedge^{p+1}H^{0}(L) \otimes H^{1}(E((q-1)L)) \right]
\end{equation}
$$= \mathrm{Coker} \left[ \bigwedge^{p+1}H^{0}(L) \otimes H^{0}(E((q-1)L)) \rightarrow H^{0}(\bigwedge^{p}M_{L} \otimes E((q-1)L)) \right]$$
\end{theorem}
In the following we give a restatement of \cite[Lemma-(2.1)]{AL15} which we will use as a criterion for vanishing of weight-one Koszul cohomology of line bundles. 
\begin{lemma} 
\label{AL_Lemma}
\cite[Lemma-(2.1)]{AL15}
    Let $L$ be an ample and globally generated line bundle on a smooth projective variety $X$ of dimension $n \geq 2,$ and let $B$ be a line bundle on $X$ such that $L-B$ is ample. Let $r = h^{0}(L)-1$ and  $n \leq q \leq r-1.$ Then if the multiplication maps of global sections:
    $$
        H^{0}(L) \otimes H^{0}(M^{\otimes k}_{L} \otimes (K_{X}+(n-1)L-B)) \overset{\mu_{k}}{\longrightarrow} H^{0}(M^{\otimes k}_{L} \otimes (K_{X}+nL-B)),
       $$
       \begin{equation}
       \label{AL_maps}
        \text{ } \hspace{0.3cm} \text{ for } 0 \leq k \leq q-n
    \end{equation}
    are surjective, then
    \begin{small}
    \begin{equation}
        K_{r-k-n, 1}(X, B, L) = H^{1}(\wedge^{(r-k)-n}M_{L} \otimes B) = H^{1}(\bigwedge^{k+1}M_{L} (K_{X}+(n-1)L-B)) = 0, \hspace{0.5cm} \text{ for all } 0 \leq k \leq q-n.
    \end{equation}  
    \end{small}
    In particular, taking $i = r-k-n,$ $K_{i, 1}(X, B, L) = 0,$ for all $r-1 \geq i \geq r-q$
\end{lemma}
\begin{remark}
Taking $B = \mathcal{O}_{X}$ in the above Lemma-(\ref{AL_Lemma}), and applying Definition \ref{Intro_def_$(M_q)$}, one has $L$ satisfies Property-$(m_{q})$ if the maps $\mu_{k}$ in (\ref{AL_maps})
are surjective. \par
Tensoring the short-exact sequence (\ref{kernel_bundle}) by $K_{X}+(n-1)L$ and taking the long-exact sequence of cohomology
\begin{itemize}
\item[\textbf{(1)}] By Kodaira vanishing, $H^{1}(K_{X}+ (n-1)L) = 0,$ so $L$ satisfies $(M_{n})$ if \\ $H^{1}(M_{L} \otimes (K_{X} + (n-1)L)) = 0.$
\item[\textbf{(2)}] For any integer $k$ with $0 \leq k \leq r-n-1$ if a line bundle satisfies $(M_{n+k-1}),$ then it satisfies $( M_{n+k})$ if $H^{1}(M^{\otimes (k+1)}_{L} \otimes (K_{X}+(n-1)L) = 0$
\end{itemize}
\end{remark}
Now we present some methods of proving the surjectivity of the maps $\mu_{k}$ in (\ref{AL_maps}) which we will use throughout this article.
\subsection{Castelnuovo Mumford regularity and surjectivity of global sections}
Here we recall the definition of Castelnuovo Mumford regularity (CM regularity or regularity for short) and summarize its implications. For a more detailed discussion on the topic we refer the reader to \cite[Section (1.8)]{Laz04} and \cite[Chapter 4]{Eis05}.
\begin{definition} \textbf{(CM regularity)} \\
\label{CM_definition}
    Let $X$ be a smooth projective variety of dimension $n$ and let $B$ be an ample and base point free line bundle on $X.$ We say that a coherent sheaf $\mathcal{F}$ on $X$ has \textbf{CM regularity} $\rho$ with respect to $B$ if $\rho$ is the smallest integer $k$ such that we have the vanishing
    \begin{equation}
        H^{i}(\mathcal{F} \otimes (k-i)B) = 0, \hspace{0.5cm} \text{ for all } 1 \leq i \leq n
    \end{equation}  
    We denote the regularity $\rho$ by $\mathrm{reg}_{B}(\mathcal{F})$. A coherent sheaf $\mathcal{F}$ is called $k-$regular with respect to $B$ if $k \geq \mathrm{reg}_{B}(\mathcal{F}).$
\end{definition}
Regarding regularity the most important well-known technique we use is the Castelnuovo lemma below. For a proof and related discussion, please see \cite[Theorem 1.8, Section 1.8]{Laz04} and \cite[Corollary 4.18, Section 4D]{Eis05}.
\begin{lemma} \textbf{(Castelnuovo Lemma)}
\label{Castelnuovo_lemma}
    For any ample and base point free line bundle $B$ and a coherent sheaf $\mathcal{F}$ on a projective variety $X$ such that $\mathcal{F}$ is $m-$regular with respect to $B.$ Then for any integer $k \geq 0$ one has the following
    \begin{itemize}
        \item[(a)] The sheaf $\mathcal{F}((m+k)B)$ is globally generated.
        \item[(b)] The multiplication map $H^{0}(kB) \otimes H^{0}(\mathcal{F}(m B)) \xrightarrow[]{\mu} H^{0}((k+m)B)$ is surjective.
        \item[(c)] $\mathcal{F}$ is $(m+k)-$regular with respect to $B$.
    \end{itemize}
\end{lemma}
In the following remark we mention an interpretation of regularity in terms of the Betti table.
\begin{remark}
    In particular, part $(b)$ of the above Lemma implies that $L = \ell B$ is projectively normal for all $\ell \geq \mathrm{reg}_{B}(\mathcal{O}_{X}).$ \\
    $\text{ }$ From \cite[Theorem (1.8.26)]{Laz04}, if $h^{0}(\mathcal{F}(-dB)) = 0$ for $d \ll 0,$ $\rho$ is the highest possible weight $j = j_{\max}$ in the minimal graded free resolution of $R = R(\mathcal{F}, B)$ as a $S-$module, where $S = \mathrm{Sym}(H^{0}(B)),$ or in other words, $j = \rho$ is the last row of the Betti table of $R.$ \\
    $\text{ }$ Note that, the index $j$ of rows Betti table of $R$ (weights of $R$) can start from a negative integer. This happens especially if $\rho < 0.$ From the definition of Koszul cohomology Definition \ref{definition_koszul_cohomology}, this cannot happen when $\mathcal{F} = \mathcal{O}_{X}$ ($R = R(B)$). This gives an alternative proof of (3) in the following proposition where we show that in the case $\mathcal{F} = \mathcal{O}_{X}$ one has $\rho \geq 0$.
\end{remark}
Now, we give estimates of regularity of some important line bundles in the following proposition.
\begin{proposition}
\label{regularity_examples}
    For any ample and base point free line bundle $B$ on a smooth projective variety $X$ of dimension $n,$ one has the following
    \begin{itemize}
    \item[\textbf{(1)}] The regularity $\mathrm{reg}_{B}(K_{X})$ of the canonical sheaf satisfies the inequality
    $$
        \mathrm{reg}_{B}(K_{X}) \leq n+1
    $$
    
        \item[\textbf{(2)}]  For any integer $\ell,$ since by definition $\mathrm{reg}_{B}(K_{X}+\ell B) = \mathrm{reg}_{B}(K_{X})-\ell,$ the regularity of $K_{X}+\ell B$ satisfies
    $$
        \mathrm{reg}_{B}(K_{X}+\ell B) \geq (n+1)-\ell
    $$
        \item[\textbf{(3)}] The regularity $\mathrm{reg}_{B}(\mathcal{O}_{X})$ of the structure sheaf satisfies the inequality
    $$
        \mathrm{reg}_{B}(\mathcal{O}_{X}) \geq 0
    $$
    Moreover, from the definition of Koszul cohomology, $K_{i, j}(X, B) = 0$ for any $j < 0,$ and therefore the Betti table of $R(B)$ starts from the row $j = 0$ and ends at $j = \mathrm{reg}_{B}(\mathcal{O}_{X}).$
        \item[\textbf{(4)}]  For any integer $\ell,$ since by definition $\mathrm{reg}_{B}(\ell B) = \mathrm{reg}_{B}(\mathcal{O}_{X})-\ell,$ the regularity $\mathrm{reg}_{B}(\ell B)$ of $\ell B$ satisfies
    $$
        \mathrm{reg}_{B}(\ell B) \geq -\ell
    $$
    \end{itemize}
        In all of the above inequalities, equality occurs if and only if $(X, B) = (\mathbb{P}^{n}, \mathcal{O}(1))$ where $\mathcal{O}(1)$ is the linear system of hyperplanes in $\mathbb{P}^{n}.$
\end{proposition}
\begin{proof}
    (1) follows from Kodaira vanishing theorem. (2) follows from (1). To show (3) we prove that $\mathcal{O}_{X}$ cannot be $(-1)-$regular with respect to $B.$ This is true because by Serre duality $H^{n}(-(n+1)B) = H^{0}(K_{X}+(n+1)B)$ and since by (2) $K_{X}+(n+1)B$ is $0-$regular with respect to $B$, by Lemma \ref{Castelnuovo_lemma} $K_{X}+(n-1)B$ is globally generated. This implies $H^{0}(K_{X}+(n+1)B) \neq 0.$ (4) follows from (3)
\end{proof}
\subsection{Method of hyperplane sections}
\label{hyperplane_cuts}
Here we introduce the method of hyperplane cuts or hyperplane sections in syzygies.
\begin{definition} \textbf{(Stratification by Hyperplane cuts)} \\
\label{definition_stratification_by_Hyperplane_cuts}
    For any ample and base point free line bundle $L$ on a smooth projective variety $X$ of dimension $n,$ one has $h^{0}(L) \geq n+1.$ From Bertini's theorem, one can choose a linear subspace $(n-1)-$general sections $W = \left\{s_{1}, \cdots, s_{n-1} \right\}$ of $V = H^{0}(L)$ such that one obtains a filtration
    \begin{equation}
    \label{hyperplane_stratification}
        X_{L}^{(n-1)} \subset X_{L}^{(n-2)} \subset \cdots \subset X_{L}^{(1)} \subset X_{L}^{(0)} = X 
    \end{equation}
    of smooth closed subschemes of $X$, where $X_{L}^{(0)} = X, \ X_{L}^{(1)}$ is a smooth member in $|L|,$ for any $0 \leq c \leq n-1$ each $X_{L}^{(c+1)}$ is a smooth member in $|L|_{X_{L}^{(c)}}|$ and each index $c$ of $X_{L}^{(c)}$ denoting the codimension $c = \mathrm{codim}(X_{L}^{(c)}, X)$ of $X_{L}^{(c)}$ in $X.$ We drop the suffix $L$ from $X^{(c)}_{L}$ and write $X^{(c)}$ when the line bundle $L$ is understood from the context. We call the stratification (\ref{hyperplane_stratification}) a \textbf{hyperplane stratification} of $X$ by $L.$ We say that $X^{(c)}$ is the $c^{th}-$\textbf{floor} of the stratification.
    \end{definition}
    Now we mention some properties of the hyperplane stratification that follow immediately from adjunction formula
    \begin{remark}
    \label{adjunction_hyperplane_strata}
        For any integer $1 \leq c \leq n-1,$ by adjunction formula $K_{X^{(c)}} = (K_{X}+cL)|_{X^{(c)}}.$ Thus by Nakai-Moishezon criterion, if $K_{X}$ is nef, each $K_{X^{(c)}}$ is ample and therefore in particular $X^{(c)}$ is a variety of general type for each integer $1 \leq c \leq n-1.$ We also denote the smooth projective curve $X^{(n-1)} \subset X$ by $C$ and the surface of general type $X^{(n-2)}$ by $Y.$ Since $K_{C}$ is ample, the curve $C$ has genus $g = (\Lambda \cdot B^{n-1}) \geq 2,$ where $\Lambda = K_{X}+(n-1)B.$ \par
        Also, if $X$ is a Fano variety of index $\lambda,$ by which we mean that $\lambda$ is the smallest non-negative integer such that $-K_{X} = \lambda H$ with $H$ ample and base point free line bundle on $X.$ Then, by adjunction, all the subvarieties $X^{(c)}$ are Fano varieties if $\lambda H-cL$ is ample. In particular, when $X$ is a Fano variety of index $\lambda,$ taking $L = H$ one has $Y$ is a del Pezzo surface if $\lambda \geq n-1,$ a $K3-$surface if $\lambda = n-2$ and a surface of general type if $\lambda \leq n-3.$ \par
        This shows that the method of hyperplane sections depends on the behavior of the canonical divisor $K_{X}$ and in particular on the Kodaira dimension of the variety.
    \end{remark}
    Philosphically, the method of hyperplane cuts relates the syzygies or the surjectivity of multiplication maps of global sections of vector bundles on a variety with the syzygies or the surjectivity of global sections of vector bundles restricted to a hyperplane section. In this way, one can obtain results on syzygies of higher dimensions using the syzygies of lower dimensions. There are multiple techniques in the method of hyperplane sections. We explore one of them here in the following proposition and explore other techniques for different Kodaira dimensions throughout the article. \\
    The navigation of syzygies between the floors of the stratification becomes easier if the cohomologies of the line bundle and the variety are zero. Thus, we make the following definitions
    \begin{definition}
        We say that a smooth projective variety $X$ of dimension $n$ is \textbf{regular} if $h^{c}(\mathcal{O}_{X}) = 0$ for all $1 \leq c \leq n-1.$ A line bundle $L$ is said to be \textbf{non-special} if $h^{i}(qL) = 0$ for all integers $1 \leq i \leq n$ and $q \geq 1.$
    \end{definition}
    Note that the notion regular variety generalizes the notion of a regular surface. A surface $X$ is called regular if $h^{1}(\mathcal{O}_{X}) = 0.$ While regular varieties include Fano varieties, Calabi-Yau varieties, rational surfaces, Enriques surfaces and many other class of varieties. \par
    We now state the following proposition which demonstrates this navigation and will be used especially in Section \ref{kodaira_negative_infinity}.
    \begin{proposition}
    \label{syzygies_of_floors}
     A non-special projectively normal line bundle $L$ on a regular smooth projective variety $X,$ satisfies Property-$(N_{p})$ if and only if any of the line bundles $L_{X^{(c)}}$ satisfies Property-$(N_{p})$ and $L$ satisfies property-$(M_{q})$ if and only if any of the line bundles $L_{X^{(c)}}$ satisfies property-$(M_{q-c})$ for all $0 \leq c \leq n-1.$
    \end{proposition}
    \begin{proof}
        It is well known that the ring of sections $R(L)$ is Cohen-Macaulay if and only if 
        $$H^{i}(qL) = 0 \text{ for all } 1 \leq i \leq n-1 \text{ and all integers } q$$
    Since $X$ is a regular smooth projective variety, $h^{i}(\mathcal{O}_{X}) = 0$ for integers $1 \leq i \leq n-1.$ Since $L$ is non-special, that is $h^{i}(qL) = 0$ for all integers $q > 0.$ Then, by Kodaira vanishing theorem $h^{i}(-qL) = 0$ for $1 \leq i \leq n-1$ and therefore we conclude by our assumption that $R(L)$ is Cohen-Macaulay. When $L$ is projectively normal (that is satisfies Property-$(N_{0})$), then $R(L) = S_{X},$ the homogeneous coordinate ring of $X$ by embedding of $L.$ Thus, if $X$ is regular and $L$ is non-special, then the embedding by $|L|$ is arithmetically Cohen-Macaulay. Thus, for any smooth effective divisor $D$ in $|L|,$ the graded Betti numbers of $L$ and $L_{D} = L|_{D}$ are same, that is $K_{i, j}(X, L) \cong K_{i, j}(D, L_{D})$ as vector spaces, for all integers $i, j.$ \par
    Thus, we conclude that $L$ satisfies Property-$(N_{p})$ if and only if $L_{D}$ satisfies Property-$(N_{p}).$ \par
    However, the case for Property-$(M_{q})$ is slight different in nature. Let $r = r(L) = h^{0}(L)-1$ and $r_{D} = r(L_{D}) = h^{0}(L_{D})-1.$ Then by our regularity hypothesis and by the long-exact sequence corresponding to the following short-exact sequence
    $$0 \rightarrow \mathcal{O}_{X} \rightarrow L \rightarrow L_{D} \rightarrow 0$$
    we get $r = r_{D}+1.$ Then,
    \begin{equation}
        K_{r-q, 1}(X, L) \cong K_{r_{D}+1-q, 1}(D, L_{D}) = K_{r_{D}-(q-1), 1}(D, L_{D})
    \end{equation}
    Thus, we conclude that $L$ satisfies Property-$(M_{q})$ of $L$ if and only if $L_{D}$ satisfies Property-$(M_{q-1}).$ Therefore, by induction on $c,$ the proof is complete.
    \end{proof}
    Thus, in the perspective of the above Proposition \ref{syzygies_of_floors} the obstruction to the property-$(M_{q})$ are the obstructions to navigation between floors which are cohomologies of the structure sheaf and the line bundle and the obstruction to the property-$(M_{q})$ of the curve $C.$ In the next sections, we use CM regularity and the techniques below to show the vanishings of these obstructions and introduce techniques to overcome their absence. A great tool during the absence of vanishing of cohomology of $X,$ one uses the following complex.
    \begin{definition}\textbf{(Koszul resolution)}
    In the situation of Definition \ref{definition_stratification_by_Hyperplane_cuts}, one has the Koszul resolution of $C$
    \begin{equation}
    \label{koszul_resolution}
        0 \rightarrow \bigwedge^{n-1} W \otimes \mathcal{O}_{X}(-(n-1)L) \rightarrow \cdots \rightarrow \bigwedge^{2} W \otimes \mathcal{O}_{X}(-2L) \rightarrow W \otimes \mathcal{O}_{X}(-L) \rightarrow \mathcal{I}_{C} \rightarrow 0
    \end{equation}
\end{definition}
    The technique of proving the vanishing of obstruction is easier for multiple line bundles and only requres CM regularity which we explore in the next section. But for other line bundles it is more complicated and depends on the Kodaira dimension of the variety and one uses method of hyperplane sections to navigate between floors and slopes of vector bundles on curves to prove results on syzygies of curves needed here. We explore these in Sections \ref{nef_canonical} and \ref{kodaira_negative_infinity}. \par
    Now we collect some known results to be used to show the vanishing of these obstructions.
\begin{observation} \cite[Observation-(1.2)]{BG00}
\label{[Observation-(1.2)]{BG00}}
For coherent sheaves $E_{1}, \cdots, E_{t}, F$ be coherent sheaves on a projective variety $X.$ Then, the multiplication map
\begin{equation}
H^{0}(E_{1} \otimes \cdots \otimes E_{t}) \otimes H^{0}(F) \overset{\sigma}{\rightarrow} H^{0}(E_{1} \otimes \cdots \otimes E_{t} \otimes F)
\end{equation}
is surjective if all the following maps
\begin{equation}
\begin{cases}
H^{0}(E_{1}) \otimes H^{0}(F) \overset{\sigma_{1}}{\longrightarrow} H^{0}(E_{1} \otimes F) \\
H^{0}(E_{2}) \otimes H^{0}(E_{1} \otimes F) \overset{\sigma_{2}}{\longrightarrow} H^{0}(E_{1} \otimes E_{2} \otimes F) \\
\ddots \hspace{1cm} \ddots \hspace{1cm} \cdots \hspace{1cm} \ddots \hspace{1cm} \ddots \\
H^{0}(E_{t}) \otimes H^{0}(E_{1} \otimes \cdots \otimes E_{t-1} \otimes F) \overset{\sigma_{t}}{\rightarrow} H^{0}(E_{1} \otimes \cdots \otimes E_{t} \otimes F)
\end{cases}
\end{equation}
are surjective.
\end{observation}
\begin{observation} \cite[Observation-(1.3)]{BG00}
\label{[Observation-(1.3)]{BG00}}
Let $X$ be a regular variety (that is a variety such that $q = h^{1}(\mathcal{O}_{X}) = 0$). Let $E$ be a vector bundle and $D$ be a base point free divisor on $X$ such that $H^{1}(E(-D)) = 0.$ Then the surjectivity of the multiplication map
$$
H^{0}(E \otimes \mathcal{O}_{D}) \otimes H^{0}(\mathcal{O}_{D}(D)) \rightarrow H^{0}(E(D) \otimes \mathcal{O}_{D})
$$
implies the surjectivity of
\begin{equation}
H^{0}(E) \otimes H^{0}(\mathcal{O}_{X}(D)) \rightarrow H^{0}(E(D)))
\end{equation}
\end{observation}
\begin{lemma} \cite[Lemma-(2.9)]{BG99}
\label{BG99_Lemma_(2.9)}
Let $q$ be a non-negative integer. For globally generated vector bundles $\mathcal{F}_{1}, \cdots, \mathcal{F}_{q},$ a line bundle $R,$ a reduced and irreducible member $Q$ in a complete linear series $|P|$ of an effective line bundle $P$ on a smooth projective variety $X,$ the collection $\mathfrak{F}|Q = \left\{ \mathcal{F}_{1}, \cdots, \mathcal{F}_{q}; Q \right\},$ $M = M_{\mathfrak{\mathcal{F}}|Q}$ and a sheaf $G$ on $Q$ satisfying
\begin{itemize}
\item[\textbf{(i)} ] the vanishing $H^{1}(\mathcal{F}_{i}(-Q)) = 0,$
\item[\textbf{(ii)} ] and the surjectivity of the following multiplication map $$H^{0}(M[\underline{i}, q](R)) \otimes H^{0}(G) \rightarrow H^{0}(M[\underline{i}, q] \otimes G(R))$$ 
for all index subset $(\underline{i}) = \left\{i_{1}, \cdots, i_{q'} \right\} \subseteq (\underline{q}) = \left\{1, 2, \cdots, q \right\},$ 
\end{itemize}
one has the surjectivity of the multiplication map
$$H^{0}(M[\underline{i}, k](R)) \otimes H^{0}(G) \rightarrow H^{0}(M[\underline{i}, k] \otimes G(R)),$$
for all $0 \leq k \leq q'$ and all index subset $(\underline{i}) = \left\{i_{1}, \cdots, i_{q'} \right\} \subseteq (\underline{q}) = \left\{1, 2, \cdots, q \right\}.$ 
\end{lemma}
\subsection{Some results on Vector bundles on curves}
Here we collect results on vector bundles on curves. Throughout this section, let $C$ be a smooth projective curve of genus $g.$
\begin{definition} (Minimal and maximal slope of vector bundles) \\
    For a vector bundle $E$ on $C$ of rank $\rho = \rho(E) = \mathrm{rank}(E)$ and degree $\delta(E) = \mathrm{deg}(E) := \mathrm{deg}(\bigwedge^{\rho}E),$ one defines the \textbf{slope} of $E$ as $\mu(E) := \delta(E)/\rho(E).$ Also, the \textbf{maximal slope} $\mu^{+}(E)$ and \textbf{minimal slope} $\mu^{-}(E)$ are defined by
    $$\mu^{+}(E) := \mathrm{max}\left\{ \mu(K) \ | \ 0 \rightarrow K \rightarrow E \ \right\}, \ \hspace{0.2cm} \ \mu^{-}(E) := \mathrm{min}\left\{ \mu(Q) \ | \ E \rightarrow Q \rightarrow 0 \ \right\} $$  
\end{definition}
One has $\mu^{-}(E) \leq \mu(E) \leq \mu^{+}(E)$ with equality if and only if $E$ is a semistable vector bundle (see \cite{Huy10} and \cite{Miy87}).
\begin{lemma} \cite[Section 1 and 2 ]{Bu94} \\
\label{[Section (1) and (2) ]{Bu94}}
For vector bundles $E, F$ and $G$ on $C$ we get
\begin{itemize}
    \item[(1)] $\mu^{+}(E \otimes F) = \mu^{+}(E)+\mu^{+}(F)$
    \item[(2)] $\mu^{-}(E \otimes F) = \mu^{-}(E)+\mu^{-}(F)$
    \item[(3)] $\mu^{+}(Sym^{\ell}F) = \ell \cdot \mu^{+}(E)$
    \item[(4)] $\mu^{-}(Sym^{\ell}F) = \ell \cdot \mu^{-}(E)$
    \item[(5)] $\mu^{-}(\bigwedge^{\ell}F) \geq \ell \cdot \mu^{-}(E)$
    \item[(6)] If $\mu^{-}(E) > 2g-2,$ then $h^{1}(C, E) = 0$
    \item[(7)] If $\mu^{-}(E) > 2g-1$ then $E$ is globally generated
    \item[(8)] If $\mu^{-}(E) > 2g,$ then $\mathcal{O}_{\mathbb{P}(E)}(1)$ is very ample
    \item[(9)] If $0 \rightarrow E \rightarrow F \rightarrow G \rightarrow 0$ is a short-exact sequence of vector bundles on $C,$ then $\mu^{-}(G) \geq \mu^{-}(F) \geq \mathrm{min}(\mu^{-}(E), \mu^{-}(G))$
\end{itemize}
\end{lemma}
\begin{remark}
   For example, from Lemma \cite[Section 1 and 2 ]{Bu94} above, we obtain that for $\mu^{-}(E) > 2g-1$ one has a well-defined kernel bundle $M_{E}.$
\end{remark}
\begin{theorem} \cite[Theorem 1.2]{Bu94}
\label{[Theorem-(1.2)]{Bu94}}
Let $E$ be a semistable vector bundle on a curve $C$ of genus $g$ with $\mu(E) \geq 2g.$ Then, $M_{E}$ is semistable and 
$$\mu(M_{E}) = \dfrac{-\mu(E)}{\mu(E)-g} \geq -2.$$
Furthermore, if $E$ is stable and $\mu(E) \geq 2g$ and either $C$ is hyperelliptic or $\Omega^{1}_{C} \hookrightarrow E.$
\end{theorem}
\begin{corollary} \cite[Corollary 3.7]{Miy87}
\label{[Corollary-(3.7)]{Miy87}}
Let $E$ and $F$ be semistable vector bundles on a curve $C.$ Then the vector bundles $E \otimes F$ and $\underline{\mathrm{Hom}}_{X}(E, F)$ are both semistable.
\end{corollary}
\begin{proposition} \cite[Proposition 2.2]{Bu94}
\label{Proposition-(2.2)_Bu94}
    Let $E$ and $F$ be vector bundles on the curve $C$ with
    \begin{align*}
        \mu^{-}(F) > 2g, \\
        \mu^{-}(F) > 2g+\mathrm{rank}(E)(2g-\mu^{-}(E))-2h^{1}(C, E)
    \end{align*}
    then, the multiplication map $H^{0}(E) \otimes H^{0}(F) \rightarrow H^{0}(E \otimes F)$ is surjective. Moreover, if $F$ is stable, then one can loosen the condition-(1) to $\mu(F) \geq 2g.$
\end{proposition}
\section{Castelnuovo Mumford Regularity and the $(M_{q})$ property}
In this section, we use the Castelnuovo Mumford regularity of the canonical bundle $K_{X}$ of a smooth projective variety to obtain results on the syzygies of multiples $\ell B$ of an ample and base point free line bundle $B.$ We first recall the definition of Castelnuovo Mumford regularity (CM regularity) below
\subsection{Property \( (M_{q}) \) for multiple line bundles}  $\text{ }$ \\
    We start out by proving some basic lemmas about the CM regularity of the syzygy bundle and tensor product of syzygy bundles with other vector bundles.
\begin{lemma}
\label{regularity_of_Syzygy_Bundle_multiple_line_bundles}
Let $B$ be an ample and base point free line bundle on a smooth projective variety $X.$ Then,
\begin{itemize}
\item[\textbf{(i)} ]  With respect to $B,$ the vector bundle $M_{mB} \otimes \mathcal{E}$ is $(\rho+1)-$regular if the sheaf $\mathcal{E}$ is $\rho-$regular. In particular, $$\mathrm{reg}_{B}(M_{mB} \otimes \mathcal{E}) \leq 1+\mathrm{reg}_{B}(\mathcal{E}).$$
\item[\textbf{(ii)} ]  Moreover, for the non-negative integer vector $\vec{m} = (m_{1}, \cdots, m_{k}),$ with respect to the line bundle $B,$ the bundle $M_{\vec{m}B}(\mathcal{E}) = \bigotimes^{k}_{i = 1}M_{m_{i}B} \otimes \mathcal{E}$ is $(\rho+k)-$regular if the sheaf $\mathcal{E}$ is $\rho-$regular. In particular, 
$$\mathrm{reg}_{B}(M_{\vec{m}B}(\mathcal{E})) \leq k+\mathrm{reg}_{B}(\mathcal{E})$$
\item[\textbf{(iii)} ] For the non-negative integer vector $\vec{m} = (m_{1}, \cdots, m_{k}),$ the vector bundle $M_{\vec{m}B}(K_{X}+\ell B)$ is $k-$regular, with respect to $B,$ for all $\ell \geq \mathrm{reg}_{B}(K_{X})$. In other words, 
\begin{equation}
    \mathrm{reg}_{B}(M_{\vec{m}B}(K_{X}+\ell B)) \leq k \ , \hspace{0.5cm} \text{ for all } \ell \geq \mathrm{reg}_{B}(K_{X}).
\end{equation}
In particular, 
\begin{equation}
    \mathrm{reg}_{B}(M_{\vec{m}B}(K_{X}+\ell B)) \leq k \ , \hspace{0.5cm} \text{ for all } \ell \geq n+1.
\end{equation}
\end{itemize}
\end{lemma}
\begin{proof}
    First we prove Part$-(i).$ It is sufficient to prove the first part of $(i)$ only. Supposing $\mathcal{E}$ is $\rho-$regular we tensor the Kernel Bundle sequence (\ref{kernel_bundle}) for $\mathcal{F} = mB$ with $\mathcal{E} \otimes (\rho+1-i)B$ and take the long-exact sequence of cohomology. Then, for $i = 1$ we have the long-exact sequence
    $$H^{0}(mB) \otimes H^{0}(\mathcal{E}(\rho B)) \xrightarrow[\hspace{1cm}]{\mu^{0, 0}_{mB, \mathcal{E}(\rho B)}} H^{0}(\mathcal{E} ((m+\rho)B))$$
    $$ \rightarrow H^{1}(M_{mB} \otimes \mathcal{E}((\rho+1-1)B)) \rightarrow H^{0}(mB) \otimes H^{1}(\mathcal{E}((\rho+1-1)B))$$
    and for $i \geq 2,$ one has the long-exact sequence
    $$H^{0}(mB) \otimes H^{i-1}(\mathcal{E}((\rho-(i-1))B)) \xrightarrow[\hspace{1cm}]{\mu^{0, i-1}_{mB, \mathcal{E}(\rho B)}} H^{i-1}((\mathcal{E}((\rho+m-(i-1))B)) $$
    $$ \rightarrow H^{i}(M_{mB} \otimes \mathcal{E}((\rho+1-i)B)) \rightarrow H^{0}(mB) \otimes H^{i}(\mathcal{E}((\rho+1-i)B))$$
    Since $\mathcal{E}$ is $\rho-$regular, $H^{i}(\mathcal{E}((\rho+1-i)B)) = 0$ for all integers $i \geq 1$ and the map $\mu_{mB, \mathcal{E}(\rho B)}^{0, 0}$ is surjective. Therefore, from the above long-exact sequences one has $H^{i}(M_{mB} \otimes \mathcal{E}((\rho+1-i)B)) = 0$ for all $i \geq 1$ which implies that $M_{mB} \otimes \mathcal{E}$ is $(\rho+1)-$regular. \\
    Now, we prove Part$-(ii).$ This part follows from induction. The $k = 1$ case follows from Part-$(i).$ Now we suppose that $k > 1$ and our induction hypothesis is that $M_{\vec{m}B}[k-1](\mathcal{E})$ is \small{$(\rho+k-1)-$}regular. Then putting $m = m_{k}$ and replacing $\mathcal{E}$ by $M_{\vec{m}B}[k-1](\mathcal{E})$ in Part-(i), we get $M_{\vec{m}B}(\mathcal{E}) = M_{mB} \otimes (M_{\vec{m}B}[k-1](\mathcal{E}))$ is \small{$((\rho+k-1)+1)-$}regular. This completes the proof. \\
    Finally, we prove Part$-(iii).$ This follows from Part$-(ii)$ by putting $\mathcal{E} = K_{X}+\ell B$ for $\ell \geq \mathrm{reg}(K_{X}),$  since by Proposition \ref{regularity_examples} (2) $K_{X}+\ell B$ is $0-$regular for $\ell \geq \mathrm{reg}_{B}(K_{X}).$ The last statement follows since by Proposition \ref{regularity_examples} (1) $\mathrm{reg}_{B}(K_{X}) \leq n+1.$
\end{proof}
In this section we give a general lower bound $\ell^{\mathrm{ceil}}_{q}$ on $\ell$ such that for any ample and base point free line bundle $B$ on an arbitrary smooth projective variety $X,$ the multiple line bundle $L = \ell B$ satisfies Property-$(M_{q}).$ Our result is analogous to \cite[Theorem 1.3]{BG99} by \textit{P. Bangere} and \textit{F. J. Gallego} but unlike the bound for property-$(N_{p})$ our bound decreases with increasing dimension.
\begin{theorem}
\label{general_vanishing_multiple_line_bundles}
    Let $X$ be a smooth projective variety of dimension $n$ and let $B$ be an ample and base point free line bundle on $X.$ \\
    Then, for any positive integer vector $\vec{m} = (m_{1}, \cdots, m_{q-n+1})$ one has the vanishing
    \begin{equation}
        H^{1}(M_{\vec{m}B} \otimes (K_{X}+\ell B)) = 0,  \hspace{0.5cm} \text{ provided } 
                    \ell \geq \ell^{\ast}_{q}(B)
    \end{equation}
    where $\ell^{\ast}_{q}(B) = q-n+\mathrm{reg}_{B}(K_{X}).$ In particular,
    \begin{equation}
            H^{1}(M_{\vec{m}B} \otimes (K_{X}+\ell B)) = 0,  \hspace{0.5cm} \text{ provided } \ell \geq q+1
    \end{equation}
    In particular, the multiple line bundle $L = \ell B$ satisfies Property-$(m_{q})$ if $\ell \geq \ell^{\mathrm{reg}}_{q}(B)$ where $\ell^{\mathrm{reg}}_{q}(B) = \ceil[\bigg]{\dfrac{q-n+\mathrm{reg}_{B}(K_{X})}{n-1}}$ and it satisfies Property-$(m_{q})$ if $\ell \geq \ell^{\mathrm{reg}}_{q}(B)$ where $\ell^{\mathrm{reg}}_{q}(B) = \max \left(\rho, \ceil[\bigg]{\dfrac{q-n+\mathrm{reg}_{B}(K_{X})}{n-1}} \right)$  \\
\end{theorem}
\begin{proof}
        We prove this using induction on $q.$ The proof is divided into three steps below. \\
     
            \textbf{(i)} \ \textbf{\underline{Base step} : } First we treat the case $q = n.$ That is, we show that given the conditions on $B$ mentioned above,
            $$H^{1}(M_{mB} \otimes (K_{X}+\ell B)) =$$ $$\mathrm{Coker} \left[ H^{0}(mB) \otimes H^{0}(K_{X}+\ell B) \rightarrow H^{0}(K_{X}+(\ell+m)B) \right] = 0$$
            If $\ell \geq \ell_{n} = \mathrm{reg}_{B}(K_{X}),$ then $K_{X}+\ell B$ is $0-$regular with respect to $B$ and therefore the following multiplication map of global sections
            $$H^{0}(mB) \otimes H^{0}(K_{X}+\ell B) \rightarrow H^{0}(K_{X}+(\ell+m)B)$$
            is surjective. \\
            \textbf{(ii)} \ \textbf{\underline{Induction hypothesis} : } Let $B$ be an ample and base point free line bundle on $X$
            Suppose that, for any $0 \leq k' \leq (q-1)-n$ and any positive integer vector $\vec{m} = (m_{1}, \cdots, m_{k'+1})$ one has the vanishing
            \begin{small}
            \begin{equation}
                H^{1} (M_{\vec{m}B} \otimes (K_{Y}+\ell' B)) = 0 \hspace{0.5cm} \text{ provided } \ell' \geq \ell^{\ast}_{q-1}(B)
            \end{equation}
            \end{small}
            \textbf{(iii)} \ \textbf{\underline{Inductive step} : } Let $m = m_{q-n+1}$ and $\vec{m} = (m_{1}, \cdots, m_{q-n+1}).$ Then, upon tensoring the short-exact sequence (\ref{kernel_bundle}) for $\mathcal{F} = mB$ with \\ \small{$M_{B}[\vec{m}-m \vec{e}_{q-n+1}](K_{X}+\ell B),$} and taking long-exact sequence of cohomology one obtains \\
            \small{
                $$H^{0}(mB) \otimes H^{0}(M_{\vec{m}B}[q-n](K_{X}+\ell B)) \overset{\mu_{m, q-n}}{\xrightarrow{\hspace{1cm}}} H^{0}(M_{\vec{m}B}[q-n](K_{X}+(\ell+m) B))$$ }
                \begin{equation}
                    \rightarrow H^{1}(M_{\vec{m}B}(K_{X}+\ell B)) \longrightarrow H^{0}(mB) \otimes H^{1}(M_{\vec{m}B}[q-n](K_{X}+\ell B))
                \end{equation}
            Since $\ell \geq \ell^{\ast}_{q}(B) \geq \ell^{\ast}_{q-1}(B),$ by induction hypothesis, \small{$H^{1}(M_{\vec{m}B}[q-n](K_{X}+\ell B)) = 0$} \\ and therefore, the vanishing $H^{1}(M_{\vec{m}B}(K_{X}+\ell B)) = 0$ is equivalent to the surjectivity of $\mu_{m, q-n}$ above. This surjectivity follows because from Lemma \ref{regularity_of_Syzygy_Bundle_multiple_line_bundles} one has 
            $$\mathrm{reg}_{B}(M_{\vec{m}B}[q-n](K_{X}+w B)) \leq q-n \hspace{0.1cm} \text{ for } w \geq \mathrm{reg}_{B}(K_{X}).$$
            Hence, the proof of the vanishing is complete. \\
            For the result on Properties$-(m_{q})$ we put $m_{i} = \ell$ for all $1 \leq i \leq q-n+1$ and replacing $\ell$ by $(n-1)\ell$ For property$-(M_{q})$ use the fact that $L = \ell B$ is projectively normal for all $\ell \geq \rho$ which follows from Lemma \ref{Castelnuovo_lemma}.
    \end{proof}
    Now we are ready to prove the Theorem in 
\textit{Proof of Theorem \ref{CM_theorem}.} For any ample and base point free line bundle $B,$ on a smooth projective variety $X$ of dimension $n,$ the Castelnuovo Mumford regularity of the canonical line bundle $K_{X}$ with respect to $B$ is $\mathrm{reg}_{B}(K_{X}) \leq n+1.$ In other words, $K_{X}+\ell B$ is $0-$regular for all $\ell \geq \mathrm{reg}(K_{X})$. In particular, $K_{X}+\ell B$ is $0-$regular for all $\ell \geq n+1.$ \\
    The proof follows from Theorem \ref{general_vanishing_multiple_line_bundles} by putting $m_{i} = \ell$ for all $1 \leq i \leq q-n+1$ and replacing $\ell$ by $(n-1)\ell.$
\subsection{Property-$(M_{2})$ of multiple line bundles on surfaces} $\text{ }$ \\
\label{M2}
We finish this section with a result on surjectivity of certain multiplication maps on all surfaces which improves our bound of Theorem \ref{CM_theorem} for Property-$(M_{2})$ in presence of more sections of the line bundle $B$ using the method of hyperplane cuts.
\begin{theorem}
\label{surj_mult_surf_maps}
Let $X$ be a surface. Let $B$ be an ample and base point free line bundle on $X$ and let $L$ be a line bundle on $X$ such that either $(L \cdot B) > 2(B^{2})$ or $(L \cdot B) = 2(B^{2})$ with $h^{0}(B) \geq 4.$ Then, \\
The following map is surjective
\begin{equation}
H^{0}(B) \otimes H^{0}(K_{X}+L)  \xrightarrow[\hspace{1cm}]{\alpha} H^{0}(K_{X}+L+B)
\end{equation}
Moreover, one has the following vanishing
\begin{equation}
\label{kernel_vanishing_for_M_2}
H^{1}(M_{mB} \otimes (K_{X}+L)) = 0 \text{ for integers } m \geq 1 
\end{equation}
In particular, the line bundle $L = \ell B$ satisfies property-$(m_{2})$ for all integers $\ell \geq 3$ or $\ell \geq 2, h^{0}(B) \geq 4.$ Moreover, it satisfies Property-$(M_{2})$ if $\ell \geq \max(\rho, 3)$ or $\ell \geq \max(\rho, 2), h^{0}(B) \geq 4.$
\end{theorem}
\begin{proof}
Since $B$ is base point free, by Bertini's theorem there is a smooth integral curve $C$ in $|B|,$ and let $W$ be the global sections of $B$ that do not vanish on $C;$ i.e. $W := H^{0}(B)/H^{0}(\mathcal{O}_{X}).$ Then, by the following commutative diagram
\small{
$$
\begin{tikzcd}[row sep = small, column sep = small]
H^{0}(K_{X}+ L) \otimes H^{0}(\mathcal{O}_{X}) \ar[hook]{r} \ar{d}{ } & H^{0}(K_{X}+ L) \otimes H^{0}(B) \ar[twoheadrightarrow]{r} \ar{d}{\alpha} & H^{0}(K_{X}+ L) \otimes W \ar{d}{ } \\
H^{0}(K_{X}+ L)  \ar[hook]{r} & H^{0}(K_{X}+L+B) \ar[twoheadrightarrow]{r} & H^{0} \left( (K_{X}+L+B) \otimes \mathcal{O}_{C} \right)
\end{tikzcd}$$
}
where the top horizontal row is a short-exact sequence by construction and the bottom horizontal row is a short-exact sequence by Kodaira vanishing theorem, the left-vertical map is trivially surjective trivially and $\alpha$ is surjective if and only if the right-vertical map is surjective and by base point free pencil trick this happens if
\begin{equation}
\label{basepointfreepenciltrick1}
h^{1}( \mathcal{O}_{C}(K_{X}+L-B)) \leq \mathrm{dim}(W)-2 = h^{0}(B)-3
\end{equation}
but
$h^{1}( \mathcal{O}_{C}(K_{X}+L-B)) = h^{1}[ K_{C}+((L-2B) \otimes \mathcal{O}_{C})) ] = 0,$ when $(L-2B \cdot B) = \mathrm{deg}((L-2B) \otimes \mathcal{O}_{C}) > 0,$
by Kodaira vanishing theorem and the right-hand-side of the inequality-(\ref{basepointfreepenciltrick1}) is $\geq 0,$ since $B$ is ample and base point free imples that the map $\phi_{B} : X \rightarrow |B|$ defined by $|B
|$ is a finite map onto its image implies $h^{0}(B) \geq 3$. (Here, We could just use the fact that $W$ is base point free on $C$ and hence $\mathrm{dim}(W) \geq 2.$) If $((L-2B) \cdot B) = 0,$ we need $h^{0}(B) \geq 4.$ This proves the surjectivity for $\alpha.$ \\
Now the surjectivity of $\alpha$ and the Observation \ref{[Observation-(1.2)]{BG00}} mentioned above, one has that the map
$$H^{0}(m B) \otimes H^{0}(K_{X}+L) \rightarrow H^{0}(K_{X}+L+mB)$$
is surjective for all $m \geq 1$ and for $((L-2B) \cdot B) > 0$ or $((L-2B) \cdot B) = 0$ with $h^{0}(B) \geq 4.$ This surjectivity is equivalent to (\ref{kernel_vanishing_for_M_2}) from Kodaira vanishing theorem and the long-exact sequence of cohomology of the short exact sequence (\ref{kernel_bundle}) for $m B.$ \\
The particular case is straightforward when we put $L = \ell B.$
\end{proof}
\section{Syzygies of varieties with nef canonical bundle}
\label{nef_canonical}
In this section, we study syzygies of adjoint linear series on smooth projective varieties with nef canonical bundle. Throughout this section, we fix the notation $\mathcal{L}_{\ell}$ for the adjoint linear series $\mathcal{L}_{\ell} = K_{X}+\ell B$ for $X$ a smooth projective variety with nef $K_{X}$ and $B$ an ample and base point free line bundle on $X$ for any positive integer $\ell.$ But first we show the following Lemma which will be used in the proof of the Theorems \ref{theorem_$(M_q)$nef_canonical_surface} and \ref{nef_general_theorem}.
\begin{lemma} 
\label{Lemma_nef_canonical}
Let $X$ be a smooth projective variety. Then the following is true.
    \begin{itemize}
    \item[(a)] For any ample line bundle $A$ on $X$ with $dA-K_{X}$ nef, one has $H^{i}(aA+N) = 0$ for any integer $a \geq d+1$ and any nef line bundle $N$ on $X.$
    \item[(b)] Suppose $X$ has nef canonical bundle $K_{X}$ and $d \geq 1$ be a rational number. For any ample and base point free line bundle $B$ on $X$ and $\Lambda = K_{X}+(n-1)B$ one has $(\Lambda \cdot B^{n-1} ) \geq 2, \text{ and } (B^{n}) \geq 2$ and for any smooth curve $C$ in $|B|$ has genus $g \geq 2.$
    \item[(c)] Suppose $X$ is a variety with nef canonical bundle $K_{X},$ $d \geq 1$ be a rational number and $B$ be an ample and base point free line bundle on $X$ such that $d(B^{n}) \geq (K_{X} \cdot B^{n-1}).$ If $H^{i}((m+n-2)B) = H^{i}((m+n-2)B_{X^{(1)}}) = \cdots = H^{i}((m+n-2)B_{Y}) =  0$ for some integers $m \geq d+1$ and $i \geq 1$ then $H^{i}(\ell B_{X^{(c)}}) = 0$ for all $\ell \geq m+n-1$ and all $0 \leq c \leq n-1.$
    \end{itemize}
\end{lemma}
\begin{proof}
First, we give a proof of (a). Clearly, $H^{i}(aA+N) = H^{i}(K_{X}+(a-d)A+(dA-K_{X})+N)=0$ by Kawamata-Viehweg vanishing for any $i \geq 1, a \geq d+1.$ This completes the proof. \par
     Since $B$ is base point free in (b), (c) and (d), as mentioned in (\ref{hyperplane_stratification}) by Bertini's theorem, we can choose the hyperplane stratification by sections of $B$
    $$C = X^{(n-1)} \subset Y = X^{(n-2)} \subset \cdots \subset X^{(1)} \subset X^{(0)} = X$$
    where $C$ is a smooth curve and $Y$ is a smooth surface. \par
    Now, we give a proof of (b). By adjunction formula we have $K_{C} = (K_{X}+(n-1)B)|_{C}$ is ample since $B$ is ample and $K_{X}$ is nef. Therefore, $g \geq 2$ which implies $(B^{n-1} \cdot \Lambda) = 2g-2 \geq 2.$ Let $L = B_{C}$ and $m = h^{0}(D)-1.$ If $(B^{n}) = 1,$ then the linear system $|L|$ is base point free and since $(B^{n}) = 1,$ $|L|$ gives a degree one map $\phi_{L} : C \rightarrow \mathbb{P}^{m},$ whose image $C'$ is also a degree one curve in $\mathbb{P}^{m}$. Thus, both $C'$ and $C$ are isomorphic and rational which contradicts our result $g \geq 2.$ \par
    Finally, we prove (c). Let $D = \ell B_{C}.$ Then $$\mathrm{deg}(D) = (\ell-n+1)(B^{n})+(n-1)(B^{n}) \geq m(B^{n})+(n-1)(B^{n})$$ $$ \geq (B^{n})+d(B^{n})+(n-1)(B^{n})  > (K_{X} \cdot B^{n-1})+(n-1)(B^{n}) = (\Lambda \cdot B^{n-1}) \geq 2g-2$$
    since $\ell \geq m+n-1.$ Therefore, $\mathrm{deg}(K_{C}-D) < 0$ and thus $H^{1}(D) = H^{0}(K_{C}-D) =  0.$ Now, tensoring the short-exact sequence
    $$0 \rightarrow \mathcal{O}_{X^{(c-1)}}(-B) \rightarrow \mathcal{O}_{X^{(c-1)}} \rightarrow \mathcal{O}_{X^{(c)}} \rightarrow 0$$
    with $\ell B$ and taking the long-exact sequence of cohomology, we get
    \begin{equation}
    \label{cohomology_induction_nef_surface}
        \cdots \rightarrow H^{i}((\ell-1)B_{X^{(k-1)}}) \rightarrow H^{i}(\ell B_{X^{(k-1)}}) \rightarrow H^{i}(\ell B_{X^{(k)}}) \rightarrow \cdots
    \end{equation}
    for all $k \geq 1.$ We take the  vanishing $H^{i}((\ell-1)B_{X^{(k-1)}}) = 0$ as our induction hypothesis where the first step in this hypothesis $H^{i}((m+n-2)B_{X^{(k-1)}}) = 0$ is included in the assumption of our lemma. Thus (c) is true from (\ref{cohomology_induction_nef_surface}) by our induction hypothesis and since $H^{j}(\ell B_{C}) = H^{j}(D) = 0$ for any $j \geq 1.$
\end{proof}
\subsection{Surfaces with nef canonical bundle}
In this section, we study conditions for Property$-(M_{q})$ of adjoint line bundles on surfaces with nef canonical bundle. \\
\subsubsection{General results}
We start by proving a theorem on Property-$(M_{2})$ for adjoint linear series on surfaces with nef canonical bundle.
\begin{theorem}
\label{$(M_{2})$surf_nef_bundle}
    Let $X$ be a surface with nef canonical bundle. Let $B$ be an ample and base point free line bundle on $X$ such that the line bundle $\Lambda = K_{X}+B$ is also base point free. Let $\ell, m, k$ be integers and $d \geq 1$ be a rational number. Then, for any nef line bundle $N$ on $X,$ the line bundles $\mathcal{L}_{\ell} = K_{X}+\ell B, \mathcal{L}_{\ell'} = K_{X}+\ell'B$ satisfy the vanishing
    \begin{equation}
    \label{$(M_2)$nef_canonical_surface}
        H^{1}(M_{\mathcal{L}_{\ell}} \otimes (K_{X}+m\mathcal{L}_{\ell'}+N)) = 0 \hspace{0.5cm} \text{ for integer } m \geq 1
    \end{equation}
    provided $m\ell'+\ell \geq d+4$ when $(dB-K_{X} \cdot \Lambda) \geq 0$ or $m\ell'+\ell \geq d+3$ if $(dB-K_{X} \cdot B) > 0$ when $(dB-K_{X} \cdot \Lambda) > 0$ when $(dB-K_{X} \cdot \Lambda) \geq 0$ with $h^{0}(\Lambda) \geq 4.$ \\
    In particular, $\mathcal{L}_{\ell}$ satisfies Property$-(m_{2})$ if $\ell \geq \dfrac{d+3}{2}$ when $(dB-K_{X} \cdot \Lambda) > 0$ or if $\ell \geq \dfrac{d+4}{2}$ or when $(dB-K_{X} \cdot \Lambda) \geq 0.$ \\
    Moreover, $L$ satisfies Property-$(M_{2})$ if
    $\ell \geq \max \left( 3, \dfrac{d+3}{2} \right)$ when $(dB-K_{X} \cdot \Lambda) > 0$ or if  $\ell \geq \max \left( 3, \dfrac{d+4}{2} \right)$ when $(dB-K_{X} \cdot \Lambda) \geq 0.$
\end{theorem}
\begin{proof}
    The vanishing (\ref{$(M_2)$nef_canonical_surface}) is equivalent to the surjectivity
    $$H^{0}(\mathcal{L}_{\ell}) \otimes H^{0}(K_{X}+m\mathcal{L}_{\ell'}+N) \xrightarrow[\hspace{0.6cm}]{\alpha} H^{0}(K_{X}+m\mathcal{L}_{\ell'}+\mathcal{L}_{\ell}+N)$$
    By Observation \ref{[Observation-(1.2)]{BG00}} to prove the surjectivity of $\alpha$ it is enough to prove the surjectivity of
    \begin{equation}
        H^{0}(B) \otimes H^{0}(K_{X}+m\mathcal{L}_{\ell'}+N) \xrightarrow[\hspace{0.6cm}]{\beta} H^{0}(K_{X}+m\mathcal{L}_{\ell'}+B+N)
    \end{equation}
    and
    \begin{equation}
        H^{0}(\Lambda) \otimes H^{0}(K_{X}+m\mathcal{L}_{\ell'}+(\ell-1)B+N) \xrightarrow[\hspace{0.6cm}]{\gamma} H^{0}(K_{X}+m\mathcal{L}_{\ell'}+\mathcal{L}_{\ell}+N)
    \end{equation}
    The surjectivity of the map $\beta$ follows from Theorem \ref{surj_mult_surf_maps}. Now we prove the surjectivity of the map $\gamma.$ For this, let $G' = K_{X}+m\mathcal{L}_{\ell'}+(\ell-1)B+N$ we choose a smooth curve $C'$ in $|K_{X}+B|$ which exists by Bertini's theorem. Then, from the diagram
    \begin{small}
    $$
        \begin{tikzcd}[row sep = large, column sep = large]
             H^{0}(G') \otimes H^{0}(\mathcal{O}_{X}) \arrow[r, hook] \arrow[d, twoheadrightarrow] &  H^{0}(G') \otimes H^{0}(\Lambda) \arrow[r, twoheadrightarrow] \arrow[d, "\gamma"] & H^{0}(G') \otimes W' \arrow[d, "\gamma' "] \\
             H^{0}(G') \arrow[r, hook] & H^{0}(G'+\Lambda) \arrow[r, twoheadrightarrow] & H^{0}((G'+\Lambda)|_{C'})
        \end{tikzcd}
    $$
    \end{small}
    Since the left vertical map is surjective, $\gamma$ is surjective if and only if the map $\gamma'$ is surjective. Finally, by the base point free pencil trick $\gamma'$ is surjective if
    $$h^{1}((G'-\Lambda)_{C'}) \leq \mathrm{dim}(W')-2 = h^{0}(\Lambda)-3$$
    Since $B$ is ample and $K_{X}$ is nef, $\Lambda = K_{X}+B$ is also ample. Since $\Lambda$ is ample and base point free, $h^{0}(\Lambda) \geq 3.$ Also, we suppose that $N' = (m-1)K_{X}+N.$ \par
    Now, $\mathrm{deg}((G'-\Lambda)_{C'}) = (K_{X}+(m\ell'+\ell-2)B+N' \cdot \Lambda)$
    \begin{small}
    $$ = (K_{X}+\Lambda \cdot \Lambda)+(dB-K_{X} \cdot \Lambda)+(\ell+m\ell'-3-d)(B \cdot \Lambda)+(N' \cdot \Lambda)$$ 
    $$= 2g-2+(dB-K_{X} \cdot \Lambda)+(\ell+m\ell'-3-d)(B \cdot \Lambda)+(N' \cdot \Lambda)$$
    \end{small}
    By our assumption on $B,$ $\mathrm{deg}((G-\Lambda)_{C'}) > 0$ if $\ell+m\ell' \geq d+3$ when $(dB-K_{X} \cdot B) > 0$ or if $\ell+m\ell' \geq 4+d$ and $(dB-K_{X} \cdot B) \geq 0.$ This proves the result for Property-$(m_{2}).$
    By \cite{BL25}, $K_{X}+\ell B$ is projectively normal for $\ell \geq 3.$ This completes the proof by showing the result for Property-$(M_{2})$.
\end{proof}
Next, we show conditions for property-$(M_{q})$ of adjoint linear series on these surfaces for $q \geq 3$ in Theorem \ref{theorem_$(M_q)$nef_canonical_surface}.
\begin{theorem}
\label{theorem_$(M_q)$nef_canonical_surface}
    Let $X$ be a surface with nef canonical bundle. Let $d \geq 1$ be a rational number and $B$ be an ample and base point free line bundle on $X$ such that $dB-K_{X}$ is nef. Let $\ell, k$ be integers and the line bundle $\Lambda = K_{X}+B$ is base point free. Then, for any nef line bundle $N$ on $X,$ the line bundles $\mathcal{L}_{\ell} = K_{X}+\ell B$ and $\mathcal{L}_{\ell'} = K_{X}+\ell' B$ satisfy the vanishing
    \begin{equation}
    \label{$(M_q)$nef_canonical_surface}
        H^{1}(M^{\otimes (k+1)}_{\mathcal{L}_{\ell}} \otimes (K_{X}+m\mathcal{L}_{\ell'}+N)) = 0 \hspace{0.5cm} \text{ for } 0 \leq k \leq q-2 \text{ and all integers } m \geq 1
    \end{equation}
    provided $\ell$ and $\ell'$ are integers satisfying
    \begin{small}
    \begin{equation}
    \label{nef_surface_conditions}
        \ell \geq d +2, \ell' \geq 2+\floor[\bigg]{\dfrac{2q+1}{(B^{2})}}
    \end{equation}
    \end{small}
    Moreover, $\mathcal{L}_{\ell}$ satisfies Property$-(M_{q})$ provided
    \begin{equation}
        \ell \geq  \max \left(d+2, 2+\floor[\bigg]{\dfrac{2q+1}{(B^{2})}} \right)
    \end{equation}
    and therefore $\mathcal{L}_{\ell}$ satisfies Property$-(M_{q})$ for all $\ell \geq \max (d+2, q+2).$ \\
    In particular, if  $d = 1$ or equivalently if $B-K_{X}$ is nef, then $\mathcal{L}_{\ell}$ satisfies Property$-(M_{q})$ for all $\ell \geq q+2.$
\end{theorem}
\begin{proof}
    To show this vanishing it is enough to prove the surjectivity
    \begin{equation}
        H^{0}(\mathcal{L}_{\ell}) \otimes H^{0}(M^{\otimes k}_{\mathcal{L}_{\ell}} \otimes (K_{X}+m\mathcal{L}_{\ell'}+N)) \rightarrow H^{0}(M^{\otimes k}_{\mathcal{L}_{\ell}} \otimes (K_{X}+m\mathcal{L}_{\ell'}+\mathcal{L}_{\ell}+N)), 
    \end{equation}
    $$\hspace{0.1cm} \text{ for }  0 \leq k \leq q-2 \text{ and for all } m \geq 1$$
    We proceed by induction on $k$. By Bertini's theorem, there are smooth curves $C$ and $C'$ of genus $g$ and $g'$ in the linear series $|B|$ and $|K_{X}+B|$ respectively. Since $B$ and $\Lambda = K_{X}+B$ are ample and base point free and $K_{X}$ is nef, as we mention in Lemma \ref{Lemma_nef_canonical} (b) both $K_{C}$ and $K_{C'}$ are ample and $g, g', (B \cdot \Lambda) = 2g-2, (B^{2})$ all are $ \geq 2.$ \\
    We now take the following vanishing as our first induction hypothesis
    \begin{equation}
    \label{induction_hypothesis_surface_nef_K}
        H^{1}(M^{\otimes (k'+1)}_{\mathcal{L}_{\ell}} \otimes (K_{X}+m\mathcal{L}_{\ell'}+P)) = 0 \hspace{0.5cm} \text{ for } 0 \leq k' \leq k-1 \text{ and all integers } m \geq 1
    \end{equation}
    where the beginning $k' = 0$ in (\ref{induction_hypothesis_surface_nef_K}) follows from Theorem \ref{$(M_{2})$surf_nef_bundle}. 
    By Observation \ref{[Observation-(1.2)]{BG00}} to prove the surjectivity of $\alpha$ it is enough to show the surjectivity of
    $$
        H^{0}(B) \otimes H^{0}(M^{\otimes k}_{\mathcal{L}_{\ell}} \otimes (K_{X}+m\mathcal{L}_{\ell'}+N)) \overset{\beta}{\longrightarrow} H^{0}(M^{\otimes k}_{\mathcal{L}_{\ell}} \otimes (K_{X}+m\mathcal{L}_{\ell'}+B+N))
    $$
    and 
    $$
       H^{0}(\Lambda) \otimes H^{0}(M^{\otimes k}_{\mathcal{L}_{\ell}} \otimes (K_{X}+m\mathcal{L}_{\ell'}+(\ell-1)B+N)  \overset{\gamma}{\longrightarrow} H^{0}(M^{\otimes k}_{\mathcal{L}_{\ell}} \otimes (K_{X}+m\mathcal{L}_{\ell'}+\Lambda+(\ell-1)B+N))
    $$
    Our first step is to show the surjectivity of the map $\beta.$ \\
    \underline{\textbf{Step-1}} \hspace{0.1cm} 
    By our induction hypothesis on $k$, 
    $$H^{1}(M^{\otimes k}_{\mathcal{L}_{\ell}} \otimes (K_{X}+m\mathcal{L}_{\ell'}+N-B)) = H^{1}(M^{\otimes k}_{\mathcal{L}_{\ell}} \otimes (K_{X}+\mathcal{L}_{m\ell'-1}+N') = 0,$$ 
    for $0 \leq k \leq q-2$,  where $\mathcal{L}_{m\ell'-1} = K_{X}+(m\ell'-1)B$ is ample and base point free and $N' = (m-1)K_{X}+N$ is nef, since $\ell' \geq 2+\floor[\bigg]{\dfrac{2q+1}{(B^{2})}}$ one has $$m\ell'-1 \geq \ell'-1 \geq \left(2+\floor[\bigg]{\dfrac{2q+1}{(B^{2})}} \right)-1 \geq \left(2+\floor[\bigg]{\dfrac{2q+1}{(B^{2})}-\dfrac{2}{(B^{2})}} \right) = \left( 2+\floor[\bigg]{\dfrac{2(q-1)+1}{(B^{2})}} \right)$$
    the beginning $k = 0$ is true by Theorem \ref{$(M_2)$nef_canonical_surface} since
    \begin{equation}
    \label{nef_surface_condition_1}
        \ell+(m\ell'-1) \geq d+4
    \end{equation}
    Therefore, by Observation \ref{[Observation-(1.3)]{BG00}}, $\beta$ is surjective if and only if its restriction
    $$
        H^{0}(B_{C}) \otimes H^{0}(M^{\otimes k}_{\mathcal{L}_{\ell}} \otimes (K_{X}+m\mathcal{L}_{\ell'}+N) |_{C}) \overset{\beta'}{\longrightarrow} H^{0}(M^{\otimes k}_{\mathcal{L}_{\ell}} \otimes (K_{X}+m\mathcal{L}_{\ell'}+B+N)|_{C})
    $$
    is surjective, which by the application $(K_{X}+m\mathcal{L}_{\ell'}+N)|_{C} = K_{C}+(\mathcal{L}_{m\ell'-1})_{C}+N'_{C}$ of adjunction $K_{C} = (K_{X}+B)_{C}$ is
    $$
        H^{0}(B_{C}) \otimes H^{0}(M^{\otimes k}_{\mathcal{L}_{\ell}} \otimes  (K_{C}+(\mathcal{L}_{m\ell'-1})_{C}+N'_{C})) \overset{\beta'}{\longrightarrow} H^{0}(M^{\otimes k}_{\mathcal{L}_{\ell}} \otimes (K_{C}+(\mathcal{L}_{m\ell'-1})_{C}+B_{C}+N'_{C}))
    $$
    and since $\ell \geq 2,$ by Kodaira vanishing $H^{1}(\mathcal{L}_{\ell}-B) = H^{1}(K_{X}+(\ell-1)B) = 0,$ $\beta'$ is surjective by Lemma \ref{BG99_Lemma_(2.9)} if
    $$
        H^{0}(B_{C}) \otimes H^{0}(M^{\otimes k}_{(\mathcal{L}_{\ell})_{C}} \otimes  (K_{C}+(\mathcal{L}_{m\ell'-1})_{C}+N'_{C}))  \overset{\beta^{(1)}}{\longrightarrow} H^{0}(M^{\otimes k}_{(\mathcal{L}_{\ell})_{C}} \otimes (K_{C}+(\mathcal{L}_{m\ell'-1})_{C}+B_{C}+N'_{C})))
    $$
    is surjective. Let $E = B_{C}, F = M^{\otimes k}_{(\mathcal{L}_{\ell})_{C}} \otimes  (K_{C}+(\mathcal{L}_{m\ell'-1})_{C}+N'_{C}).$ By adjunction formula $K_{C} = (K_{X}+B)_{C},$ $\mathrm{deg}((\mathcal{L}_{\ell})_{C}) = (K_{X}+B+(\ell-1)B \cdot B) = 2g-2+(\ell-1)(B^{2})$ and $ \mathrm{deg}(K_{C}+(\mathcal{L}_{m\ell'-1})_{C}+N'_{C}) = 2g-2+(m\ell'-1)(B^{2})+(K_{X} \cdot B)+(N' \cdot C).$ Since $\ell \geq d+2 \geq 3,$ we have $(\ell-1)(B^{2}) \geq 2$ and $M_{L_{C}}$ is semistable. Moreover, by Proposition \ref{Proposition-(2.2)_Bu94} the map $\beta^{(1)}$ is surjective since $E = B_{C}$ is globally generated and
    \begin{align}
        \mu(F) > 2g \\
        \mu(F) > 2g+\mathrm{rank}(E)(2g-\mu(E))-2h^{1}(E)
    \end{align}
    Indeed,
    $$
         \mu(F) \geq -2k+2g-2+(m\ell'-1)(B^{2})+(K_{X} \cdot B)+(N' \cdot C)$$ 
         $$= 2g+(m\ell'-2)(B^{2})-2(k+1)+(2g-2)+(N' \cdot C) > 2g$$
      and
      $$
         \mu(F) \geq 2g+(m\ell'-2)(B^{2})-2(k+1)+(2g-2)+(N' \cdot C)
    $$
      $$> 2g+\mathrm{rank}(E)(2g-\mu(E))-2h^{1}(E) = 2g+(2g-(B^{2}))-2h^{1}(B_{C}) $$
     since 
     \begin{equation}
     \label{nef_surface_condition_2}
         m\ell' \geq 1+\ceil[\bigg]{\dfrac{2q+1}{(B^{2})}} 
     \end{equation}
    $2g-2 \geq 2$ and $0 \leq k \leq q-2$. This proves the surjectivity of $\beta.$ Now we prove the surjectivity of the map $\gamma$ in our next step. \\ 
    \underline{\textbf{Step-2}} \hspace{0.1cm}
    For this step, let $w = m\ell'+\ell-d-2.$ By our induction hypothesis on $k$, 
    $$H^{1}((M^{\otimes k}_{\mathcal{L}_{\ell}} \otimes (K_{X}+m\mathcal{L}_{\ell'}+(\ell-1)B+N-\Lambda)) = H^{1}(M^{\otimes k}_{\mathcal{L}_{\ell}} \otimes (K_{X}+\mathcal{L}_{w}+N''))= 0,$$
    since
    $$\ell+w \geq d+3 \text{ when } (dB-K_{X} \cdot \Lambda) > 0 \text{ or } \ell+w \geq d+4 \text{ otherwise, }$$ which is equivalent to
    \begin{equation}
    \label{nef_surface_condition_3}
        m\ell'+2\ell \geq 2d+5 \text{ if } (dB-K_{X} \cdot \Lambda) > 0 \text{ and } m\ell'+2\ell \geq 2d+6 \text{ otherwise}
    \end{equation}
    where $N'' = (dB-K_{X})+N'$ is nef since both $N', dB-K_{X}$ are nef. \\
    Therefore, $\gamma$ is surjective if and only if its restriction
    $$
        H^{0}(\Lambda_{C'}) \otimes H^{0}(M^{\otimes k}_{\mathcal{L}_{\ell}} \otimes (K_{X}+m\mathcal{L}_{\ell'}+(\ell-1)B+N) |_{C'}) \overset{\gamma'}{\longrightarrow} H^{0}(M^{\otimes k}_{\mathcal{L}_{\ell}} \otimes (K_{X}+m\mathcal{L}_{\ell'}+(\ell-1)B+\Lambda+N)|_{C'})
    $$
    is surjective, which by the application $(K_{X}+m\mathcal{L}_{\ell'}+(\ell-1)B+N)|_{C'} = K_{C'}+(\mathcal{L}_{w})_{C'}+N''_{C'}$ of adjunction $K_{C'} = (K_{X}+\Lambda)_{C'}$ is
    $$
        H^{0}(\Lambda_{C'}) \otimes H^{0}(M^{\otimes k}_{\mathcal{L}_{\ell}} \otimes  (K_{C'}+(\mathcal{L}_{w})_{C'}+N''_{C'})) \xrightarrow[\hspace{0.5cm}]{\gamma'} H^{0}(M^{\otimes k}_{\mathcal{L}_{\ell}} \otimes  (K_{C'}+(\mathcal{L}_{w})_{C'}+\Lambda_{C'}+N''_{C'}))
    $$
    Now, taking $A = B, a = \ell-1$ in Lemma \ref{Lemma_nef_canonical} (a) we have $H^{1}(\mathcal{L}_{\ell}-\Lambda) = H^{1}((\ell-1)B) = 0$ as 
    \begin{equation}
    \label{nef_surface_condition_4}
        \ell \geq d+2.
    \end{equation}
    Thus, using Lemma \ref{BG99_Lemma_(2.9)} we conclude that $\gamma'$ is surjective  if
    \begin{equation}
         H^{0}(\Lambda_{C'}) \otimes H^{0}(M^{\otimes k}_{(\mathcal{L}_{\ell})_{C'}} \otimes  (K_{C'}+(\mathcal{L}_{w})_{C'}+N''_{C'})) \xrightarrow[\hspace{0.5cm}]{\gamma^{(1)}} H^{0}(M^{\otimes k}_{(\mathcal{L}_{\ell})_{C'}} \otimes (K_{C'}+(\mathcal{L}_{w})_{C'}+\Lambda_{C'}+N''_{C'}))
    \end{equation}
    is surjective. \\
    Let $E' = \Lambda_{C'}, F' = M^{\otimes k}_{(\mathcal{L}_{\ell})_{C'}} \otimes  (K_{C'}+(\mathcal{L}_{w})_{C'}+N''_{C'}).$ By adjunction formula $K_{C'} = (K_{X}+\Lambda)_{C'},$ 
    $$\mathrm{deg}((\mathcal{L}_{\ell})_{C'}) = 2g'-2+(\ell-d-1)(B \cdot \Lambda)+(dB-K_{X} \cdot \Lambda) \geq 2g'$$ 
    since $\ell \geq d+2$ and $(B \cdot \Lambda) \geq 2.$ \\
    Therefore $M_{(\mathcal{L}_{\ell})_{C'}}$ is semistable and 
    $$ \mathrm{deg}(K_{C'}+(\mathcal{L}_{w})_{C'}+N''_{C'}) = (2g'-2)+w(B \cdot \Lambda)+(K_{X} \cdot \Lambda)+(N'' \cdot \Lambda).$$
    Moreover, by Proposition \ref{Proposition-(2.2)_Bu94} the map $\gamma^{(1)}$ is surjective since $E' = \Lambda_{C'}$ is globally generated and
    \begin{align}
        \mu(F') > 2g' \\
        \mu(F') > 2g'+\mathrm{rank}(E')(2g'-\mu(E'))-2h^{1}(E')
    \end{align}
    Indeed,
    $$
         \mu(F') \geq -2k+(2g'-2)+w(B \cdot \Lambda)+(K_{X} \cdot \Lambda)+(N'' \cdot \Lambda) > 2g' $$
      since $w(B \cdot \Lambda) \geq (2q-2)+1 > 2k+2$ and
      $$
         \mu(F') \geq -2k+(2g'-2)+w(B \cdot \Lambda)+(K_{X} \cdot \Lambda)+(N'' \cdot \Lambda) > 2g'+\mathrm{rank}(E')(2g-\mu(E'))-2h^{1}(E')
    $$
      $$= 2g'+(2g'-2-(\Lambda^{2}))+2-2h^{1}(\Lambda_{C'}) =  2g'+(K_{X} \cdot \Lambda)+2-2h^{1}(\Lambda_{C'})$$
      since 
      \begin{equation}
      \label{nef_surface_condition_5}
          w(B \cdot \Lambda) \geq 2q+1.
      \end{equation}
      This proves the surjectivity of $\gamma.$ The conditions are obtained by comparing (\ref{nef_surface_condition_1}), (\ref{nef_surface_condition_2}), (\ref{nef_surface_condition_3}), (\ref{nef_surface_condition_4}) and (\ref{nef_surface_condition_5}). \\ 
      The condition for Property-$(M_{q})$ follows by taking $\ell' = \ell, m = 1.$ Since $(B^{2}) \geq 2,$ $\ell \geq 2+\floor[\bigg]{\dfrac{2q+1}{2}} = q+2 \geq 2+\floor[\bigg]{\dfrac{2q+1}{(B^{2})}}.$ Therefore, one can use the larger bound $\ell \geq q+2$ instead of $\ell \geq 2+\floor[\bigg]{\dfrac{2q+1}{(B^{2})}}.$  The line bundle $\mathcal{L}_{\ell}$ is normally generated for $\ell \geq 3$ and therefore we get Property-$(M_{q})$ for the mentioned bounds.
\end{proof}
We mention a corollary of the above theorem which we will use in Section \ref{Section_Varieties_nef_bundle} as a base step for induction.
\begin{corollary}
    For any ample and base point free line bundle $B$ on a regular surface $X$ with nef canonical bundle with $\Lambda = K_{X}+B$ base point free, $d \geq 1$ a rational number and $dB-K_{X}$ nef. Let $q, m$ be integers satisfying $q \geq 2, m \geq 1.$
    Then, for integers $\ell, \ell', k$ the following are true.
    \begin{itemize}
    \item[\textbf{(1)}] One has the cohomology vanishing
    \begin{equation}
        H^{1}(M^{\otimes (k+1)}_{\mathcal{L}_{\ell}} \otimes (K_{X}+m\mathcal{L}_{\ell'}+N)) = 0, \hspace{0.5cm} \text{ for } 0 \leq k \leq q-2
    \end{equation}
    provided $\ell \geq d+2, m\ell' \geq 2+\floor[\bigg]{\dfrac{2q+1}{(B^{2})}}$
        \item[\textbf{(2)}] The multiplication map
    $$
        H^{0}(B) \otimes H^{0}(M^{\otimes k}_{\mathcal{L}_{\ell}} \otimes (K_{X}+m\mathcal{L}_{\ell'}+N)) \xrightarrow[]{\beta} H^{0}(M^{\otimes k}_{\mathcal{L}_{\ell}} \otimes (K_{X}+m\mathcal{L}_{\ell'}+B+N))
    $$
    is surjective for $0 \leq k \leq q-2$ provided $\ell \geq d+2, m\ell' \geq 2+\floor[\bigg]{\dfrac{2q+1}{(B^{2})}}$
    \item[\textbf{(3)}] The multiplication map
    \begin{small}
    $$
     H^{0}(K_{X}+B) \otimes H^{0}(M^{\otimes k}_{\mathcal{L}_{\ell}} \otimes (K_{X}+m\mathcal{L}_{\ell'}+(\ell-\lambda)B+N)) \xrightarrow[]{\gamma} H^{0}(M^{\otimes k}_{\mathcal{L}_{\ell}} \otimes (K_{X}+m\mathcal{L}_{\ell'}+\mathcal{L}_{\ell-\lambda+1}+N))
    $$
    \end{small}
    is surjective for $0 \leq k \leq q-2$ provided $\ell \geq d+2, m\ell'+\ell \geq \lambda+2+\floor[\bigg]{\dfrac{2q+1}{(B^{2})}}$
 \end{itemize}
\end{corollary}
\begin{proof}
    Setting $N' = (m-1)K_{X}+N$ we get $K_{X}+m\mathcal{L}_{\ell'}+N = K_{X}+\mathcal{L}_{m\ell'}+N'$ and $K_{X}+m\mathcal{L}_{\ell'}+(\ell-\lambda)B+N = K_{X}+\mathcal{L}_{m\ell'+\ell-\lambda}+N.'$ Now we apply the statement the Theorem \ref{theorem_$(M_q)$nef_canonical_surface} above with $\ell'$ replaced with $m\ell'$ and argue similar to the proof of Theorem \ref{theorem_$(M_q)$nef_canonical_surface} to get (1) and (2). To get (3) we use the statement of Theorem with $\ell'$ replaced with $m\ell'+\ell-\lambda.$
\end{proof}
\subsubsection{K3-surfaces}
\label{K3}
Let $L$ be a line bundle on a $K3-$surface $X$. Since $K_{X} \sim 0,$ by Koszul duality one has $K_{r-q, 1}(X, L) = K_{q-2, 2}(X, L)$ and therefore property-$(M_{q})$ is the same as property-$(N_{q-2}).$ In this case, the section ring $R(L)$ is Gorenstein and therefore the Betti table has this symmetry. \\

\subsubsection{Enriques surfaces}
\label{Enriques}
In this section, by an Enriques surface we mean a regular ($h^{1}(\mathcal{O}_{X}) = 0$) surface whose canonical divisor $K_{X}$ is numerically trivial. Thus the results of this section are valid for $K3-$surfaces as well. \\
Let $L$ be a line bundle on an Enriques surface $X.$ Since $K_{X} \equiv 0$ the section ring $R(L)$ is $\mathbb{Q}-$Gorenstein but not Gorenstein in general. \\
However, one gets the following variant of \cite[Theorem 2.2]{BG99} which follows along the same lines as the proof of \textit{Gallego} and \textit{Purnaprajna} as one replaces $L$ by $K_{X}+L$ or $L'$ by $K_{X}+L'.$ The very ampleness follows since they prove that $L$ is projectively normal. Because their technique is dependent on intersection numbers and $K_{X}$ is numerically trivial, $K_{X}$ has no effect in the final intersection number computations and the proof works in the presence of $K_{X}$ as well. Similarly, all results of \cite[Section 2]{BG99} have a corresponding variant for Property-$(M_{q}).$ \\
We mention the results here without proof.
\begin{theorem} (compare with \cite[Theorem 2.2]{BG99})
\label{(M_2)Enriques}
    Let $X$ be an Enriques surface. Let $B_{1}, B'_{1}, B_{2},$ and $ B'_{2}$ be ample and base point free line bundles on $X$ such that $B_{i} \equiv B'_{i}$ for $i = 1, 2$ and the following condition $(\ref{enriques_condition})$ is satisfied
    \begin{equation}
    \label{enriques_condition}
         \overset{- \ - \ -}{\hspace{2cm}}\begin{cases}
        \text{ either } & (B_{1} \cdot B_{2}) \geq 4, (B^{2}_{i}) \geq 6 \text{ for } i = 1, 2 \\
        \text{ or } & (B_{1} \cdot B_{2}) \geq 5
    \end{cases}
    \end{equation}
    Let $r, s, k, l$ be integers satisfying $r, s, k \geq 1$ and $l \geq 0$. Let $L = rB_{1}+sB_{2}$ and $L' = kB'_{1}+lB'_{2}.$ Then, the multiplication maps
    \begin{eqnarray}
        H^{0}(L) \otimes H^{0}(K_{X}+L') \xrightarrow{\alpha} H^{0}(K_{X}+L+L') \\
        H^{0}(K_{X}+L) \otimes H^{0}(L') \xrightarrow{\beta} H^{0}(K_{X}+L+L')
    \end{eqnarray}
    are surjective and one has the following vanishings
    \begin{equation}
        H^{1}(M_{L} \otimes (K_{X}+L')) = H^{1}(M_{L'} \otimes (K_{X}+L)) = 0
    \end{equation}
    In particular, $L$ is very ample and satisfies Property-$(M_{2}).$
\end{theorem}
One now gets the following corollary which follows from \cite[Lemma 2.7]{BG99} and Theorem \ref{(M_2)Enriques}.
\begin{corollary}
    Let $X$ be an Enriques surface and $A_{i}$ be an ample line bundle on $X$ for all integers $1 \leq i \leq t.$ Let $L = K_{X}+A_{1}+ \cdots+A_{t}.$ If $t \geq 4,$ then $L$ satisfies Property-$(M_{2}).$
\end{corollary}
Similarly one has the following variant of \cite[Theorem 2.11]{BG99}.
\begin{theorem} (compare with \cite[Theorem 2.11]{BG99})
    Let $X$ be an Enriques surface. Let $B_{i} \equiv B'_{i}$ be ample and base point free divisors on $X$ for $i = 1, 2,$ with $(B_{1} \cdot B_{2}) \geq 6.$ Let $L = rB_{1}+sB_{2}$ and $L' = kB'_{1}+lB'_{2}.$ If $k, l, r, s \geq 1,$ then
    \begin{equation}
        H^{1}(M^{\otimes 2}_{L} \otimes (K_{X}+L')) = 0
    \end{equation}
    In particular, $L$ satisfies Property-$(M_{3}).$
\end{theorem}
The following corollary is a variant of \cite[Corollary 2.12]{BG99}
\begin{corollary}(compare with \cite[Corollary 2.12]{BG99}) 
    Let $A, A_{i}$ be ample line bundles on an Enriques surface $X$ for $i = 1, 2, \cdots, t$. Let $L = K_{X}+A_{1}+ \cdots +A_{t}$ and $L' = K_{X}+\ell A.$ Then $L$ and $L'$ satisfy Property-($M_{3}$) for $t \geq 5$ and $\ell \geq 4$ respectively.
\end{corollary}
The following theorem is a variant of \cite[Theorem 2.14]{BG99}. 
\begin{theorem}
    Suppose $X$ is an Enriques surface. Let $B$ be an ample and base point free line bundle on $X$ with $(B^{2}) \geq 6$ and let $N, M$ be line bundles numerically equivalent to zero (i.e. they are trivial or equal to $K_{X}$). Let $L = \ell B+N$ and $L' = \ell'B+M$ for integers $\ell$ and $\ell'.$ Then, one the vanishing
    \begin{equation}
        H^{1}(M^{\otimes k+1}_{L} \otimes (K_{X}+L')) = 0 \hspace{1cm} \text{ for } 0\leq k \leq q-2 \text{ and } \ell, \ell' \geq q-1
    \end{equation}
    In particular, $L$ satisfies Property-($M_{q}$).
\end{theorem}
\begin{remark}
    In fact, the above Theorem follows from \cite[Theorem 2.14]{BG99} if one takes $N' = K_{X}+M$ in the notation of \textit{Gallego} and \textit{Purnaprajna}. Notice that with the extra condition $(B^{2}) \geq 6$ one is able to improve the bound $\ell \geq q+1$ given by Theorem \ref{CM_theorem}.
\end{remark}
Now, we close this section by stating a variant of \cite[Corollary 2.15]{BG99}
\begin{corollary}
    Suppose $X$ be an Enriques surface, let $A$ be an ample line bundle and let $B$ be an ample and base point free line bundle on $X.$ Then, $L = K_{X}+\ell B$ satisfies Property-$(M_{q})$ and $L' = K_{X}+\ell' A$ satisfies Property-$(M_{q})$ for $\ell \geq  \geq \text{ and } \ell' \geq 2q-2.$
\end{corollary}
\subsubsection{Abelian and Bielliptic surfaces}
\label{Abelian_Bielliptic}
\begin{theorem}
    Let $B$ be an ample and base point free line bundle on an Abelian or Bielliptic surface satisfying $(B^{2}) \geq 5.$ Let $N$ be a numerically trivial line bundle on $X.$ If $L \equiv \ell B$ and $L' \equiv \ell' B.$ Then, for one has
    \begin{equation}
        H^{1}(M^{\otimes (k+1)}_{L} \otimes (K_{X}+L')) = H^{1}(M^{\otimes (k+1)}_{L'} \otimes (K_{X}+L)), \hspace{0.2cm} \text{ for } 0 \leq k \leq q-2 \text{ and } \ell, \ell' \geq q-1 \geq 1
    \end{equation}
    In particular, for $L \equiv K_{X}+\ell B$ and $L' \equiv K_{X}+\ell' B$ one has
    \begin{equation}
        H^{1}(M^{\otimes (k+1)}_{L} \otimes (K_{X}+L')) = H^{1}(M^{\otimes (k+1)}_{L'} \otimes (K_{X}+L)) = 0, \hspace{0.2cm} \text{ for } 0 \leq k \leq q-2 \text{ and } \ell, \ell' \geq q-1 \geq 1
    \end{equation}
\end{theorem}
\begin{corollary}
    Let $A$ be an ample line bundle on an Abelian or Bielliptic surface $X.$ Then $L = K_{X}+A_{1}+\cdots+A_{t}$ satisfies Property-$(M_{q})$ for $\ell \geq 2q-2.$
\end{corollary}
\subsection{Varieties with nef canonical bundle}
\label{Section_Varieties_nef_bundle}
Finally, we end this section by generalizing of Theorem \ref{theorem_$(M_q)$nef_canonical_surface} in dimension $n \geq 3$ in the following Theorem \ref{nef_general_theorem}. For this we use the method of hyperplane cuts, but before stating the theorem, we need the following lemma for the inheritance of cohomology vanishing through the floors down to the surface for the hyperplane stratification:
\begin{equation}
\label{hyperplane_strata_K_nef}
        Y := X^{(n-2)} \subset \cdots \subset X^{(c)} \subset \cdots \subset X^{(1)} \subset X^{(0)} = X \hspace{0.5cm} \text{ for integers } 0 \leq c \leq n-2
    \end{equation}
    till codimension $(n-2)$ of $X$ via hyperplane cuts from $|B|.$
\begin{lemma}
\label{inheritance_cohomology_vanishing}
    If $X$ is any smooth projective variety, then for any integer $1 \leq c \leq n-2$ and any $1 \leq i < n-c,$ one has
    \begin{equation}
        h^{i}(\mathcal{O}_{X}) = 0 \implies h^{i}(\mathcal{O}_{X^{(k)}}) = 0, \text{ for all } 1 \leq k \leq c.
    \end{equation}
\end{lemma}
\begin{proof}
    The proof follows by taking the long-exact sequence of cohomology of the short-exact sequence
    \begin{equation}
        0 \rightarrow \mathcal{O}_{X^{(k-1)}}(-B) \rightarrow  \mathcal{O}_{X^{(k-1)}} \rightarrow  \mathcal{O}_{X^{(k)}} \rightarrow 0
    \end{equation}
    and applying $H^{i+1}(\mathcal{O}_{X^{(k-1)}}(-B)) = 0$ inductively in the range of $i, k$ and $c.$ The latter follows from the Kodaira vanishing theorem.
\end{proof}
\begin{theorem}
\label{nef_general_theorem}
    Let $X$ be a smooth projective variety with nef canonical bundle of dimension $n$ with $H^{1}(\mathcal{O}_{X}) = 0$. Let $d \geq 1$ be a rational number and $B$ be an ample and base point free line bundle on $X$ such that $dB-K_{X}$ is nef. Let $\ell, \ell', k$ be integers and assume that line bundle $\Lambda = K_{X}+(n-1) B$ is base point free. Then, for any nef line bundle $N$ and $\mathcal{L}_{\ell} = K_{X}+\ell B, \ \mathcal{L}_{\ell'} = K_{X}+\ell'B$ one has the vanishing
    \begin{equation}
    \label{vanishing_for_nef}
        H^{1}(M^{\otimes (k+1)}_{\mathcal{L}_{\ell}} \otimes (K_{X}+m \mathcal{L}_{\ell'}+N)) = 0, \hspace{1cm} \text{ for }  0 \leq k \leq q-n \text{ and for all } m \geq 1
    \end{equation}
    provided $\ell$ and $\ell'$ are integers satisfying
    \begin{equation}
    \label{nef_condition}
        \ell \geq d+n, \ell' \geq 2+\floor[\bigg]{\dfrac{2q-2n+5}{(B^{n})}}
    \end{equation}
    In particular, the vanishing (\ref{vanishing_for_nef}) is true for
    \begin{equation}
        \ell \geq d+n, \ell' \geq q-n+4
    \end{equation}
\end{theorem}
\begin{proof}
    We use induction for this proof. We prove the vanishing
    \begin{equation}
    \label{vanishing_nef_-1}
        H^{1}(M^{\otimes k}_{\mathcal{L}_{\ell}} \otimes (K_{X}+m\mathcal{L}_{\ell'}+N)) = 0, \hspace{1cm} \text{ for }  0 \leq k \leq q-n+1 \text{ and for all } m \geq 1
    \end{equation}
    because the vanishing (\ref{vanishing_nef_-1}) is just the Kodaira vanishing for $k = 0$ and (\ref{vanishing_for_nef}) can be obtained from (\ref{vanishing_nef_-1}) if one replaces the power $k$ by $k+1$ when $1 \leq k \leq q-n+1.$ This helps us with the inductive argument. For the induction, we take the hyperplane stratification (\ref{hyperplane_strata_K_nef}). We proceed by forward induction on $k$ and backward induction on $c.$ \par
    Fixing some $1 \leq k \leq q-n$ we now take the following vanishing as our first induction hypothesis
    \begin{equation}
    \label{first_induction_hypothesis}
        H^{1}(M^{\otimes k^{(c)}}_{(\mathcal{L}_{\ell})_{X^{(c)}}} \otimes (K_{X^{(c)}}+m^{(c)}(\mathcal{L}_{\ell'})_{X^{(c)}}+N^{(c)})) = 0 
    \end{equation} 
  $\text{ and for all integers }  m^{(c)} \geq 1,$ for all  $0 \leq k^{(c)} \leq k+1$ when $ 1 \leq c \leq n-3,$ for all $0 \leq k^{(0)} < k+1$ when $0 \leq c \leq n-3 $ and any nef line bundle $N^{(c)}$ . \par
  We begin the induction on c from $c = n-2.$ We treat this case first. Since $K_{X}$ is nef, from Remark \ref{adjunction_hyperplane_strata}, $X^{(c)}$ is a variety of general type for all $1 \leq c \leq n-2.$ Moreover, since $h^{1}(\mathcal{O}_{X}) = 0,$ from Lemma \ref{inheritance_cohomology_vanishing} one has $h^{1}(\mathcal{O}_{X^{(c)}}) = 0$ for all $0 \leq c \leq n-2.$ Thus, $Y = X^{(n-2)}$ is a surface of general type and by adjunction, $(dB-K_{X})|_{Y} = (d+(n-2))B_{Y}-K_{Y}$ and therefore one can apply Theorem \ref{theorem_$(M_q)$nef_canonical_surface} on $(Y, B_{Y})$ to obtain the vanishing (\ref{first_induction_hypothesis}) for $c = n-2, 0 \leq k^{(c)} = k^{(n-2)} \leq k+1,$ which is the case of surfaces with nef canonical bundle. Indeed, by our hypothesis,
  \begin{equation}
  \label{beginning_step_induction_c=n-2}
      \ell \geq d_{Y}+2, \ell' \geq 2+\floor[\bigg]{\dfrac{2q_{Y}+1}{(B^{2}_{Y})}}
  \end{equation}
  where $q_{Y}$ is defined by $q_{Y}-2 = q-n$ and $d_{Y} = d+(n-2).$ Thus, (\ref{first_induction_hypothesis}) follows by replacing $d$ by $d_{Y},$ $q$ by $q_{Y},$ $B$ by $B_{Y}$ and so on in Theorem \ref{theorem_$(M_q)$nef_canonical_surface}. The condition (\ref{beginning_step_induction_c=n-2}) gives the condition (\ref{nef_condition}) in the statement of our theorem \par
  Now, we are ready to prove the vanishing
  \begin{equation}
        H^{1}(M^{\otimes (k+1)}_{\mathcal{L}_{\ell}} \otimes (K_{X}+m \mathcal{L}_{\ell'}+N)) = 0 \hspace{0.5cm} \text{ for any integer }  m \geq 1,
    \end{equation} 
  To show this vanishing it is enough to prove the surjectivity
    $$
        H^{0}(M^{\otimes k}_{\mathcal{L}_{\ell}} \otimes (K_{X}+m \mathcal{L}_{\ell'}+N)) \otimes H^{0}(\mathcal{L}_{\ell}) \xrightarrow[]{\alpha} H^{0}(M^{\otimes k}_{\mathcal{L}_{\ell}} \otimes (K_{X}+m\mathcal{L}_{\ell'}+N)), \hspace{0.2cm} \text{ for } m \geq 1
    $$
    By Observation \ref{[Observation-(1.2)]{BG00}}, to prove the surjectivity of $\alpha$ it is enough to show the surjectivity of
    $$
        H^{0}(M^{\otimes k}_{\mathcal{L}_{\ell}} \otimes (K_{X}+m\mathcal{L}_{\ell'}+N)) \otimes H^{0}(B) \overset{\beta}{\longrightarrow} H^{0}(M^{\otimes k}_{\mathcal{L}_{\ell}} \otimes (K_{X}+m\mathcal{L}_{\ell'}+B+N))
    $$
    and 
    $$
        H^{0}(M^{\otimes k}_{\mathcal{L}_{\ell}} \otimes (K_{X}+m\mathcal{L}_{\ell'}+(\ell-n+1)B+N)) \otimes H^{0}(\Lambda) \overset{\gamma}{\longrightarrow} H^{0}(M^{\otimes k}_{\mathcal{L}_{\ell}} \otimes (K_{X}+m\mathcal{L}_{\ell'}+\mathcal{L}_{\ell}+N))
    $$
    Our first step is to show the surjectivity of the map $\beta.$ \par
    \underline{\textbf{Step-1}} \hspace{0.1cm} 
    For this, taking $\mathcal{L}^{(c)}_{\ell} = K_{X^{(c)}}+\ell B_{X^{(c)}}, P^{(c)} = K_{X}+m^{(c)}\mathcal{L}_{\ell'}+N^{(c)}, P = P^{(0)}$ we assume the surjectivity the following vanishing as our second induction hypothesis
    \begin{small}
    \begin{equation}
        H^{0}(M^{\otimes k}_{(\mathcal{L}_{\ell})_{X^{(c)}}} \otimes  P^{(c)}) \otimes H^{0}(B_{X^{(c)}}) \overset{\beta^{(c)}}{\longrightarrow} H^{0}(M^{\otimes k}_{(\mathcal{L}_{\ell})_{X^{(c)}}} \otimes (P^{(c)}+B_{X^{(c)}}))
    \end{equation}
    \end{small}
    %and
    %\begin{equation}
        %H^{0}(M^{\otimes k}_{L_{X^{(c)}}} \otimes  (K_{X^{(c)}}+(m-c)L_{X^{(c)}}+P) \otimes H^{0}(\Lambda_{X^{(c)}}) \overset{\alpha''}{\longrightarrow} H^{0}(M^{\otimes k}_{L_{X^{(c)}}} \otimes (K_{X^{(c)}}+(m-c)L_{X^{(1)}}+B_{X^{(c)}}+P))
    %\end{equation}
    for any $m^{(c)} \geq 1, m^{(0)} = m$ provided $1 \leq c \leq n-3,$ and for any nef line bundle $N^{(c)}$ on $X^{(c)}$ with $N^{(0)} = N.$  We have already proved this vanishing for $c = n-2,$ and now using backward induction on $c$ we show it for $c = 0$. \\
    By our first induction hypothesis on $k$, 
    $$H^{1}(M^{\otimes k}_{\mathcal{L}_{\ell}} \otimes(P-B)) = H^{1}(M^{\otimes k}_{\mathcal{L}_{\ell}} \otimes (K_{X}+\mathcal{L}_{m\ell-1}+N') = 0,$$ 
    where $N' = (m-1)K_{X}+N$ is nef since all of the line bundles $K_{X}, B, N$ and $\ell > 1$ and $m \geq 1$. Therefore, by Observation \ref{[Observation-(1.3)]{BG00}}, $\beta$ is surjective if and only if its restriction
    \begin{small}
    $$
        H^{0}(M^{\otimes k}_{\mathcal{L}_{\ell}} \otimes P |_{X^{(1)}}) \otimes H^{0}(B_{X^{(1)}}) \overset{\beta'}{\longrightarrow} H^{0}(M^{\otimes k}_{\mathcal{L}_{\ell}} \otimes (P+B)|_{X^{(1)}})
    $$
    \end{small}
    is surjective, which by adjunction $P|_{X^{(1)}} = K_{X^{(1)}}+\mathcal{L}^{(1)}_{m\ell'-2}+N'_{X^{(1)}}$ is
    \begin{small}
    $$
        H^{0}(M^{\otimes k}_{\mathcal{L}_{\ell}} \otimes  (K_{X^{(1)}}+(\mathcal{L}_{m\ell'-1})_{X^{(1)}}+N'_{X^{(1)}})) \otimes H^{0}(B_{X^{(1)}}) \overset{\beta'}{\longrightarrow} H^{0}(M^{\otimes k}_{\mathcal{L}_{\ell}} \otimes (K_{X^{(1)}}+(\mathcal{L}_{m\ell'-1})_{X^{(1)}}+B_{X^{(1)}}+N'_{X^{(1)}}))
    $$
    \end{small}
    and since $\ell \geq 2,$ by Kodaira vanishing $H^{1}(\mathcal{L}_{\ell}-B) = H^{1}(K_{X}+(\ell-1)B) = 0,$ $\beta'$ is surjective, by Lemma \ref{BG99_Lemma_(2.9)} if
    \begin{small}
    $$
        H^{0}(M^{\otimes k}_{\mathcal{L}^{(1)}_{\ell-1}} \otimes  (K_{X^{(1)}}+\mathcal{L}^{(1)}_{m\ell'-2(1)}+N'_{X^{(1)}})) \otimes H^{0}(B_{X^{(1)}}) \xrightarrow[\hspace{0.5cm}]{\beta^{(1)}} H^{0}(M^{\otimes k}_{\mathcal{L}^{(1)}_{\ell}} \otimes (K_{X^{(1)}}+\mathcal{L}^{(1)}_{m\ell'-2(1)}+B_{X^{(1)}}+N'_{X^{(1)}}))
    $$
    \end{small}
    is surjective. Repeating the above argument, we get that $\beta^{(1)}$ is surjective if
    \begin{small}
    $$
        H^{0}(M^{\otimes k}_{\mathcal{L}^{(c)}_{\ell-c}} \otimes  (K_{X^{(c)}}+\mathcal{L}^{(c)}_{m\ell'-2(c)}+N'_{X^{(c)}})) \otimes H^{0}(B_{X^{(c)}}) \xrightarrow[\hspace{0.5cm}]{\beta^{(c)}} H^{0}(M^{\otimes k}_{\mathcal{L}^{(c)}_{\ell-c}} \otimes (K_{X^{(c)}}+\mathcal{L}^{(c)}_{m\ell'-2(c)}+B_{X^{(c)}}+N'_{X^{(c)}}))
    $$
    \end{small}
    is surjective for some $1 \leq c \leq n-2.$ The latter follows from our second induction hypothesis on $c.$ This proves the surjectivity of $\beta.$ Now we prove the surjectivity of the map $\beta$ in our next step. \\ 
    \underline{\textbf{Step-2}} \hspace{0.1cm} For this, taking $Q^{(c)} = K_{X^{(c)}}+m^{(c)}L_{X^{(c)}}+(\ell-n+1))B_{X^{(c)}}+N^{(c)}, Q = Q^{(0)}$ and $w = m\ell'+\ell-d-2(n-1)$ we assume the surjectivity the following vanishing as our third induction hypothesis
    \begin{equation}
        H^{0}(M^{\otimes k}_{(\mathcal{L}_{\ell})_{X^{(c)}}} \otimes  Q^{(c)}) \otimes H^{0}(\Lambda_{X^{(c)}}) \overset{\gamma^{(c)}}{\longrightarrow} H^{0}(M^{\otimes k}_{(\mathcal{L}_{\ell})_{X^{(c)}}} \otimes (K_{X^{(c)}}+\mathcal{L}_{\ell}+m^{(c)}(\mathcal{L}_{\ell'})_{X^{(c)}}+N^{(c)}))
    \end{equation}
    for any $m^{(c)} \geq 1, m^{(0)} = m$ provided $1 \leq c \leq n-3,$ and for any nef line bundle $N^{(c)}$ on $X^{(c)}$ with $N^{(0)} = N.$  We have already proved this vanishing for $c = n-2$ which is the case of surfaces and now using backward induction on $c$ we show it for $c = 0$. \\
    By our first induction hypothesis on $k$, 
    $$H^{1}(M^{\otimes k}_{\mathcal{L}_{\ell}} \otimes (Q-\Lambda)) = H^{1}(M^{\otimes k}_{\mathcal{L}_{\ell}} \otimes (K_{X}+\mathcal{L}_{w}+N'') = 0,$$ 
    where $N'' = (dB-K_{X})+N'$ is nef since all of the line bundles $K_{X}, B, N$ are nef and
   \begin{equation}
       \ell+w \geq d+4
   \end{equation}
   as $$m\ell'+2\ell \geq \ell'+2\ell \geq \left(2+\floor[\bigg]{\dfrac{2q-2n+5}{(B^{n})}} \right)+2(d+n) \geq 2d+2n+2$$
    Therefore, $\gamma$ is surjective if and only if its restriction
    \begin{small}
    $$
        H^{0}(M^{\otimes k}_{\mathcal{L}_{\ell}} \otimes Q |_{X^{(1)}}) \otimes H^{0}(\Lambda_{X^{(1)}}) \overset{\gamma'}{\longrightarrow} H^{0}(M^{\otimes k}_{\mathcal{L}_{\ell}} \otimes (K_{X}+m\mathcal{L}_{\ell'}+\mathcal{L}_{\ell}+N)|_{X^{(1)}})
    $$
    \end{small}
    which by adjuction $(K_{X}+m\mathcal{L}_{\ell}+N)|_{X^{(1)}} = K_{X^{(1)}}+(\mathcal{L}_{w})_{X^{(1)}}+N''_{X^{(1)}}$ is
    \begin{equation}
        H^{0}(M^{\otimes k}_{\mathcal{L}_{\ell}} \otimes  (K_{X^{(1)}}+(\mathcal{L}_{w})_{X^{(1)}}+N''_{X^{(1)}})) \otimes H^{0}(\Lambda_{X^{(1)}})
    \end{equation}
    $$
        \xrightarrow[\hspace{0.5cm}]{\gamma^{(1)}} H^{0}(M^{\otimes k}_{\mathcal{L}_{\ell}} \otimes (K_{X^{(1)}}+(\mathcal{L}_{w})_{X^{(1)}}+\Lambda_{X^{(1)}}+N''_{X^{(1)}}))
    $$
    and since $H^{1}(\mathcal{L}_{\ell}-B) = H^{1}(K_{X}+(\ell-1)B+N) = 0,$ $\gamma'$ is surjective, by Lemma-(\ref{BG99_Lemma_(2.9)}) (\cite[Lemma-(2.9)]{BG00}) above, if
    \begin{equation}
        H^{0}(M^{\otimes k}_{(\mathcal{L_{\ell}})_{X^{(1)}}} \otimes  (K_{X^{(1)}}+(\mathcal{L}_{w})_{X^{(1)}}+N''_{X^{(1)}})) \otimes H^{0}(\Lambda_{X^{(1)}})
    \end{equation}
    $$
        \xrightarrow[\hspace{0.5cm}]{\gamma^{(1)}} H^{0}(M^{\otimes k}_{(\mathcal{L_{\ell}})_{X^{(1)}}} \otimes ((K_{X^{(1)}}+(\mathcal{L}_{w})_{X^{(1)}}+\Lambda_{X^{(1)}}+N''_{X^{(1)}}))
    $$
    is surjective. Which follows from our second induction hypothesis on $c.$ This proves the surjectivity of $\gamma.$
\end{proof}
\begin{corollary}
\label{vanishing_for_general_m}
    Let $X$ be a smooth projective variety with nef canonical bundle of dimension $n$ with $H^{1}(\mathcal{O}_{X}) = 0$. Let $d \geq 1$ be a rational number and $B$ be an ample and base point free line bundle on $X$ such that $dB-K_{X}$ is nef. Let $\ell, \ell', m, k$ be integers and assume that line bundle $\Lambda = K_{X}+(n-1) B$ is base point free. Then, for any nef line bundle $N$ and $\mathcal{L}_{\ell} = K_{X}+\ell B, \ \mathcal{L}_{\ell'} = K_{X}+\ell'B$ one has the vanishing
    \begin{equation}
    \label{vanishing_for_nef_corollary}
        H^{1}(M^{\otimes (k+1)}_{\mathcal{L}_{\ell}} \otimes (K_{X}+m \mathcal{L}_{\ell'}+N)) = 0, \hspace{1cm} \text{ for }  0 \leq k \leq q-n
    \end{equation}
    provided $\ell$ and $\ell'$ are integers satisfying
    \begin{equation}
        \ell \geq d+n, m\ell' \geq 2+\floor[\bigg]{\dfrac{2q-2n+5}{(B^{n})}}
    \end{equation}
    In particular, the vanishing (\ref{vanishing_for_nef_corollary}) is true for
    \begin{equation}
    \label{conditions_for_vanishing_of_nef}
        \ell \geq d+n, m\ell' \geq q-n+4
    \end{equation}
\end{corollary}
\begin{proof}
    Setting $N' = (m-1)K_{X}+N$ we get $K_{X}+m\mathcal{L}_{\ell'}+N = K_{X}+\mathcal{L}_{m\ell'}+N'.$ Now we apply the statement the Theorem \ref{nef_general_theorem} above with $\ell'$ replaced with $m\ell'$ to show the vanishing $H^{1}(M^{\otimes (k+1)}_{\mathcal{L}_{\ell}} \otimes (K_{X}+\mathcal{L}_{m\ell'}+N'))$ for $0 \leq k \leq q-n.$ This proves (\ref{vanishing_for_nef_corollary}). Now, (\ref{conditions_for_vanishing_of_nef}) follows since $(B^{n}) \geq 2$ implies $2+\floor[\bigg]{\dfrac{2q-2n+5}{2}} = q-n+4 \geq 2+\floor[\bigg]{\dfrac{2q-2n+5}{(B^{n})}}.$
\end{proof}
\begin{corollary}
    Let $X$ be a smooth projective variety with nef canonical bundle of dimension $n$ with $H^{1}(\mathcal{O}_{X}) = 0$. Let $d \geq 1$ be a rational number and $B$ be an ample and base point free line bundle on $X$ such that $dB-K_{X}$ is nef. Let $\ell, k$ be integers and assume that line bundle $\Lambda = K_{X}+(n-1) B$ is base point free. Then, the line bundle $\mathcal{L}_{\ell} = K_{X}+\ell B$ satsifies Property$-(M_{q})$ provided 
    \begin{equation}
        \ell \geq d+n, (n-1)\ell \geq 2+\floor[\bigg]{\dfrac{2q-2n+5}{(B^{n})}}
    \end{equation}
    Moreover, $\mathcal{L}_{\ell}$ satisfies Property$-(M_{q})$ for $\ell \geq \max \left(d+n, \dfrac{q-n+4}{n-1} \right) = \max \left(d+n, \dfrac{q+3}{n-1}-1 \right).$ \\
    In particular, $\mathcal{L}_{\ell}$ satisfies Property$-(M_{q})$ for $\ell \geq \max \left(d+n, \ell^{\mathrm{ceil}}_{q} \right)$ when $n \geq 3$ and for $\ell \geq \max(d+n, q+2)$ when $n = 2,$ where $\ell^{\mathrm{floor}}_{q} = \floor[\bigg]{\dfrac{q+1}{n-1}}.$
    \end{corollary}
    \begin{proof}
        This Corollary follows immediately by taking $\ell = \ell'$ and $m = n-1$ in Corollary \ref{vanishing_for_nef_corollary}. Since by \cite[Theorem 1.2]{BL25}, $\mathcal{L}_{\ell}$ is projectively normal for $\ell \geq n+1,$ and in our case $\ell \geq d+n \geq n+1,$ we get Property$-(M_{q}).$ Our last bound follows since 
        $$\ell^{\mathrm{floor}}_{q} = \floor[\bigg]{\dfrac{q+1}{n-1}} = \ceil[\bigg]{\dfrac{q+3}{n-1}}-1 = \ceil[\bigg]{\dfrac{q+1}{n-1}-\left(1-\dfrac{2}{n-1} \right)}$$
    \end{proof}
\section{Varieties with Kodaira dimension $-\infty$}\label{kodaira_negative_infinity}
\subsection{Rational surfaces}
We use this to prove the following theorem for rational surfaces which should be thought of a generalization of the \textit{Gonality Conjecture} to such surfaces. But before we can state the theorem we make the following definitions.
\begin{definition} \textbf{[Gonality of curves and divisors on surfaces]} \\
    Let $C$ be a smooth projective curve of genus $g.$ We define the \textbf{gonality} of $C$ denoted by $\mathrm{gon}(C)$ to be the quantity
    \begin{small}
    \begin{equation}
        \mathrm{gon}(C) = \min \left\{\ d \ |  \text{ a map }  C \xrightarrow[]{d \ :\ 1} \mathbb{P}^{1} \text{ exists } \ \right\} = \max \left\{\ k \ | \ K_{C} \text{ is } (k+2)-\text{very ample } \ \right\}
    \end{equation}
    \end{small}
    Let $X$ be a smooth projective surface and let $L$ be a line bundle on $X.$ Then, we define the following quantities:
    \begin{align}
        \mathrm{gon}_{\max}(L) = \max \left\{ \ \mathrm{gon}(C) \ | \ C \text{ is a smooth curve in } |L| \ \right\} \\
        \mathrm{gon}_{\min}(L) = \min \left\{ \ \mathrm{gon}(C) \ | \ C \text{ is a smooth curve in } |L| \ \right\} \\
        \mathrm{gon}_{\mathrm{gen}}(L) = \left\{ \ \mathrm{gon}(C) \ | \ C \text{ is a smooth general curve in } |L| \ \right\}
    \end{align}
    We also adapt the convention that all of these quantities are zero if the linear system $|L|$ is not effective or has no smooth curve in it. By Bertini's theorem, all of these quantities are non-zero if $L$ is base point free. All of these quantities are same and equal to $d-1$ for $(X, L) = (\mathbb{P}^{2}, dH)$ where $H$ is a line in $\mathbb{P}^{2}$. By \textit{G. Martens}, these quantities are equal to $a$ for $(X, L) = (\mathbb{F}_{e}, aC_{0}+bf).$ By work of \textit{M. Green} and \textit{R. Lazarsfeld} these three quantities are same on any $K3-$surface except for the \textit{Donagi-Morrison} $K3-$surface and these quantities can be different on Abelian surfaces.
\end{definition}
\begin{theorem}
\label{$(M_q)$_rational_surfaces}
Let $X$ be a rational surface and let $L$ be an ample and base point free line bundle on $X$. Let $|L|$ be a linear system containing curves $C$ of genus $g \geq 1$ and $q$ be an integer satisfying $q \geq 2 = \mathrm{dim}(X)$. Then, 
\begin{itemize}
\item[\textbf{(1)} ] If $(-K_{X} \cdot L) \geq q+2,$ then $L$ satisfies property-$(M_{q})$ if and only if $q \leq \mathrm{gon}_{\max}(L).$
\item[\textbf{(2)} ] If $(-K_{X} \cdot L) = q+1,$ then $L$ does not satisfy property-$(M_{q})$ if and only if one of the following holds
\begin{itemize}
    \item[\textbf{(i)} ] $q \geq \mathrm{gon}_{\max}(L),$ or 
    \item[\textbf{(ii)} ] $-K_{X} \otimes \mathcal{O}_{C}$ is very ample, $h^{0}(-K_{X} \otimes \mathcal{O}_{C}) = 3,$ that is if $-K_{X} \otimes \mathcal{O}_{C}$ embeds $C$ as a plane curve of degree $q+1 \geq 4.$
\end{itemize}
\item[\textbf{(3)} ] For $q \geq 3.$ If $(-K_{X} \cdot L) = q,$ then $L$ does not satisfy property-$(M_{q})$ if and only if one of the following holds
\begin{itemize}
    \item[\textbf{(i)} ] $q \geq \mathrm{gon}_{\max}(L),$ or 
    \item[\textbf{(ii)} ] $h^{0}(-K_{X} \otimes \mathcal{O}_{C}) = 2,$ that is if $-K_{X} \otimes \mathcal{O}_{C}$ is a $g^{1}_{q}$ on $C.$
\end{itemize}
\end{itemize}
\end{theorem}
\begin{proof}
Since in any case $(-K_{X} \cdot L) \geq 3,$ by \cite[Theorem 1.3]{BG01} $L$ is very ample and the embedding $\phi_{L} : X \hookrightarrow \mathbb{P}^{r}$ is projectively normal. Now, by \cite[Lemma 1.2]{BG01} we have $h^{1}(d L) = 0$ for $d > 0,$ since $X$ is regular $h^{1}(dL) = 0$ for $d = 0$ and $h^{1}(dL) = 0$ for $d < 0$ by Kodaira vanishing theorem. Therefore, by the discussion above, $\phi_{L}$ is arithmetically Cohen-Macaulay and therefore for any smooth curve $C$ in $|L|$ we have that $L$ satisfies Property-$(M_{q})$ if and only if $L_{C}$ satisfies Property-$(M_{q-1}).$ Thus, we reduce the problem to vanishing of weight one syzygies on the curve $C$. \\
Let $N = -K_{X} \otimes \mathcal{O}_{C},$ be the line bundle on $C$ of degree $(-K_{X} \cdot L)$ then by adjunction formula $L_{C} = K_{C}+N$ and therefore $\mathrm{deg}(L_{C}) = 2g-2+(-K_{X} \cdot L)$.
\\
First we prove the statement for $q = 2.$ One needs to show that $L_{C}$ satisfies Property-$(M_1).$ But $\mathrm{deg}(L_{C}) = 2g-2+3 = 2g+1$ and therefore by \cite[(3.6)]{GL86} $L_{C}$ it satisfies Property-$(M_{1})$ if and only if $g \geq 1$ which is equivalent to saying that $(L^{2}) > (-K_{X} \cdot L).$ \\
Let $p = q-2.$ Now $q \leq \mathrm{gon}(C)$ if and only if $p \leq \mathrm{gon}(C)-2$ and this is equlivalent to saying $K_{C}$ is $p-$very ample. \\
For \textbf{(1)} $\mathrm{deg}(N) \geq q+2 = p+4$ implies $\mathrm{deg}(L_{C}) = 2g-2+p+4 = 2g+p+2$ in which case by \cite[Theorem 5.2 (1)]{NP24} one has $$K_{r_{C}-(q-1)}(C, L) = K_{r_{C}-1-p, (1+1)-1}(C, L) = K_{p, 1}(C, K_{C}, L) = 0.$$ Similarly, for \textbf{(2)} $\mathrm{deg}(N) \geq q+1 = p+3$ implies $\mathrm{deg}(L_{C}) = 2g+p+1$ in which case by \cite[Theorem 5.2 (2) and Corollary 5.3]{NP24} one has $K_{r_{C}-(q-1)}(C, L) =  K_{p, 1}(C, K_{C}, L) \neq 0$ if and only if either $K_{C}$ is not $p-$very ample which is equivalent to $q \geq \mathrm{gon}(C)$ or for $q \geq 3$ when $N$ embeds $C$ as a plane curve of degree $p+3 \geq 4.$ Finally, for \textbf{(3)} $\mathrm{deg}(N) \geq q = p+2$ implies $\mathrm{deg}(L_{C}) = 2g+p$ in which case by \cite[Theorem 5.2 (3) and Corollary 5.3]{NP24} one has $K_{r_{C}-(q-1)}(C, L) =  K_{p, 1}(C, K_{C}, L) \neq 0$ if and only if either $K_{C}$ is not $p-$very ample which is equivalent to $q \geq \mathrm{gon}(C)$ or for $q \geq 3$ when $N$ embeds $C$ as a plane curve of degree $p+3 \geq 4.$
\end{proof}
We now prove the following corollary of Theorem \ref{$(M_q)$_rational_surfaces} which connects Property$-(M_{q})$ with $k-$very ampleness.
\begin{corollary} (Weight-one syzygies and $k-$very ampleness)
\label{K}
    Let $L$ be an ample and base point free line bundle on a del Pezzo surface $X$ with $(-K_{X} \cdot L) \geq k+4.$ Then, the following conditions are equivalent
    \begin{itemize}
        \item[\textbf{(a)}] $L$ satisfies Property-$(M_{2+k}).$
        \item[\textbf{(b)}] $K_{X}+L$ is birationally $k-$very ample.
        \item[\textbf{(c)}] $K_{X}+L$ is birationally $k-$spanned.
        \item[\textbf{(d)}] Any smooth curve $C$ in $|L|$ has gonality $\mathrm{gon}(C) \geq k+2.$
    \end{itemize}
\end{corollary}
\begin{proof}
    This follows by putting $q = k+2$ in Theorem \ref{$(M_q)$_rational_surfaces} and using \cite[Theorem 1.1]{Knu03}.
\end{proof}
\subsubsection{Multiple line bundles on rational surfaces}
\begin{theorem} \label{$(M_{q})$_multiple_bundles_rational_surfaces}
    Let $Y$ be a rational surface and let $B$ be an ample and base point free line bundle on $Y$ with $(-K_{Y} \cdot B) \geq 3.$ \\
    Then, one has the following vanishing
    \begin{equation}
        \label{$q-$vanishing_rational_surfaces}
        H^{1}(M_{ \vec{m}B} (K_{Y}+\ell B)) = 0
    \end{equation}
    for all positive-integer vectors $\vec{m} = (m_{1}, \cdots, m_{k+1})$ and for all $0 \leq k \leq q-2$ and all $\ell \geq \ell_{q}^{\min}$ where 
    $$\ell^{\min}_{q} = \begin{cases}
    q+1, & \text{ for } (Y, B) = (\mathbb{P}^{2}, \mathcal{O}_{\mathbb{P}^{2}}(1))  \\
    2+\ceil[\bigg]{\dfrac{2q-2}{(B^{2})}}, & \text{ otherwise }
    \end{cases}$$ 
    is the minimal integer $t$ such that
    \begin{equation}
        \label{conditions_for_$q-$vanishing_rational_surfaces}
        \begin{cases}
            \text{  } (t-1)(B^{2}) > 2q-2, & \text{ for } \ell \geq 2 \\
            \text{  } t \geq q+1, & \text{ for } (Y, B) = (\mathbb{P}^{2}, \mathcal{O}_{\mathbb{P}^{2}}(1))
        \end{cases}
    \end{equation}
    In particular, under the conditions (\ref{conditions_for_$q-$vanishing_rational_surfaces}) the line bundle $L = \ell B$ satisfies property-$(M_{q}).$ 
\end{theorem}
    \begin{proof} We proceed the proof through induction on the integer $q \geq 2.$
        \begin{itemize}
            \item[\textbf{(i)} ] First we treat the case $q = 2.$ That is, we show that given the conditions on $B$ mentioned above,
            $$H^{1}(M_{mB}(K_{Y}+\ell B)) = 0$$
            Notice that for $q = 2,$ $\ell^{\text{ceil}}_{q} \geq 3$ for $(B^{2}) \leq 2q-2 = 2$ and $\ell^{\text{ceil}}_{q} = 2$ for $(B^{2}) > 2q-2 = 2,$ that is $(B^{2}) \geq 3.$ \\
            For $\ell \geq 3,$ this vanishing follows from Theorem \ref{surj_mult_surf_maps} and for $\ell \geq 2,$ from Riemann-Roch theorem, \cite[Lemma 1.2]{BG01} and using $(-K_{Y} \cdot B) \geq 2$ and $(B^{2}) \geq 3$, we get
            $$h^{0}(B) = 1+\dfrac{1}{2} \left[ (B^{2})+(-K_{Y} \cdot B) \right] \geq 4$$
            this vanishing follows from Theorem \ref{surj_mult_surf_maps}.
            \item[\textbf{(ii)} ]  \textbf{\underline{Induction hypothesis} : } Let $B$ be an ample and base point free line bundle on $X$ satisfying (\ref{conditions_for_$q-$vanishing_rational_surfaces}) 
            Then, one has
            $$
                H^{1}(M_{\vec{m}'B}(K_{Y}+\ell' B)) = 0 \hspace{0.2cm} \text{ for } 
                \begin{cases}
                    \text{positive integer vectors} \\
                    \vec{m}' = (m'_{1}, \cdots, m'_{k'+1}), \\
                    0 \leq k' \leq (q-1)-2 \ , \\
                    \ell' \geq \ell^{\min}_{q-1} \ .
                \end{cases}
            $$
            \item[\textbf{(iii)} ]  \textbf{\underline{Inductive Step} : } Let $q > 2,$ $B$ be an ample and base point free line bundle on $X$ satisfying (\ref{conditions_for_$q-$vanishing_rational_surfaces}), $\vec{m} = (m_{1}, \cdots, m_{k+1}),$ with $0 \leq k \leq q-2$ as before. Suppose that $\mathcal{E}_{\ell} = M_{\vec{m}B}[(k+1)-1](K_{Y}+\ell B) = M_{\vec{m}B}[k](K_{Y}+\ell B)$ for any integer $\ell$ and $m = m_{k+1}.$ Then $M_{\vec{m}B}(K_{Y}+\ell B) = M_{mB}(\mathcal{E}_{\ell})$ and from the induction hypothesis, $H^{1}(\mathcal{E}_{\ell}) = H^{1}(\mathcal{E}_{\ell-1})  = 0$ for $\ell \geq \ell^{\text{ceil}}_{q}.$ Tensoring the kernel bundle sequence (\ref{kernel_bundle}) for $L = mB$ with $\mathcal{E}_{\ell}$ and taking the long-exact sequence of cohomology one gets
            \small{$$\text{ } \hspace{0.5cm} H^{0}(mB) \otimes H^{0}(\mathcal{E}_{\ell}) \xrightarrow[\hspace{0.7cm}]{\mu_{m, \ell}} H^{0}(\mathcal{E}_{\ell}(mB)) \rightarrow H^{1}(M_{mB}(\mathcal{E}_{\ell})) \rightarrow H^{0}(mB) \otimes H^{1}(\mathcal{E}_{\ell})$$}
            Since $H^{1}(\mathcal{E}_{\ell}) = 0,$ the vanishing $H^{1}(M_{mB}(\mathcal{E}_{\ell})) = 0$ is equivalent to the surjectivity of $\mu_{m, \ell}.$ Using Observation \ref{[Observation-(1.2)]{BG00}} above, the surjectivity of $\mu_{m, \ell}$ follows from the surjectivity of the maps $\mu_{1, \ell+j}$ for $0 \leq j \leq m-1.$ Thus, it is sufficient to prove the surjectivity of $\mu_{1, \ell}$ for $\ell \geq \ell_{q}$. Let $C$ be a smooth curve in $|B|.$ Since by induction hypothesis, $H^{1}(\mathcal{E}_{\ell}(-B)) = H^{1}(\mathcal{E}_{\ell-1}) = 0,$ using Observation \ref{[Observation-(1.3)]{BG00}} we conclude that it suffices to show the surjectivity of
            $$\begin{tikzcd}[row sep = large, column sep = large]
            H^{0}(B_{C}) \otimes H^{0}(M_{\vec{m}B}[k, 0](K_{Y}+\ell B)|_{C}) \ar{r}{ } & H^{0}(M_{\vec{m}B}[k, 0](K_{Y}+(\ell+1) B)|_{C}) \\
             H^{0}(B_{C}) \otimes H^{0}(\mathcal{E}_{\ell}|_{C}) \ar[u,-, blue ,double equal sign distance] \ar{r}{\mu_{1, \ell}|_{C}} & H^{0}(\mathcal{E}_{\ell}|_{C}(B_{C})) \ar[u,-, blue ,double equal sign distance]
            \end{tikzcd}$$
            and since $H^{1}((m_{i}-1)B) = 0$ for all $1 \leq i \leq k+1$ it is sufficient to prove the surjectivity of
            $$\begin{tikzcd}[row sep = large, column sep = large]
            H^{0}(B_{C}) \otimes H^{0}(M_{\vec{m}B}[k, k](K_{Y}+\ell B)|_{C}) \ar{r}{ } & H^{0}(M_{\vec{m}B}[k, k](K_{Y}+(\ell+1) B)|_{C}) \\
             H^{0}(B_{C}) \otimes H^{0}(M_{\vec{m}B_{C}}[k](K_{Y}+\ell B)|_{C}) \ar[u,-, blue ,double equal sign distance] \ar{r}{\gamma} & H^{0}(M_{\vec{m}B_{C}}[k](K_{Y}+(\ell+1) B)|_{C}) \ar[u,-, blue ,double equal sign distance]
            \end{tikzcd}$$
            By adjunction formula, $\mathcal{O}_{C}(K_{Y}+\ell B) = K_{C}+(\ell-1) B_{C}$ To show the surjectivity of $\gamma$ we show that the cokernel of $\gamma$ vanishes and by Kodaira vanishing theorem the cokernel of $\gamma$ is $H^{1}(M_{\vec{m}B_{C}} (K_{C}+(\ell-1)B|_{C})).$ Next, when $(-K_{Y} \cdot B) \geq 2$ one has $\mathrm{deg}(B|_{C}) \geq 2g$. Then, by Proposition \ref{[Theorem-(1.2)]{Bu94}} and Corollary \ref{[Corollary-(3.7)]{Miy87}} above, one has that $\bigotimes^{k}_{i = 1}M_{m_{i} B|_{C}}, \bigotimes^{k+1}_{i = 1}M_{m_{i}B|_{C}}$ and $\bigotimes^{k}_{i = 1}M_{m_{i} B|_{C}}(R|_{C}), \bigotimes^{k+1}_{i = 1}M_{m_{i}B|_{C}}((R-B)|_{C})$ are semistable. \\
            Now, when $(-K_{Y} \cdot B) \geq 2, \ell \geq 2$ one has,
            $$\mu(M_{\vec{m}B|_{C}} (K_{C}+(\ell-1)B|_{C})) \geq -2(k+1)+(2g-2)+(\ell-1)(B^{2})$$ $$\geq [(\ell-1)(B^{2})-2q+2]+2g-2 > 2g-2$$ (since $0 \leq k \leq q-2$) for $(\ell-1)(B^{2}) > 2q-2,$ which follows from (\ref{conditions_for_$q-$vanishing_rational_surfaces}). \\
            Now we come to the case $(Y, B) = (\mathbb{P}^{2}, H).$ Since, $(-K_{Y} \cdot B) = 3$ by adjunction formula $1 = (H^{2}) = \mathrm{deg}(B|_{C}) = 2g+1,$ so $C \cong \mathbb{P}^{1},$ and $B|_{C} = \mathcal{O}_{\mathbb{P}^{1}}(1).$ The vanishing $H^{1}(M_{\vec{m}B} (K_{Y}+\ell B)) = 0$ then follows from the vanishing $H^{1}(M_{\vec{m}B}[k+1, k+1] (K_{C}+(\ell-1)B|_{C})) = H^{1}(\bigotimes^{k+1}_{i = 1}M_{\mathcal{O}_{\mathbb{P}^{1}}(m_{i})} \otimes \mathcal{O}_{\mathbb{P}^{1}}(\ell-3)) = 0$ But $M_{\mathcal{O}_{\mathbb{P}^{1}}(m_{i})}  \cong \mathcal{O}_{\mathbb{P}^{1}}(-1)^{\oplus m_{i}},$ for all integers $1 \leq i \leq k+1$ and therefore, the last vanishing is true if $\ell-(k+1)-3 \geq 0$ that is if $\ell \geq q+1.$ \\ This completes the proof.
            \end{itemize}
\end{proof}
\begin{corollary} $($ Linear bounds on Property$-(M_{q})$ for multiple bundles $)$ \\
    Let $B$ be an ample and base point free line bundle on a rational surface $Y.$ Then the multiple bundle $L = \ell B$ satisfies the $(M_{q})-$property if $\ell \geq \ell_{q}$ where
    \begin{equation}
        \ell_{q} = 
        \begin{cases}
            \text{  } q+1 & \text{ for } (B^{2}) \geq 2 \\
            \text{  } q+1 & \text{ for }  (Y, B) =  (\mathbb{P}^{2}, \mathcal{O}_{\mathbb{P}^{2}}(1)) \\
            \text{  } 2q & \text{ for } (B^{2}) = 1 \text{ but } (Y, B) \neq  (\mathbb{P}^{2}, \mathcal{O}_{\mathbb{P}^{2}}(1)) \\
        \end{cases}
    \end{equation}
\end{corollary}
\subsection{Fano varieties of dimension $n \geq 3$}
In this section, we study the syzygies of Fano varieties using the method of hyperplane sections. We follow the notation of Section-(\ref{hyperplane_cuts}). \par A smooth projective variety $X$ of dimension $n$ is called a \textbf{Fano} if there is an ample line bundle $A$ such that $-K_{X} = mA$ for some positive integer $m.$ The \textbf{index}, denoted by $\lambda(X),$ of a Fano variety $X$ is the largest integer $m$ such that there is an ample line bundle $A$ on $X$ such that $-K_{X} = mA.$ It is known that $\lambda(X) \leq n+1,$ with equality if and only if $X = \mathbb{P}^{n}.$ For any ample and base point free line bundle $B$ on $X,$ by $\mathrm{gon}_{\max}(B)$ we denote
\begin{equation}
    \mathrm{gon}_{\max}(B) = \max \left\{ \mathrm{gon}(C) \ | \ C \text{ is a smooth curve in } |B|_{X^{(n-2)}}| \ \right\}
\end{equation}
\begin{theorem}
\label{Fano}
    Let $B$ be an ample and base point free line bundle on a smooth projective variety $X$ with $-K_{X} = \lambda B$ where $\lambda$ is an integer satisfying $\lambda \geq n-1.$ Then, $B$ satisfies Property$-(M_{q})$ for all $q \leq (n-2)+\mathrm{gon}_{\max}(B)$ provided
    \begin{equation}
        (B^{n}) \geq \max \left(3, \dfrac{q-n+4}{\lambda-n+2} \right)
    \end{equation}
\end{theorem}
\begin{proof}
    Let $Y = X^{(n-2)}.$ Since $(B^{n}) \geq 3,$ from \cite[Theorem 2.1, Corollary 2.8]{BG01} it follows that $B$ is normally generated. Now, since $K_{X} = -\lambda B,$ $R(B)$ is Gorenstein and therefore by Proposition \ref{syzygies_of_floors}, $B$ satisfies Property$-(M_{q})$ if and only if $B_{Y}$ satisfies Property$-(M_{q-(n-2)}).$ Moreover, by adjunction formula, 
    $$\lambda(B^{n})= (-K_{X} \cdot B^{n-1}) = (-K_{Y} \cdot B_{Y})+(n-2)(B^{n})$$
    and $Y$ is a del Pezzo surface. In particular, it is rational. Therefore, by Theorem \ref{$(M_q)$_rational_surfaces} the proof is complete.
\end{proof}
\subsection{Ruled varieties}
In this section we study the syzygies of line bundles on ruled varieties over a smooth curve $C$ of genus $g \geq 0.$
\begin{definition}
    Let $C$ be a smooth projective curve of genus $g$ and $E$ be a vector bundle of rank $n.$ We say that the projective bundle $\pi: X = \mathbb{P}(E) \rightarrow C$ is a \textbf{ruled variety} of dimension $n$ over $C.$ When $n = 2$ we say that $X$ is a \textbf{ruled surface} over $C.$
\end{definition}
\begin{notation}
    In the following by $C$ we denote a smooth projective curve of genus $g.$ We will denote by $E$ a vector bundle of rank $n$ on a smooth projective curve $C$ and $\pi : X = \mathbb{P}(E) \rightarrow C$ the ruled variety of dimension $n$ over $C.$ We denote by $C_{0}$ as a divisor in the linear system $|\mathcal{O}_{\mathbb{P}(E)}(1)|$ and let $f$ be a divisor corresponding to the fiber $\pi^{\ast}(p) \cong \mathbb{P}^{1}$ of a general point $p \in{C}.$ Then, $\mathrm{Pic}(X) \cong \mathbb{Z} \cdot C_{0} \oplus \mathrm{Pic}(C) \cdot f$ and $\mathrm{Num}(X) \cong \mathbb{Z} \cdot C_{0} \oplus \mathbb{Z} \cdot f$ is generated by $C_{0}$ and $f.$ We denote by $\mathfrak{e} = \bigwedge^{n}E$ and $e = -\mathrm{deg}(\mathfrak{e}) = -\mathrm{deg}(E).$ Also, by $A$ we denote the number $k_{L} = \begin{pmatrix}
    n+a-1 \\
    a
\end{pmatrix}-1$
\end{notation}
\begin{remark}
    Taking determinants of the sequence of K\"{a}hler differentials:
$$0 \rightarrow \pi^{\ast}K_{C} \rightarrow \Omega^{1}_{X} \rightarrow \Omega^{1}_{X/C} \rightarrow 0$$
and the Euler sequence:
$$0 \rightarrow \Omega^{1}_{X/C} \rightarrow \pi^{\ast}E \otimes \mathcal{O}_{\mathbb{P}(E)}(-1) \rightarrow \mathcal{O}_{X} \rightarrow 0$$
we get, $K_{X} \equiv (-n)C_{0}+(2g-2+\mathrm{deg}(E))f.$ Thus, the canonical divisor satisfies: $K_{X} \equiv (-n)C_{0}+(2g-2-e)f.$
\end{remark}
Now we analyze syzygies of all line bundles on a ruled variety. First, we state the following generalization of \cite[Lemma 3.1]{EPark05}. The proof follows by exactly same argument as \cite[Lemma 3.1]{EPark05} with only difference being for ruled varieties one has $\mathrm{rank}(\mathcal{K}_{L}) = k_{L}$ where $\mathcal{K}_{L}$ is defined by
\begin{equation}
    0 \rightarrow \mathcal{K}_{L} \rightarrow \pi^{\ast}(\pi_{\ast}L) \rightarrow L \rightarrow 0
\end{equation}
in the notation of \cite[Lemma 3.1]{EPark05}.
\begin{lemma} 
\label{[Lemma 3.1]{EPark05}}
Let $\pi : X \rightarrow C$ be a ruled variety over a curve of genus $g$ and let $L \equiv aC_{0}+bf$ and $N \equiv sC_{0}+tf$ be two line bundles on $X$ such that $b+a\mu^{-}(E) \geq 2g+1.$ Let $k_{L} = \begin{pmatrix}
    n+a-1 \\
    a
\end{pmatrix}-1$ \\
Then, for any integer $m \geq 1,$ one has the vanishings
\begin{equation}
\mathrm{H}^{1}(\bigwedge^{m}M_{L} \otimes N) = 0 \hspace{0.5cm} \text{ if } s \geq \mathrm{min}(m, A) \text{ and } t+s\mu^{-}(E) > \dfrac{m(b+a\mu^{-}(E))}{b+a\mu^{-}(E)-g}+2g-2
\end{equation}
and
\begin{equation}
\mathrm{H}^{2}(\bigwedge^{m}M_{L} \otimes N) = 0 \hspace{0.5cm} \text{ if } s+1 \geq m.
\end{equation}
\end{lemma}
\begin{theorem} 
\label{M_q_ruled_variety_gen_line_bundle}
For any integer $n \leq q \leq n+k_{L}-1,$ a line bundle $L \equiv aC_{0}+bf$ on any ruled variety $\pi: X = \mathbb{P}(E) \rightarrow C$ of dimension $n$ over a curve of genus $g$ with $a \geq 1$ and $b+a\mu^{-}(E) \geq 2g+1,$ satisfies the vanishing
\begin{equation}
     H^{1}(\bigwedge^{k+1}M_{L} \otimes (K_{X}+(n-1)L)) = 0 \text{ for } 0 \leq k \leq q-n,
\end{equation}
provided
\begin{equation}
\label{condition_for_M_q_ruled_variety_gen_line_bundle}
a \geq \ell^{\mathrm{ceil}}_{q} \hspace{0.5cm} \text{ and } \hspace{0.5cm}  \dfrac{r(L)+1}{k_{L}+1} = b+a\mu^{-}(E)+(1-g) \geq \ell^{\mathrm{ceil}}_{q}
\end{equation}
where $r = h^{0}(L)-1.$ In particular, if $L$ satisfies (\ref{condition_for_M_q_ruled_variety_gen_line_bundle}), then it satisfies Property$-(M_{q})$. Moreover, when $E$ is stable, the second inequality in (\ref{condition_for_M_q_ruled_variety_gen_line_bundle}) becomes
\begin{equation}
    a \geq \ell^{\mathrm{ceil}}_{q} \hspace{0.2cm} \text{ and } \hspace{0.2cm} b \geq \ell^{\mathrm{ceil}}_{q}+\dfrac{ae}{n}+g-1
\end{equation}
\end{theorem}
\begin{proof}
Let $\Delta \mu(E) = \mu(E)-\mu^{-}(E)$ and $x = \mu(\pi_{\ast}L) = b+a \mu^{-}(E).$ So that, $E$ is stable if and only if $\Delta \mu(E) = 0.$\\ 
From the conditions of the theorem, one has
$K_{X}+(n-1)L \equiv (-n)C_{0}+(2g-2-e)f+(n-1)(aC_{0}+bf)$ \\ $ = ((n-1)a-n)C_{0}+((n-1)b+2g-2-e)f,$ where $e = -n \mu(E).$ Now, choosing $s = (n-1)a-n,$ $t = (n-1)b+(2g-2-e)$ and $N = K_{X}+(n-1)L \equiv sC_{0}+tf$ we get 
$$t+s\mu^{-}(E) = (n-1)x+(2g-2+n \cdot \Delta \mu(E) ).$$
From Lemma \ref{[Lemma 3.1]{EPark05}} above, for the vanishing we need
\begin{eqnarray}
\label{c_(M_q)_ruled_gen_2_1}
    s \geq min(q-n+1, k_{L}) \text{ and } \\
    t+s\mu^{-}(E) > (q-n+1) \left( \dfrac{x}{x-g} \right)+2g-2
    \label{c_(M_q)_ruled_gen_2_2}
\end{eqnarray}
The first part of (\ref{condition_for_M_q_ruled_variety_gen_line_bundle}) comes from (\ref{c_(M_q)_ruled_gen_2_1}) and the second part of (\ref{condition_for_M_q_ruled_variety_gen_line_bundle}) comes from  (\ref{c_(M_q)_ruled_gen_2_2}). \\
Indeed, the inequality (\ref{c_(M_q)_ruled_gen_2_2})  is same as
\begin{equation}
\label{c_(M_q)_ruled_gen_3_0}
    x+\left(\dfrac{n(x-g)}{(n-1)x} \cdot \Delta\mu^{-}(E) \right) > \dfrac{q-n+1}{n-1}+g = \dfrac{q}{n-1}-(1-g)
\end{equation}
Thus, (\ref{c_(M_q)_ruled_gen_3_0}) is satisfied for the following bound
\begin{equation}
\label{c_(M_q)_ruled_gen_3}
    x+(1-g) \geq \ceil[\bigg]{\dfrac{q+1}{n-1}}
\end{equation}
When $E$ is stable, the inequality (\ref{c_(M_q)_ruled_gen_3_0})  is same as
\begin{equation}
\label{c_(M_q)_ruled_gen_3_1}
    x > \dfrac{q-n+1}{n-1}+g = \dfrac{q}{n-1}-(1-g) \text{ or equivalently } x+(1-g) \geq \ceil[\bigg]{\dfrac{q+1}{n-1}}
\end{equation}
 But by Riemann Roch on the vector bundle $\pi_{\ast}L$ on $C,$ one gets
 $$\dfrac{\chi(\pi_{\ast}L)}{rank(\pi_{\ast}L)} = \mu(\pi_{\ast}L)+(1-g) = x+(1-g)$$
Since by our hypothesis $x \geq 2g+1$ we have $h^{i}(\pi_{\ast}L) = 0$ for $i > 0$ and 
$rank(\pi_{\ast}L) = 
\begin{pmatrix}
n+a-1 \\
a
\end{pmatrix}
$ by cohomology base change formula one gets
\begin{equation}
    \dfrac{h^{0}(L)}{\begin{pmatrix}
n+a-1 \\
a
\end{pmatrix}} = \dfrac{\chi(L)}{rank(\pi_{\ast}L)} = \dfrac{\chi(\pi_{\ast}L)}{rank(\pi_{\ast}L)} = x+(1-g)
\end{equation}
combining the inequality (\ref{c_(M_q)_ruled_gen_2_2}), the inequality (\ref{c_(M_q)_ruled_gen_3}) and the above two equalities one gets the right-hand-side of (\ref{condition_for_M_q_ruled_variety_gen_line_bundle}). \\
Now we analyze (\ref{c_(M_q)_ruled_gen_2_1}) in the following.
Since $A \geq q-n+1$, we need
\begin{equation}
\label{first_inequality-ruled_2_0}
    (n-1)a-n \geq q-n+1 \ \ \text{ which is same as} \ \ a \geq \dfrac{q+1}{n-1}
\end{equation}
This and (\ref{c_(M_q)_ruled_gen_3}) proves (\ref{condition_for_M_q_ruled_variety_gen_line_bundle}). Now, when $E$ is stable, (\ref{first_inequality-ruled_2_0}) and (\ref{c_(M_q)_ruled_gen_3}) become
\begin{eqnarray}
\label{first_inequality-ruled_2}
    (n-1)a-n \geq q-n+1 \ \ \text{ which is same as} \ \ a \geq \dfrac{q+1}{n-1}  \\
\label{second_inequality-ruled_2}
    b \geq \ell^{\mathrm{ceil}}_{q}+\dfrac{ae}{n}+g-1
\end{eqnarray}
This completes the proof.
\begin{remark}
    In the above theorem, if $q > n+k_{L}-1,$ then the Lemma \ref{[Lemma 3.1]{EPark05}} is not applicable to prove property-$(M_{q})$ since it is an impossible to have $(n-1)a-n \geq k_{L} = \min(q-n+1, k_{L})$.
\end{remark}
Now, we mention two corollaries of Theorem \ref{M_q_ruled_variety_gen_line_bundle} which are analogues of Theorem 2A and Theorem 2B of \cite{Bu94} respectively.
\end{proof}
\begin{corollary}
\label{Butler2A}
    For all integers $1 \leq i \leq t,$ let $A_{i} \equiv a_{i}C_{0}+b_{i}f$ be ample line bundles on a ruled variety $\pi : X = \mathbb{P}(E) \rightarrow C$ of dimension $n \geq 2$ over a curve of genus $g.$ Then the line bundle $L = A_{1}+ \cdots + A_{t}$ satisfies Property-$(M_{q})$ for $q \leq n+k_{L}-1$ if
    \begin{equation}
        \dfrac{t}{n} \geq \ell^{\mathrm{ceil}}_{q}
    \end{equation}
    In particular, $L$ satisfies Property-$(M_{q})$ for $t \geq 2q+2.$
\end{corollary}
\begin{proof}
    Since the line bundles $A_{1}, \cdots, A_{t}$ are ample, one has $a_{i} \geq 1$ and $b_{i}+a_{i}\mu^{-}(E) = \dfrac{p_{i}}{q_{i}}$ for integers $p_{i} \geq 1$ and $1 \leq q_{i} \leq n.$ Suppose $a = \sum^{t}_{i = 1} a_{i}$ and $b = \sum^{t}_{i = 1}b_{i}.$ Then, $L \equiv aC_{0}+bf$ and
    $$a \geq t \hspace{0.5cm} \text{ and } \hspace{0.5cm} b+a\mu^{-}(E) \geq \dfrac{t}{n}$$
    Now, if one has $\dfrac{t}{n} \geq \ell^{\mathrm{ceil}}_{q}$ then $a \geq \ell^{\mathrm{ceil}}_{q}$ and $b+a\mu^{-}(E) \geq \ell^{\mathrm{ceil}}_{q}+(g-1)$ since $g-1 \geq -1$ for any genus $g.$ Thus, the conditions of Theorem \ref{M_q_ruled_variety_gen_line_bundle} are satisfies if $\dfrac{t}{n} \geq \ell^{\mathrm{ceil}}_{q}.$ Now, $\dfrac{t}{n} \geq \ell^{\mathrm{ceil}}_{q}.$ is true if $t \geq 2q+2$
\end{proof}
\begin{corollary}
\label{Butler2B}
    For all integers $1 \leq i \leq t,$ let $A_{i} \equiv a_{i}C_{0}+b_{i}f$ be ample line bundles on a ruled variety $\pi : X = \mathbb{P}(E) \rightarrow C$ of dimension $n \geq 2$ over a curve of genus $g.$ Then the line bundle $L = K_{X}+A_{1}+ \cdots + A_{t}$ satisfies Property-$(M_{q})$ for $q \leq n+k_{L}-1$ if
    \begin{equation}
        \dfrac{t}{n} \geq \ell^{\mathrm{ceil}}_{q}+\max(1, e+1-g)
    \end{equation}
    In particular, $L$ satisfies Property-$(M_{q})$ for $t \geq 2q+1+\max(1, e+1-g).$
\end{corollary}
\begin{proof}
    Since the line bundles $A_{1}, \cdots, A_{t}$ are ample, one has $a_{i} \geq 1$ and $b_{i}+a_{i}\mu^{-}(E) = \dfrac{p_{i}}{q_{i}}$ for integers $p_{i} \geq 1$ and $1 \leq q_{i} \leq n.$ Suppose $a = -n+\sum^{t}_{i = 1} a_{i}$ and $b = (2g-2-e)+\sum^{t}_{i = 1}b_{i}.$ Then, $L \equiv aC_{0}+bf$ and
    $$a \geq t-n \hspace{0.5cm} \text{ and } \hspace{0.5cm} b+a\mu^{-}(E) \geq (2g-2-e)+\dfrac{t}{n}$$
    Now, the condition (\ref{condition_for_M_q_ruled_variety_gen_line_bundle}) on $a$ and $b$ is true if
    $$t-n \geq q+1 \geq \ell^{\mathrm{ceil}}_{q} \text{ and } (2g-2-e)+\dfrac{t}{n} \geq \ell^{\mathrm{ceil}}_{q}+(g-1)$$
    and this is equivalent to $\dfrac{t}{n} \geq 1+\ell^{\mathrm{ceil}}_{q}$ and $\dfrac{t}{n} \geq \ell^{\mathrm{ceil}}_{q}+e+(1-g).$ Therefore, by Theorem \ref{M_q_ruled_variety_gen_line_bundle} $L$ satisfies Property-$(M_{q})$ for $\dfrac{t}{n} \geq \ell^{\mathrm{ceil}}_{q}+\max(1, e+1-g).$ Moreover, if $t \geq 2q+2+\max(1, e+1-g)$ then $t \geq \ell^{\mathrm{ceil}}_{q}+\max(1, e+1-g)$
\end{proof}
\section{Conjectures, Boundary examples and optimality}
\label{Open_problems_examples}
In this section, we show by examples that our bounds in Theorems above are optimal. We also compare the two properties$-(M_{q})$ and $-(N_{p})$. 
To make this comparison we imagine the minimal resolution as a rope and the two properties playing a tug of war with the resolution and one of the properties winning the game. With this perspective, we make the following definitions
\begin{definition}
    Let $X$ be a smooth projective variety and $L$ be a line bundle on $X.$ We define the following numbers
    \begin{align}
        p_{\max}(L) = \max \left\{ p \ | \ L \text{ satisfies property} -(N_{p}) \ \right\} \\
        q_{\max}(L) = \max \left\{ q \ | \ L \text{ satisfies property} -(M_{q}) \ \right\} \\
        \mathfrak{tug}(L) = p_{\max}(L)-q_{\max}(L) \\
        \label{definition_of_delta}
        \delta(L) = (r(L)-1)-p_{\max}(L)-q_{\max}(L)
    \end{align}
    We say that property$-(M_{q})$ \textbf{wins the tug of war } (of syzygies) over property-$(N_{p}),$ or that property-$(M_{q})$ \textbf{tugs} property-$(N_{p})$ if $\mathfrak{tug}(L) > 0.$ We also express this situation by saying that property$-(N_{p})$ \textbf{pushes} or \textbf{dominates} property-$(M_{q})$. Similarly, we say that property$-(N_{p})$ \textbf{tugs} property-$(M_{q}),$ or that property-$(M_{q})$ \textbf{pushes} property-$(N_{p})$ if $\mathfrak{tug}(L) < 0.$ We say that $\delta(L)$ is the \textbf{length of mixed syzygies} of $L$ since there are both weight-one and more than weight-one syzygies in stages $i = p_{\max}$ to $i = p_{\max}+\delta$ of the resolution (\ref{free_resolution_of_R_L}) of $R(L).$ \par
    In this section, we also justify the Conjecture \ref{conjecture} through the following examples:    
\end{definition}
\begin{example} (\textbf{Projective spaces}) \\
    Here we analyze the situation of the projective spaces $\mathbb{P}^{n}$ for $n \geq 2.$ We start with the projective plane. \\
 \textbf{(i)} ( \textbf{Syzygies of } $\mathbb{P}^{2}$ ) \textbf{ } Let $X = \mathbb{P}^{2},$ with chosen hyperplane $H \in |\mathcal{O}_{\mathbb{P}^{2}}(1)|,$ $D = -K_{\mathbb{P}^{2}} = -3H.$ \\
From Theorem \ref{$(M_q)$_rational_surfaces} or Theorem \ref{CM_theorem} it follows that the Veronese embedding 
$$\phi_{\mathcal{O}(\ell)} : \mathbb{P}^{2} \hookrightarrow \mathbb{P}^{\begin{pmatrix}
    \ell+2 \\
    2
\end{pmatrix}-1}$$
defined by $L = \ell \cdot H$ satisfies property-($M_{q}$) for $\ell \geq q+1$ because $reg_{H}((q+1)H) \leq 0$ for all $q \geq 2.$  \\
By Theorem \ref{CM_theorem} the bound $\ell \geq q+1$ is sharp since any plane curve $C$ of degree $q+1$ (that is a member of $|(q+1)H|$ has gonality $\mathrm{gon}(C) = q.$ In other words, $\ell H$ satisfies Property-$(M_{q})$ if and only if $\ell \geq q+1.$ This proves the optimality of Theorems \ref{CM_theorem}, \ref{$(M_q)$_rational_surfaces} and \ref{$(M_{q})$_multiple_bundles_rational_surfaces} . \\
Below we give an elementary proof of the fact that the Veronese embedding $\phi_{\mathcal{O}(2)}$ defined by $\mathcal{O}_{\mathbb{P}^{2}}(2)$ does not satisfy property-$(m_{2})$ even though $\mathcal{O}(2)$ is projectively normal. \\
Indeed, if $L = 2H,$ then from the Euler sequence of $\mathbb{P}^{5}$ restricted to the Veronese surface $\phi_{\mathcal{O}(2)}(\mathbb{P}^{2})$ that is the image of the embedding $\phi_{\mathcal{O}(2)} : \mathbb{P}^{2} \hookrightarrow \mathbb{P}^{5}$ we get
\begin{equation}
\label{Euler_for_$v_2P2$}
0 \rightarrow \Omega^{1}_{\mathbb{P}^{5}}|_{\mathbb{P}^{2}} \rightarrow \mathcal{O}(-2)^{\oplus 6} \rightarrow \mathcal{O} \rightarrow 0
\end{equation}
Twisting the diagram (\ref{Euler_for_$v_2P2$}) by $\mathcal{O}(1)$ we and taking long-exact sequence of cohomology, we get
$$0 = \mathrm{H}^{0}(\mathcal{O}(-1))^{\oplus 6} \rightarrow \mathrm{H}^{0}(\mathcal{O}(1)) \rightarrow \mathrm{H}^{1}(\Omega^{1}_{\mathbb{P}^{5}}|_{\mathbb{P}^{2}}(1))$$
which shows that $h^{1}(M_{L} \otimes (K_{\mathbb{P}^{2}}+L)) = h^{1}(\Omega^{1}_{\mathbb{P}^{5}}|_{\mathbb{P}^{2}}(1)) \geq h^{0}(\mathcal{O}(1)) = 3$ and similarlt, $h^{1}(M_{L} \otimes L) = h^{1}(\Omega^{1}_{\mathbb{P}^{5}}|_{\mathbb{P}^{2}}(4)) = 0.$ \\
Below, we list Betti tables of $\mathcal{O}_{\mathbb{P}^{2}}(\ell)$ for $2 \leq \ell \leq 5.$ We refer the reader to the excellent resource \href{https://syzygydata.com/#}{syzygydata.com} for more information about syzygies of Veronese embeddings. In the table of $\mathcal{O}(5)$ we write $a(1^b)$ in place of a number $N$ if $\ceil[\bigg]{\dfrac{N}{a(10^b)}} = 1$ for integers $a$ and $b$ satisfying $0 \leq a \leq 9$ and $b \geq 0.$
\begin{scriptsize}
\begin{table}[!htb]
    \caption{Betti tables of Veronese embeddings $\mathcal{O}(\ell)$ for $2 \leq \ell \leq 5$}
    \begin{minipage}{.5\linewidth}
      \centering
        \begin{tabular}{c | c c c c c}
        \cline{2-6}
        & \multicolumn{4}{c}{Betti table of $\mathcal{O}(2)$} \\
        \cline{1-6}
$j  \diagdown i$ & $0$ & $1$ & $2$ & $3$ & $4$ \\
\cline{1-6}
&  &  &  & \\
$0$ & $1$ & $0$ & $0$ & $0$ & $0$  \\
$1$ & $0$ & $6$ & $8$ & $3$ & $0$ \\
$2$ & $0$ & $0$ & $0$ & $0$ & $0$
\end{tabular}
    \end{minipage}%
    \begin{minipage}{.5\linewidth}
      \centering
        \begin{tabular}{c | c c c c c c c c c}
        \cline{2-10}
        & \multicolumn{8}{c}{Betti table of $\mathcal{O}(3)$} \\
        \cline{1-10}
$j  \diagdown i$ & $0$ & $1$ & $2$ & $3$ & $4$ & $5$ & $6$ & $7$ & $8$ \\
\cline{1-10}
&  &  &  &  &  &  & & & \\
0 & $1$ & $0$ & $0$ & $0$ & $0$ & $0$ & $0$ & $0$ & $0$ \\
$1$ & $0$ & $27$ & $105$ & $189$ & $189$ & $105$ & $27$ & $0$ & $0$ \\
$2$ & $0$ & $0$ & $0$ & $0$ & $0$ & $0$ & $0$ & $1$ & $0$
\end{tabular}
    \end{minipage} 
\end{table}
\\
\begin{table}[!htb]
        \begin{tabular}{c | c c c c c c c c c c c c c c}
        \cline{2-15}
        & \multicolumn{8}{c}{Betti table of $\mathcal{O}(4)$} \\
        \cline{1-15}
$j  \diagdown i$ & $0$ & $1$ & $2$ & $3$ & $4$ & $5$ & $6$ & $7$ & $8$ & $9$ & $10$ & $11$ & $12$ & $13$ \\
\cline{1-15}
&  &  &  &  &  &  & & & & & & & & \\
$0$ & $1$ & $0$ & $0$ & $0$ & $0$ & $0$ & $0$ & $0$ & $0$ & $0$ & $0$ & $0$ & $0$ & $0$ \\
$1$ & $0$ & $75$ & $536$ & $1947$ & $4488$ & $7095$ & $7920$ & $6237$ & $3344$ & $1089$ & $120$ & $0$ & $0$ & $0$ \\
$2$ & $0$ & $0$ & $0$ & $0$ & $0$ & $0$ & $0$ & $0$ & $0$ & $0$ & $55$ & $24$ & $3$ & $0$
\end{tabular}
\end{table}
\end{scriptsize}
\begin{flushleft}
\resizebox{\textwidth}{!}{
        \begin{tabular}{c | c c c c c c c c c c c c c c c c c c c c}
        \cline{2-21}
        & \multicolumn{3}{c}{ } & \multicolumn{8}{c}{Betti table of $\mathcal{O}(5)$} \\
        \cline{1-21}
$j  \diagdown i$ & $0$ & $1$ & $2$ & $3$ & $4$ & $5$ & $6$ & $7$ & $8$ & $9$ & $10$ & $11$ & $12$ & $13$ & $14$ & $15$ & $16$ & $17$ & $18$ & $19$ \\
\cline{1-21}
&  &  &  &  &  &  & & & & & & & &  & & & & &  \\
$0$ & $1$ & $0$ & $0$ & $0$ & $0$ & $0$ & $0$ & $0$ & $0$ & $0$ & $0$ & $0$ & $0$ & $0$ & $0$ & $0$ & $0$ & $0$ & $0$ & $0$  \\
$1$ & $0$ & $165$ & $2(1^3)$ & $1^4$ & $4(1^4)$ & $1^5$ & $2(1^5)$ & $4(1^5)$ & $5(1^5)$ & $6(1^5)$ & $5(1^5)$ & $3(1^5)$ & $(1^5)$ & $4(1^4)$ & $5(1^3)$ & $375$ & $0$ & $0$ & $0$ & $0$ \\
$2$ & $0$ & $0$ & $0$ & $0$ & $0$ & $0$ & $0$ & $0$ & $0$ & $0$ & $0$ & $0$ & $0$ & $2(1^3)$ & $4(1^3)$ & $2(1^3)$ & $595$ & $90$ & $6$ & $0$
\end{tabular}}
\end{flushleft}
\textbf{ } The Betti tables above show the optimality of Theorem \ref{$(M_q)$_rational_surfaces} as they show that $q_{\max}(\mathcal{O}(\ell)) $ $= \ell-1 = \mathrm{gon}(\mathcal{O}(\ell)).$ From \textit{G. Ottaviani} and \textit{R. Paoletti}, one has $p_{\max}(\ell) = 3(\ell-1).$ Combining this with our result, we get the length of the region of mixed syzygies
$$\delta(\ell) = \begin{pmatrix}
    \ell+2 \\
    2
\end{pmatrix}-4(\ell-1)-2 = \begin{pmatrix}
    \ell-2 \\
    2
\end{pmatrix} = h^{0}(K_{\mathbb{P}^{2}})-\mathrm{gon}(\ell H)+1$$
and thus Conjecture \ref{conjecture} is true for $\mathbb{P}^{2},$ $p_{\max} > q_{\max}$ or in other words, $(M_{q})$ tugs $(N_{p})$ for all Veronese embeddings. We also can say that in the whole Betti table of Veronese embeddings, only weight one and weight two syzygies occur since $\mathrm{reg}_{A}(\mathcal{O}_{\mathbb{P}^{2}}) = 2$ for any ample line bundle $A$ on $\mathbb{P}^{2}.$ The examples above verify this. Indeed, from the tables $\delta(2) = \delta(3) = 0,$ $\delta(4) = 1$ and $ \delta(5) = 3.$ Also, from we get $\delta(6) = 6.$  \\ \textbf{ } \\
 \textbf{(ii)} ( \textbf{Syzygies of } $\mathbb{P}^{n}$ ) \textbf{ } Let $H$ be a hyperplane in $P^{n}.$ Since $\mathrm{reg}_{B}(\mathcal{O}_{X}) \leq 2$ for any ample and base point free line bundle $B$ on $X,$ by Theorem \ref{CM_theorem} we conclude that $\ell H$ satisfies property-$(M_{q})$ for all $\ell \geq \ell^{\mathrm{ceil}}_{q} = \ceil[\bigg]{\dfrac{q+1}{n-1}}.$ But since $\ell^{\mathrm{ceil}}_{n} = 2,$ and $\phi_{H} : \mathbb{P}^{n} \rightarrow \mathbb{P}^{n}$ is an isomorphism, $|\ell^{\mathrm{ceil}}_{n}H|$ is the first non-trivial embedding of $\mathbb{P}^{n}$ and therefore our bound for property-$(M_{n})$ in Theorem \ref{CM_theorem} is the optimal in this sense. We do not know about optimality of the bound in Theorem \ref{CM_theorem} for $q \geq n+1$ since proving the optimality would be the same as answering the question
 \begin{question}
    The highest value $q_{\max}$ of $q$ for which $\phi_{\ell H} : \mathbb{P}^{n} \hookrightarrow \mathbb{P}^{\begin{pmatrix}
        \ell+n \\
        n
    \end{pmatrix}-1}$ satisfies property-$(M_{q})$ is
    \begin{equation}
        q_{\max} = 2\ell-1
    \end{equation}
 \end{question}
 \end{example}
Next, we analyze the example of $\mathbb{P}^{1} \times \mathbb{P}^{1}$ as both a rational surface and also a ruled surface.
\begin{example}
    Using one has the following Betti tables of $L \equiv aC_{0}+bf$ on $X = \mathbb{P}^{1} \times \mathbb{P}^{1}$ according to the pair $(a, b)$ below. We start with the syzygies of $L \equiv 2C_{0}+bf$ for $2 \leq b \leq 5$ and using $(-K_{\mathbb{P}^{1} \times \mathbb{P}^{1}} \cdot L) = 4 > 3$ show that our bounds for Theorem \ref{$(M_q)$_rational_surfaces} are optimal. 
    \flushleft{
    \begin{scriptsize}
\begin{table}[!htb]
    \caption{Betti tables of $L = aC_{0}+bf$ for $a = 2$ and $2 \leq b \leq 5$}
    \begin{minipage}{.5\linewidth}
      \centering
        \begin{tabular}{c | c c c c c c c c c}
        \cline{2-9}
        & \multicolumn{7}{c}{Betti table of $(2, 2)$} \\
        \cline{1-9}
$j  \diagdown i$ & $0$ & $1$ & $2$ & $3$ & $4$ & $5$ & $6$ & $7$ \\
\cline{1-9}
&  &  &  &  &  &  & & \\
$0$ & $1$ & $0$ & $0$ & $0$ & $0$ & $0$ & $0$ & $0$ \\
$1$ & $0$ & $20$ & $64$ & $90$ & $64$ & $20$ & $0$ & $0$ \\
$2$ & $0$ & $0$ & $0$ & $0$ & $0$ & $0$ & $1$ & $0$
\end{tabular}
    \end{minipage}%
    \begin{minipage}{.5\linewidth}
    \begin{tiny}
      \centering
        \begin{tabular}{c | c c c c c c c c c c c}
        \cline{2-12}
        & \multicolumn{8}{c}{Betti table of $(2, 3)$} \\
        \cline{1-12}
$j  \diagdown i$ & $0$ & $1$ & $2$ & $3$ & $4$ & $5$ & $6$ & $7$ & $8$ & $9$ & $10$ \\
\cline{1-12}
&  &  &  &  &  &  & & & & & \\
$0$ & $1$ & $0$ & $0$ & $0$ & $0$ & $0$ & $0$ & $0$ & $0$ & $0$ & $0$ \\
$1$ & $0$ & $43$ & $222$ & $558$ & $840$ & $798$ & $468$ & $147$ & $8$ & $0$ & $0$ \\
$2$ & $0$ & $0$ & $0$ & $0$ & $0$ & $0$ & $0$ & $0$ & $9$ & $2$ & $0$
\end{tabular}
\end{tiny}
    \end{minipage} 
\end{table}
\begin{table}[!htb]
        \begin{tabular}{c | c c c c c c c c c c c c c c}
        \cline{2-15}
        & \multicolumn{8}{c}{Betti table of $(2, 4)$} \\
        \cline{1-15}
$j  \diagdown i$ & $0$ & $1$ & $2$ & $3$ & $4$ & $5$ & $6$ & $7$ & $8$ & $9$ & $10$ & $11$ & $12$ & $13$ \\
\cline{1-15}
&  &  &  &  &  &  & & & & & & & & \\
$0$ & $1$ & $0$ & $0$ & $0$ & $0$ & $0$ & $0$ & $0$ & $0$ & $0$ & $0$ & $0$ & $0$ & $0$ \\
$1$ & $0$ & $75$ & $536$ & $1947$ & $4488$ & $7095$ & $7920$ & $6237$ & $3344$ & $1089$ & $120$ & $11$ & $0$ & $0$ \\
$2$ & $0$ & $0$ & $0$ & $0$ & $0$ & $0$ & $0$ & $0$ & $0$ & $0$ & $66$ & $24$ & $3$ & $0$
\end{tabular}
\end{table}
\end{scriptsize}
\begin{tiny}
\begin{flushleft}
\begin{table}[!htb]
        \begin{tabular}{c | c c c c c c c c c c c c c c c c c}
        \cline{2-18}
        & \multicolumn{3}{c}{ } & \multicolumn{8}{c}{Betti table of $(2, 5)$} \\
        \cline{1-18}
$j  \diagdown i$ & $0$ & $1$ & $2$ & $3$ & $4$ & $5$ & $6$ & $7$ & $8$ & $9$ & $10$ & $11$ & $12$ & $13$ & $14$ & $15$ & $16$ \\
\cline{1-18}
&  &  &  &  &  &  & & & & & & & &  & & &  \\
$0$ & $1$ & $0$ & $0$ & $0$ & $0$ & $0$ & $0$ & $0$ & $0$ & $0$ & $0$ & $0$ & $0$ & $0$ & $0$ & $0$ & $0$  \\
$1$ & $0$ & $116$ & $(1^3)$ & $5(1^3)$ & $2(1^4)$ & $3(1^4)$ & $6(1^4)$ & $7(1^4)$ & $7(1^4)$ & $5(1^4)$ & $2(1^4)$ & $8(1^3)$ & $(1^3)$ & $195$ & $14$ & $0$ & $0$ \\
$2$ & $0$ & $0$ & $0$ & $0$ & $0$ & $0$ & $0$ & $0$ & $0$ & $0$ & $0$ & $0$ & $455$ & $210$ & $45$ & $4$ & $0$
\end{tabular}
\end{table}
\end{flushleft}
\end{tiny} \par
\textbf{ } From \textit{G. Martens}, for any curve $C$ in the linear system $|aC_{0}+bf|,$ one has $\mathrm{gon}(C) = a.$ Thus, the examples above show the optimality of Theorem \ref{$(M_q)$_rational_surfaces} since $q_{\max}(2, 2) = q_{\max}(2, 3) = q_{\max}(2, 4) = q_{\max}(2, 5) = 2 = \mathrm{gon}(2C_{0}+bf)$ for $b = 2, 3, 4, 5.$ They also give an idea of the region with mixed syzygies since we get $\delta(2, 2) = 0, \delta(2, 3) = 1,$ $ \delta(2, 4) = 2, \delta(2, 5) = 3.$ According to we also have $\delta(2, 6) = 4, \delta(2, 7) = 5$ and $\delta(2, 8) = 6.$ \\
Similarly, we consider the Betti tables of $L = 3C_{0}+bf$ for $b = 3, 4$ from which it follows that $(M_{q})$ again tugs $(N_{p}),$ $q_{\max}(3, 3) = q_{\max}(3, 4) = 3 = \mathrm{gon}(3C_{0}+bf)$ and $\delta(3, 3) = 2, \delta(3, 4) = 4$.
\begin{tiny}
\begin{flushleft}
\resizebox{\textwidth}{!}{
        \begin{tabular}{c | c c c c c c c c c c c c c c c}
    \cline{2-16}
    & \multicolumn{3}{c}{ } & \multicolumn{8}{c}{Betti table of $(3,3)$} & \multicolumn{3}{c}{ } \\
    \cline{1-16}
    $j \diagdown i$ & 0 & 1 & 2 & 3 & 4 & 5 & 6 & 7 & 8 & 9 & 10 & 11 & 12 & 13 & 14 \\
    \cline{1-16}
    \\[-6pt]
    0 & $1$ & $0$ & $0$ & $0$ & $0$ & $0$ & $0$ & $0$ & $0$ & $0$ & $0$ & $0$ & $0$ & $0$ & $0$  \\
    1 & $0$ & $87$ & $676$ & $2691$ & $6864$ & $12155$ & $15444$ & $14157$ & $9152$ & $3861$ & $780$ & $22$ & $0$ & $0$ & $0$ \\
    2 & $0$ & $0$ & $0$ & $0$ & $0$ & $0$ & $0$ & $0$ & $0$ & $0$ & $165$ & $144$ & $39$ & $4$ & $0$ \\
\end{tabular}}
\end{flushleft}
\end{tiny}
\begin{flushleft}
\resizebox{\textwidth}{!}{
        \begin{tabular}{c | c c c c c c c c c c c c c c c c c c c}
    \cline{2-20}
    & \multicolumn{3}{c}{ } & \multicolumn{8}{c}{Betti table of $(3,4)$} & \multicolumn{8}{c}{ } \\
    \cline{1-20}
    $j \diagdown i$ 
      & 0 & 1 & 2 & 3 & 4 & 5 & 6 & 7 & 8 & 9 
      & 10 & 11 & 12 & 13 & 14 & 15 & 16 & 17 & 18 \\
    \cline{1-20}
    \\[-6pt]
    0 
      & $1$ & $0$ & $0$ & $0$ & $0$ & $0$ & $0$ & $0$ & $0$ & $0$ 
      & $0$ & $0$ & $0$ & $0$ & $0$ & $0$ & $0$ & $0$ & $0$ \\

    1 
      & $0$ & $147$ & $2(1^3)$ & $8(1^3)$ & $3(1^4)$ & $8(1^4)$ 
      & $2(1^5)$ & $2(1^5)$ & $3(1^5)$ & $2(1^5)$ 
      & $2(1^5)$ & $9(1^4)$ & $3(1^4)$ & $4(1^3)$ 
      & $238$ & $15$ & $0$ & $0$ & $0$ \\

    2 
      & $0$ & $0$ & $0$ & $0$ & $0$ & $0$ & $0$ & $0$ & $0$ & $0$
      & $0$ & $0$ & $(1^3)$ & $3(1^3)$ & $2(1^3)$ 
      & $528$ & $85$ & $6$ & $0$ \\
\end{tabular}}
\end{flushleft}
With Theorem \ref{$(M_q)$_rational_surfaces}, \textit{G. Martens} and the syzygies of these line bundles on $\mathbb{P}^{1} \times \mathbb{P}^{1}$ in mind, we make the following conjecture about the syzygies of $\mathbb{P}^{1} \times \mathbb{P}^{1}$ :
\begin{conjecture}
    The length of the region of mixed syzygies of the line bundle $aC_{0}+bf$ on $\mathbb{P}^{1} \times \mathbb{P}^{1}$ is given by
    \begin{equation}
    \label{conjectural_delta_on_PxP}
        \delta(a, b) = (a-1)(b-2) = h^{0}(K_{X}+L)-\mathrm{gon}_{\max}(L)+1
    \end{equation}
    In other words, the highest value $p_{\max}$ of $p$ for which $L = aC_{0}+bf$ satisfies property-$(N_{p})$ is
    \begin{equation}
    \label{conjectural_$p_max$_on_PxP}
        p_{\max}(a, b) = ab+a+b-1-(a-1)(b-2)-a = 2a+2b-3
    \end{equation}
\end{conjecture}
Note that (\ref{conjectural_delta_on_PxP}) and (\ref{conjectural_$p_max$_on_PxP}) are equivalent by (\ref{definition_of_delta}), Riemann Roch on surfaces and since by \textit{G. Martens} and Theorem \ref{$(M_q)$_rational_surfaces}, $q_{\max}(a, b) = \mathrm{gon}(aC_{0}+bf) = a.$ Clearly, Conjecture \ref{conjecture} is true for $X = \mathbb{P}^{1} \times \mathbb{P}^{1}$ if Conjecture \ref{conjectural_delta_on_PxP} is true. But we have shown in the above examples that Conjecture \ref{conjectural_delta_on_PxP} is true for $(a, b) = (2, 2), (2, 3), \cdots, (2, 8), (3, 3), (3, 4).$}
\end{example}
\section{Appendix: Projective normality on surfaces with nef canonical bundle}
We begin this section by proving a theorem about projective normality of adjoint linear series on surfaces with nef canonical bundle.
\begin{theorem}
    Let $X$ be a surface with nef canonical bundle. Let $d \geq 1$ be an integer and $B$ be an ample and base point free line bundle on $X$ such that the line bundle $\Lambda = K_{X}+B$ is base point free. Let $\ell, m, k$ be integers and $d \geq 1$ be any rational number. Then, for any nef line bundle $N$ on $X,$ the line bundles $L = K_{X}+\ell B, L' = K_{X}+\ell'B$ satisfy the vanishing
    \begin{equation}
    \label{$(N_0)$nef_canonical_surface}
        H^{1}(M_{L} \otimes (mL'+N)) = 0 \hspace{0.5cm} \text{ for integer } m \geq 1
    \end{equation}
    if one of the following conditions are satisfied
    \begin{small}
    $$\begin{cases}
    \textbf{ } m\ell'+\ell \geq 2d+2 & \text{ if } ((2d-1)B-2K_{X} \cdot \Lambda) > 0 \text{ or } ((2d-1)B-2K_{X} \cdot \Lambda) \geq 0 \text{ with } h^{0}(\Lambda) \geq 4 \\
    \textbf{ } m\ell'+\ell \geq 2d+3 & \text{ if }  (dB-K_{X} \cdot \Lambda) > 0 \text{ or } (dB-K_{X} \cdot \Lambda) \geq 0 \text{ with } h^{0}(\Lambda) \geq 4 \\
    \textbf{ } m\ell'+\ell \geq 2d+4 & \text{ if }  (dB-K_{X} \cdot \Lambda) \geq 0.
    \end{cases}$$
    \end{small}
    Moreover, $L$ is normally generated if one of the following conditions are satisfied
    \begin{small}
    $$\begin{cases}
    \textbf{ } \ell \geq d+1 & \text{ if } ((2d-1)B-2K_{X} \cdot \Lambda) > 0 \text{ or } ((2d-1)B-2K_{X} \cdot \Lambda) \geq 0 \text{ with } h^{0}(\Lambda) \geq 4 \\
    \textbf{ } \ell  \geq d+2 & \text{ if }  (dB-K_{X} \cdot \Lambda) \geq 0.
    \end{cases}$$
    \end{small}
\end{theorem}
\begin{proof}
    The vanishing (\ref{$(N_0)$nef_canonical_surface}) above is equivalent to the surjectivity
    \begin{equation}
        H^{0}(L) \otimes H^{0}(mL'+N) \xrightarrow[]{\alpha} H^{0}(mL'+L+N)
    \end{equation}
    By Observation to prove the surjectivity of $\alpha$ it is enough to prove the surjectivity of
    \begin{equation}
        H^{0}(B) \otimes H^{0}(mL'+N) \xrightarrow[\hspace{0.6cm}]{\beta} H^{0}(mL'+B+N)
    \end{equation}
    and
    \begin{equation}
        H^{0}(\Lambda) \otimes H^{0}(mL'+(\ell-1)B+N) \xrightarrow[\hspace{0.6cm}]{\gamma} H^{0}(mL'+L+N)
    \end{equation}
    Since $H^{i}(G-iB) = H^{i}(K_{X}+(m\ell'-i)B+N')$ for $1 \leq i \leq 2 = \mathrm{dim}(X),$ $G = mL'+N$ and $N' = (m-1)K_{X}+N,$ one has $\mathrm{reg}_{B}(mL'+N) \leq 0$ or equivalently if $m\ell'-2 \geq 1$ or $\ell' \geq 3 \geq \ceil[\bigg]{\dfrac{3}{m}}.$ Now by Bertini's theorem, there is a smooth curve $C$ in $|B|$ and there is a diagram
    \begin{small}
    $$
        \begin{tikzcd}[row sep = large, column sep = large]
             H^{0}(G) \otimes H^{0}(\mathcal{O}_{X}) \arrow[r, hook] \arrow[d, twoheadrightarrow] &  H^{0}(G) \otimes H^{0}(B) \arrow[r, twoheadrightarrow] \arrow[d, "\beta"] & H^{0}(G) \otimes W \arrow[d, "\beta^{W} "] \\
             H^{0}(G) \arrow[r, hook] & H^{0}(G+B) \arrow[r, twoheadrightarrow] & H^{0}((G+B)|_{C})
        \end{tikzcd}
    $$
    \end{small}
    where $W = H^{0}(B)/H^{0}(\mathcal{O}_{X}).$ Since the left vertical map is surjective, $\beta$ is surjective if and only if the map $\beta^{W}$ is surjective. Finally, by the base point free pencil trick $\beta^{W}$ is surjective if
    \begin{equation}
    \label{pencil_trick_adjoint_projectively_normalI}
        h^{1}((G-B)_{C}) \leq \mathrm{dim}(W)-2 = h^{0}(B)-3
    \end{equation}
    Now, we prove the surjectivity of the map $\beta.$ Now, since $\mathrm{deg}((G-B)_{C}) = (K_{X}+B+(m\ell'-2)B+N)|_{C} = \mathrm{deg}(K_{C}+(m\ell'-2)B_{C}+N_{C}) = 2g-2+(m\ell'-2)(B^{2})+(N \cdot C) > 2g-2$ if $(m\ell'-2) \geq 1$ or in other words if $\ell' \geq 3 \geq \ceil[\bigg]{\dfrac{3}{m}}.$ \\
    Moreover, for $m\ell-2 = 0$ or equivalently $\ell' = \dfrac{2}{m},$ which can happen only if $m = 1, 2$ (since $\ell'$ is an integer) one has $\mathrm{deg}((G-B)|_{C}) = 2g-2+(N \cdot C)$ and therefore either $h^{1}((G-B)_{C}) = 0,$ which happens if $(N \cdot C) > 0$ or if $N_{C}$ is a non-zero divisor of degree zero, or if $h^{1}((G-B)|_{C}) = 1$ if $N_{C} = 0,$ which happens for example if $G = K_{X}+2B$ or in general if $G_{C} = K_{C}+B_{C}.$
    In this second case, (\ref{pencil_trick_adjoint_projectively_normalI}) $h^{0}(B) \geq 4$ in addition. \\
    Now we show that $\gamma$ is surjective. Let $G' = mL'+(\ell-1)B+N$ and we choose a smooth curve $C'$ in $|K_{X}+B|$ which exists by Bertini's theorem. Then, from the diagram
    \begin{small}
    $$
        \begin{tikzcd}[row sep = large, column sep = large]
             H^{0}(G') \otimes H^{0}(\mathcal{O}_{X}) \arrow[r, hook] \arrow[d, twoheadrightarrow] &  H^{0}(G') \otimes H^{0}(\Lambda) \arrow[r, twoheadrightarrow] \arrow[d, "\gamma"] & H^{0}(G') \otimes W' \arrow[d, "\gamma' "] \\
             H^{0}(G') \arrow[r, hook] & H^{0}(G'+\Lambda) \arrow[r, twoheadrightarrow] & H^{0}((G'+\Lambda)|_{C'})
        \end{tikzcd}
    $$
    \end{small}
    Since the left vertical map is surjective, $\gamma$ is surjective if and only if the map $\gamma'$ is surjective. Finally, by the base point free pencil trick $\gamma'$ is surjective if
    \begin{equation}
    \label{pencil_trick_adjoint_projectively_normalII}
        h^{1}((G'-\Lambda)_{C'}) \leq \mathrm{dim}(W')-2 = h^{0}(\Lambda)-3
    \end{equation}
    Since $B$ is ample and $K_{X}$ is nef, $\Lambda = K_{X}+B$ is also ample. Since $\Lambda$ is ample and base point free, $h^{0}(\Lambda) \geq 3.$ \\
    Next, $\mathrm{deg}((G'-\Lambda)_{C'}) = (K_{X}+\Lambda+(m\ell'+\ell-2d-3)B+2(dB-K_{X})+N' \cdot \Lambda)$ 
    $$ = 2g'-2+(m\ell'+\ell-2d-3)(B \cdot \Lambda)+2(dB-K_{X} \cdot \Lambda)+(N' \cdot \Lambda)$$
    Now, $h^{1}((G-\Lambda)_{C'}) = 0$ when $\mathrm{deg}((G'-\Lambda)_{C'}) > 2g'-2$ which happens if $m\ell'+\ell \geq 2d+4$ or if $m\ell'+\ell \geq 2d+3$ when $(dB-K_{X} \cdot \Lambda) > 0$ or if $m\ell'+\ell \geq 2d+2$ when $2(dB-K_{X} \cdot \Lambda)-(B \cdot \Lambda) = ((2d-1)B-2K_{X} \cdot \Lambda) > 0.$ In this case, (\ref{pencil_trick_adjoint_projectively_normalII}) is true because $h^{0}(\Lambda) \geq 3.$ Also, if $m\ell'+\ell = 2d+3$ with $(dB-K_{X} \cdot \Lambda) = 0$ or if $m\ell'+\ell = 2d+2$ when $((2d-1)B-2K_{X} \cdot \Lambda) = 0$ and $N'_{C'} = 0$ in both cases, then $h^{1}((G-\Lambda)_{C'}) = 0.$ On the other hand, if in the last mentioned two cases $N'_{C'}$ is degree zero non-zero divisor on $C',$ then $h^{1}((G-\Lambda)_{C'}) = 1$ and in this latter case we need $h^{0}(\Lambda) \geq 4.$
    \par
    \textbf{ } Now, coming to the case of projective normality, we choose $\ell' = \ell$ and therefore $L' = L$ and show the vanishing for all integers $m \geq 1.$ From the discussion above we obtain the conditions $\ell \geq 3$ or $\ell \geq 2$ when $h^{0}(B) \geq 4$ for the surjectivity of $\beta.$ We the conditions
    $\ell \geq d+2 \geq \ceil[\bigg]{\dfrac{2d+4}{m+1}}$ or $\ell \geq d+2 \geq \ceil[\bigg]{\dfrac{2d+3}{m+1}}$ when $(dB-K_{X} \cdot \Lambda) > 0$ or when $(dB-K_{X} \cdot \Lambda) \geq 0$ with $h^{0}(\Lambda) \geq 4$ or $\ell \geq d+1 \geq \ceil[\bigg]{\dfrac{2d+2}{m+1}}$ if $((2d-1)B-2K_{X} \cdot \Lambda) > 0$ or when $((2d-1)B-2K_{X} \cdot \Lambda) \geq 0$ with $h^{0}(\Lambda) \geq 4.$ Where the second inequality for each lower bound of $\ell$ is true for all integers $m \geq 1.$  This completes the proof.
\end{proof}
\begin{remark}
    Notice that in \cite{BL25}, the authors prove in particular that  for any ample and base point free line bundle $B$ on a smooth projective variety $X$ of dimension $n,$ $\mathcal{L}_{\ell} = K_{X}+\ell B$ is normally generated for $\ell \geq n+1.$ Thus, for any surface $X,$ one has the bound $\ell \geq 3$ for projective normality. Thus, our bound $\ell \geq d+1$ in Theorem does not follow from their Theorem for $d = 1.$ 
\end{remark}
\begin{example}
    The result above says in particular that for any ample and base point free line bundle $B$ on a minimal surface $X$ of Kodaira dimension $\kappa = 0$, $L = K_{X}+\ell B$ is normally generated for all $\ell \geq 2.$
    In particular, for a $K3$ or Abelian surface $L = \ell B$ is normally generated for all $\ell \geq 2.$
\end{example}

\end{document}